\documentclass[a4paper,11pt]{article}

\usepackage{fullpage}
\usepackage{amsmath}
\usepackage{amssymb}
\usepackage{setspace}
\usepackage[pdftex]{graphicx}
\usepackage[usenames, dvipsnames]{xcolor}  
\usepackage{booktabs}
\usepackage{multirow}
\usepackage{authblk}
\usepackage{subcaption}
\usepackage{natbib}  

\usepackage{xspace}    
\usepackage{palatino}
\usepackage{textcomp} 

\usepackage{amssymb,amsmath,verbatim,tabularx,enumitem,colonequals,graphicx,psfrag}

\usepackage{hyperref}                                           
\hypersetup{
	bookmarksnumbered,
	pdfstartview={FitH},
	linkcolor=blue,
	citecolor=blue,
	colorlinks=true,
	pdfpagemode={UseNone},
	plainpages=false
}%
\usepackage[titletoc,title]{appendix}

\usepackage{cleveref}
\usepackage{soul}

\newcommand{\cV}{\mathcal{V}}
\newcommand{\cA}{\mathcal{A}}
\newcommand{\cK}{\mathcal{K}}

\newcommand{\cU}{\mathcal{U}}

\newtheorem{lemma}{Lemma}
\newenvironment{proof}{\par\noindent\textbf{Proof:}}{\hfill\ensuremath{\Box}}

\crefname{lemma}{Lemma}{Lemmas}

\newtheorem{theorem}{Theorem}
\newtheorem{corollary}[theorem]{Corollary}

\usepackage{authblk}

\title{An Efficient Approach to Distributionally Robust Network Capacity Planning}
\author[1]{Trivikram Dokka}
\author[1]{Francis Garuba}
\author[2]{Marc Goerigk}
\author[1]{Peter Jacko}

\affil[1]{Department of Management Science, Lancaster University, United Kingdom}
\affil[2]{Network and Data Science Management, University of Siegen, Germany}

\date{}

\begin{document}

\maketitle

\begin{abstract}
In this paper, we consider a network capacity expansion problem in the context of telecommunication networks, where there is uncertainty associated with the expected traffic demand. We employ a distributionally robust stochastic optimization (DRSO) framework where the ambiguity set of the uncertain demand distribution is constructed using the moments information, the mean and variance. The resulting DRSO problem is formulated as a bilevel optimization problem. We develop an efficient solution algorithm for this problem by characterizing the resulting worst-case two-point distribution, which allows us to reformulate the original problem as a convex optimization problem.

In computational experiments the performance of this approach is compared to that of the robust optimization approach with a discrete uncertainty set. The results show that solutions from the DRSO model outperform the robust optimization approach on highly risk-averse performance metrics, whereas the robust solution is better on the less risk-averse metric.
\end{abstract}


\section{Introduction}
\label{sec:literature}
Uncertainty has been recognized as a reality of our day-to-day living where choices are often made under partial or unknown information. Hence mitigating against uncertainty in decision making has always been a key business driver. In operations research, frameworks have been developed that help to address decision making under uncertainty in two broad areas, namely the stochastic and the robust optimization approaches.

In the stochastic approach, we assign probabilities to the random variables by assuming that the probability distribution of these variables or uncertain data is known or can be accurately estimated from historic data \cite{Chen2008}. A drawback of this approach is that  in real life, the probabilities are often not available or correctly estimated.
Robust optimization on the other hand addresses the problem of data uncertainty by assuming that the data lie within a closed set \cite{Ben_Tal_1999}. It provides an uncertainty immune solution for the worst case of the uncertain data set. Whereas in the robust optimization approach, we optimize the worst-case objective, in the stochastic optimization approach we optimize relevant statistical measures, e.g. expectation, median, CVaR etc.
Ignoring the probability information has been a main criticism of the robust approach which may produce an overly conservative solution \cite{Chen2007}. Despite the limitations of these two approaches, network design problems under demand uncertainty using these approaches have been frequently considered. Network design has a strategic role within the planning function of most organizations. The task is to ensure the highest quality of design while efficiently balancing the requirement of just enough capacity with the capacity investment cost. \cite{Magnanti1984,Minoux1989} provide a survey of the network design models as well as a unifying framework for many of such models. Network design has found application in many areas, such as transportation, supply chain, communications, and social networks.


\cite{Thapalia2012} showed that stochastic just like robust optimization models are often NP-hard, and even the deterministic network design model itself may be difficult to solve for problems of industrial size. \cite{Nemirovski2009} investigated the heuristic methods based on Monte Carlo sampling techniques, stochastic approximation (SA) and sample average approximation (SAA), in their attempt to find a robust stochastic solution. 
On the other hand, \cite{Bai2014} compared the result of the deterministic model to the stochastic model using a standard commercial solver while they proposed the use of heuristics or relaxation methods to solve large-scale problems. 
\cite{Sun2017} determined the quality of the deterministic solutions for a stochastic multi-commodity network design problem and conclude that this solution can be used to find a good heuristic solution to the stochastic multi-commodity network design model. The deterministic solution is hence contained in the stochastic solution and using it as a skeleton improves results with as much as 97\% of the initial loss recovered. The framework consists of solving a deterministic network design problem, extracting the discrete variables, fixing them in the stochastic model and then solving a stochastic linear problem. 

\cite{Santoso2005} solved a supply chain network design problem using sample average approximation and combined this with an accelerated Benders decomposition algorithm to solve a  problem with a large number of scenarios. \cite{Bidhandi2011}  also study a supply chain network, which was solved using SAA combined with a Benders decomposition algorithm. \cite{Kilic2014} formulate a distribution network as a two-stage design problem and solve this using a MIP methodology called the TSMIP.

Literature abounds on the application of robust optimization to network planning and design following the seminal work of \cite{Soyster_1973} and \cite{Ben-Tal1998,Laguna_1998,Kouvelis_1997,El_Ghaoui_1998,Ben_Tal_1999,Bertsimas_2003,Bertsimas_2004} among many others. Of particular importance is the work of \cite{Ben_Tal_2004} which introduced the adjustable robust framework that addresses two-stage decision problem where the network planning and design problem is situated. \cite{Atamturk2007} applied this to network flow design problem under demand uncertainty, a two stage problem which allows the control of the conservatism of the solution via a parameterized budget uncertainty set, \cite{Ord_ez_2007} studied the network capacity problem under both demand and travel time uncertainty for a multi-commodity flow problem with single source and sink per commodity while \cite{Ouorou2007} introduced affine routing concept in their robust capacity planning model under demand uncertainty using path-based formulation and of recent, \cite{Garuba2019b} addressed non-linearity of cost function in robust network capacity expansion problem. Also, there have been work towards exact solution with \cite{Mattia_2012} being the forerunner while others are \cite{Zeng2013,Bertsimas_2015,Pessoa_2015,Ayoub_2016} using different types of decomposition algorithms.

In this paper, to address the drawback of the two approaches, we leverage on their respective strengths  and investigate a distributionally robust stochastic stochastic (DRSO) approach to a multi-commodity network design problem, where the probability distribution itself is affected by ambiguity. This combination of robust and stochastic optimization for a network capacity planning problem is then analyzed with respect to time of execution, resource usage and applicability in a real world setting

This approach has found increasing application in diverse areas/field since its introduction by \cite{Scarf} in his min-max solution of inventory problem. 
\cite{Popescu2007} used this approach in the mean-covariance of the uncertain data distribution to derive a robust solution approach for a max-min and min-max stochastic problem without recourse.  \cite{Delage2010} in their work on distributionally robust stochastic programming develop a model that combines the distribution and moment of the uncertain data. They show that their model outperforms the naive approximated stochastic model proposed by \cite{Popescu2007}. However, a polynomial-time algorithm for the larger range of utility functions considered in \cite{Popescu2007} was beyond the scope of their work. \cite{Goh2010} on the other hand develop a tractable approximation to a distributionally robust optimization problem. Unlike most min-max stochastic programs, the expectation of recourse variable was included in their model.

\cite{Mak2013} applied the distributionally robust framework to solve an electric vehicle battery swapping station location problem. They also validate the accuracy of this solution against that of the SAA and were able to show that it provides a good approximation to the SAA. \cite{Cheng2016} proposed a new reformulation and approximation for solving the distribution robust shortest path problem building on their earlier work \citep{Cheng2013}.
A unifying framework was proposed by \cite{Wiesemann2014} for modeling and solving distributionally robust optimization problems based on standardized ambiguity sets which encompasses many sets from the literature. They identify conditions under which this framework is tractable and develop a tractable conservative approximation for problems that violate these conditions. \cite{Bertsimas2018} developed a framework for solving adaptive distributionally robust linear optimization problem. 


This paper presents the following contributions: We consider an uncertain network capacity expansion problem and develop an efficient algorithm {to its distributionally robust counterpart by characterizing the resulting two point distribution using Richter-Rogonski's theorem \cite{Richter_1957,Rogosinski1958}}. We compare the quality of solutions from this algorithm to the solutions obtained from a discrete robust approach. It is observed from our numerical experiments that solutions found by the DRSO algorithm outperform solutions from the robust optimization approach on highly risk-averse performance metrics.

The model from the literature that is the most similar to ours, but with a flow cost, can be found in \cite{Nakao2017}. While their approach resulted in a computationally challenging solution approach (including the discretization of continuous values), the model presented here is considerably easier to solve.

The rest of this paper is organized as follows. In \autoref{sec:2}, we describe the general distributionally robust stochastic optimization concept. \autoref{sec:3} explains the robust network design problem we consider. In \autoref{sec:Algo_DRSO}, our efficient formulation of the distributionally robust problem is developed. We first focus on the single-commodity case and then extend results to the multi-commodity case. In \autoref{sec:Exp}, the experimental setup and computational result are discussed.  Finally, \autoref{sec:conclude} concludes our work and points out future research directions.

\section{Distributionally Robust Stochastic Optimization}
\label{sec:2}

Distributionally robust stochastic optimization is a data-driven modeling methodology for optimization under uncertainty. 
It encompasses aspects from both robust and stochastic optimization, frameworks that are
complementary to each other though differing in their approaches to addressing the uncertainty \cite{Ben-Tal2009}.

Robust optimization provides a framework to immunize against uncertainty that is believed to lie within a closed and bounded set known as the uncertainty set, while the stochastic optimization framework assumes that the probability distribution of these parameter uncertainty is known. However, in real world application, "true" distribution knowledge is never completely known but at most can only be estimated from available data \cite{Shapiro2002}. On the other hand, one of the major attractiveness which has resulted in the explosion of research into the robust framework is its tractability to a wide range of challenging problems. Nevertheless, this methodology also faced criticism for its inability to factor in the distribution knowledge of the underlying uncertain data, leading to overly conservative solutions \cite{Wiesemann2014,Chen2007}.

Uncertainty can be viewed as \textit{risk} when the probability distribution is known or as an \textit{ambiguity} otherwise \cite{Bertsimas2018}. Neither of these two approaches is suitable to deal with ambiguity from the perspective of decision theory. However, combining the uncertainty set of the robust approach with the probability distribution from the stochastic approach produces a more potent methodology that is able to handle both risk and ambiguity. Hence, under the distributionally robust stochastic framework, the probability distribution is also subjected to uncertainty. The aim is to find a decision such that for any possible probability distribution from the ambiguity set, the stochastic constraints of the model are satisfied.

In the era of growing data-driven applications, moments that constitute the ambiguity are estimated in the face of limited historical data. \cite{Scarf} was the first to apply this methodology in his min-max proposal to the newsvendor problem where the "true" distribution, though not known completely, is only characterized by its mean and standard deviation and belongs to a class of probability distributions with the same mean and standard deviation. The distributionally robust stochastic approach hence seeks to maximize the expectation by considering the worst case distribution in this probability distributions class also known as the ambiguity set.

\section{Problem Description}
\label{sec:3}

Planning for capacities to be added in networks is a major strategic problem in most telecommunication organizations. Usually, this decision is made under uncertainty of the future traffic demand. As with most strategic roles, this often involves large capital expenditure investment. Hence, in this paper, we consider a multi-commodity network demand flow problem, where additional capacities are added to accommodate uncertain traffic demands while minimizing the total cost involved subject to design constraints.

\subsection{Basic Network Expansion Problem}

The problem can be represented by a directed graph $G=(\cV, \cA)$ which denotes the network of interest. Each of the arcs $a \in \cA$ has an original capacity $u_a$ which can be upgraded at a cost $c_a$ per incremental unit of capacity. A set of commodities $\cK=\{1,\ldots,K\}$ need to be routed across the network with each commodity $k \in \cK$ consisting of a demand $d^k \ge 0$, a source node $s^k\in\cV$, and a sink node $t^k\in\cV$. Additionally, let $\phi$ be the cost of not satisfying one unit of demand over the planning horizon (i.e., by outsourcing it). Under complete demand certainty, the nominal network capacity expansion problem can then be formulated as:
\begin{align}
\min\ &\sum_{a\in \cA} c_a x_a + \phi \sum_{k\in\cK} \left[ d^k - \sum_{a\in\delta^-(t^k)} f^k_a + \sum_{a\in\delta^+(t^k)} f^k_a \right]_+ \label{con1a}\\
\text{s.t. } &  \sum_{a\in \delta^-(v)} f^k_a - \sum_{a\in \delta^+(v)} f^k_a \ge
0
& \forall k\in \cK, d \in \cU, v\in \cV\setminus\{s^k,t^k\} \label{con2a} \\
& \sum_{k\in \cK} f^k_a \leq u_a + x_a & \forall   a\in \cA  \label{con3a}\\
& f^k_a \ge 0 & \forall k\in\cK,d\in\cU,a\in\cA \label{con4a} \\
& x_a \geq 0 & \forall a\in\cA \label{con5a}
\end{align}
Here, $[y]_+$ denotes the positive part $\max\{0,y\}$, while $\delta^+(v)$ and $\delta^-(v)$ are the sets of the outgoing and incoming arc at node $v \in \cV$, respectively. Variables $f^k_a$ denote the flow of commodity $k\in\cK$ along edge $a\in\cA$, while $x_a$ models the amount of capacity being added to arc $a$. The objective function~\eqref{con1a} is to minimize the sum of capacity expansion cost and outsourcing costs. Constraints~\eqref{con2a} are a variant of flow constraints, where we allow an arbitrary amount of flow to leave the source node $s^k$. Through the objective, only the flow arriving in $t^k$ is counted. Finally, Constraints~\eqref{con3a} which model the capacity on each edge ensure that amount of flow does not exceed the sum of initial and added capacity.

 \subsection{Robust Problem Formulation}
 \label{sec:prob_form}
 
 Since the actual demand values $\pmb{d}=(d^1,\ldots,d^K)\in\mathbb{R}^K_+$ are uncertain, we assume here that they can take any value in a predetermined uncertainty set $\cU$, which can be represented as  $\cU = \{\pmb{d}^1,\ldots,\pmb{d}^N\}$. The network capacity expansion problem using the robust optimization framework is to minimize the cost of capacity investment and the worst-case costs of outsourced (unsatisfied) demand while satisfying the network constraints.
 The robust network design for uncertainty set $\cU$ is hence:
 \begin{align}
 \min\ &\sum_{a\in \cA} c_a x_a + \max_{\pmb{d}\in\cU} \phi \sum_{k\in\cK} \left[ d^k - \sum_{a\in\delta^-(t^k)} f^k_a(\pmb{d}) + \sum_{a\in\delta^+(t^k)} f^k_a(\pmb{d}) \right]_+ \label{con1}\hspace*{-15mm}\\
 \text{s.t. } &  \sum_{a\in \delta^-(v)} f^k_a(\pmb{d}) - \sum_{a\in \delta^+(v)} f^k_a(\pmb{d}) \ge
 0
 & \forall k\in \cK, \pmb{d} \in \cU, v\in \cV\setminus\{s^k,t^k\} \label{con2} \\
 & \sum_{k\in \cK} f^k_a(\pmb{d}) \leq u_a + x_a & \forall  \pmb{d} \in \cU, a\in \cA  \label{con3}\\
 & f^k_a(\pmb{d}) \ge 0 & \forall k\in\cK,\pmb{d}\in\cU,a\in\cA \label{con4} \\
 & x_a \geq 0 & \forall a\in\cA \label{con5}
 \end{align}
 with a scenario $\pmb{d}$ being the demand vector over all commodities.
 As before, $\phi$ is a penalty parameter for uncovered demand (e.g., outsourcing costs). We can send as much flow as we like, but flow cannot appear outside the source, and sending insufficient flow creates a penalty. The constraints~\eqref{con1a}-\eqref{con5a} have been updated to take uncertainty into account, hence the worst case is considered in constraints~\eqref{con1}.
 The positive part $[\cdot]_+$ in the objective can be easily linearized using additional variables $\tau^k \ge 0$ and constraints $\tau^k \ge d^k - \sum_{a\in\delta^-(t^k)} f^k_a(\pmb{d}) + \sum_{a\in\delta^+(t^k)} f^k_a(\pmb{d})$ for all $k\in\cK$. The inner maximum can be linearized in an analogous way.

\section{Distributionally Robust Stochastic Problem Formulation}
\label{sec:Algo_DRSO}

\subsection{Single-Commodity Case}

The single commodity network design problem with outsourcing can be modeled as below. Notice that we drop the subscript $k$ from $d^k$, $s^k$ and $t^k$, because we assume a single commodity problem, i.e., $K=1$; 
\begin{align}
\min\ &\sum_{a\in \cA} c_a x_a + \phi \max_{\mathbb{P} \in \mathcal{P}} \mathbb{E}_{\mathbb{P}}\big{[} d-\tilde{d}\big{]}_+\\
\text{s.t. } &  \sum_{a\in \delta^-(v)} f_a - \sum_{a\in \delta^+(v)} f_a = 0
& \forall  v\in \cV\backslash \{s,t\} \\
 &  \sum_{a\in \delta^-(t)} f_a - \sum_{a\in \delta^+(t)} f_a = \tilde{d}
& \\
 &  \sum_{a\in \delta^-(s)} f_a - \sum_{a\in \delta^+(s)} f_a = -\tilde{d}
& \\
& f_a \leq u_a + x_a & \forall  a\in \cA \\
& f_a \ge 0 & \forall a\in\cA \\
& \tilde{d} \ge 0 \\
& x_a \geq 0 & \forall a\in\cA 
\end{align}
In this formulation, $\tilde{d}$ is the amount of demand that we intend to satisfy.
We consider the problem as a bilevel optimization problem, where first the network owner makes his decision of the amount of demand he wishes to satisfy, and then nature chooses a probability distribution for the demand which maximizes the expected unsatisfied demand. So the second level problem (nature's problem) is 
\begin{align}
 \max_\mathbb{P} &   \quad \mathbb{E}_{\mathbb{P}}\big{[} d-\tilde{d} \big{]}_+ \label{nature_prob} \\
 \text{s.t. } & \quad \mathbb{E}_{\mathbb{P}}\big{[} d\big{]} = \mu \label{mean_const}\\
 & \quad \mathbb{E}_{\mathbb{P}}\big{[} d - \mu \big{]^2} = \sigma^2 \label{variance_const} 
\end{align}
This means we consider all probability distributions $\mathbb{P}$ over $d$ that have the same mean $\mu$ and variance $\sigma^2$. We denote this set as $\mathcal{P}$.
We can thus rewrite the DRSO problem as
\[ \min_{(\pmb{x},\pmb{f},\tilde{d})\in X} \sum_{a \in \cA} c_ax_a + \phi N(\tilde{d}) \]
where $N(\tilde{d})$ denotes the value of the inner nature's problem, and $X$ the set of feasible solutions with respect to $\pmb{x}$, $\pmb{f}$ and $\tilde{d}$.

\subsection{Model Reformulation}

\begin{lemma}
\label{Nature1}
Let some first-stage solution $(\pmb{x},\pmb{f},\tilde{d})$ be fixed.
The optimal objective value of nature's problem can then be written as
\begin{equation}
N(\tilde{d})=
\begin{cases}
1/2 \left(\mu-\tilde{d}+\sqrt{(\tilde{d}-\mu)^2 +\sigma^2}\right) \quad &\mbox{if} \hspace{0.3pc} \tilde{d} > \frac{\mu^2 + \sigma^2}{2\mu} \label{fln1_3}\\
\mu - \tilde{d}\frac{\mu^2}{\mu^2 + \sigma^2} \quad &\mbox{if} \hspace{0.3pc} \tilde{d} \le \frac{\mu^2 + \sigma^2}{2\mu} 
\end{cases} 
\end{equation}
\end{lemma}
\begin{proof}
A proof of the result can be found in ~\cite{lo1987semi} . Recall that $N(\tilde{d}) = \max_{\mathbb{P} \in \mathcal{P}} \mathbb{E}_{\mathbb{P}}[\max (d-\tilde{d},0)]$. 
It can be shown that there is a worst-case distribution that is a 
two-point distribution, which follows from the Richter-Rogonski theorem 
\cite{Richter_1957,Rogosinski1958}. For the sake of completeness, we present a proof in Appendix \ref{Nat_proof}.
\end{proof}\\
An example for the shape of function $N(\tilde{d})$ is presented in~\autoref{Natur}.

\begin{figure}[htbp]
\begin{center}
\includegraphics[width=.5\textwidth]{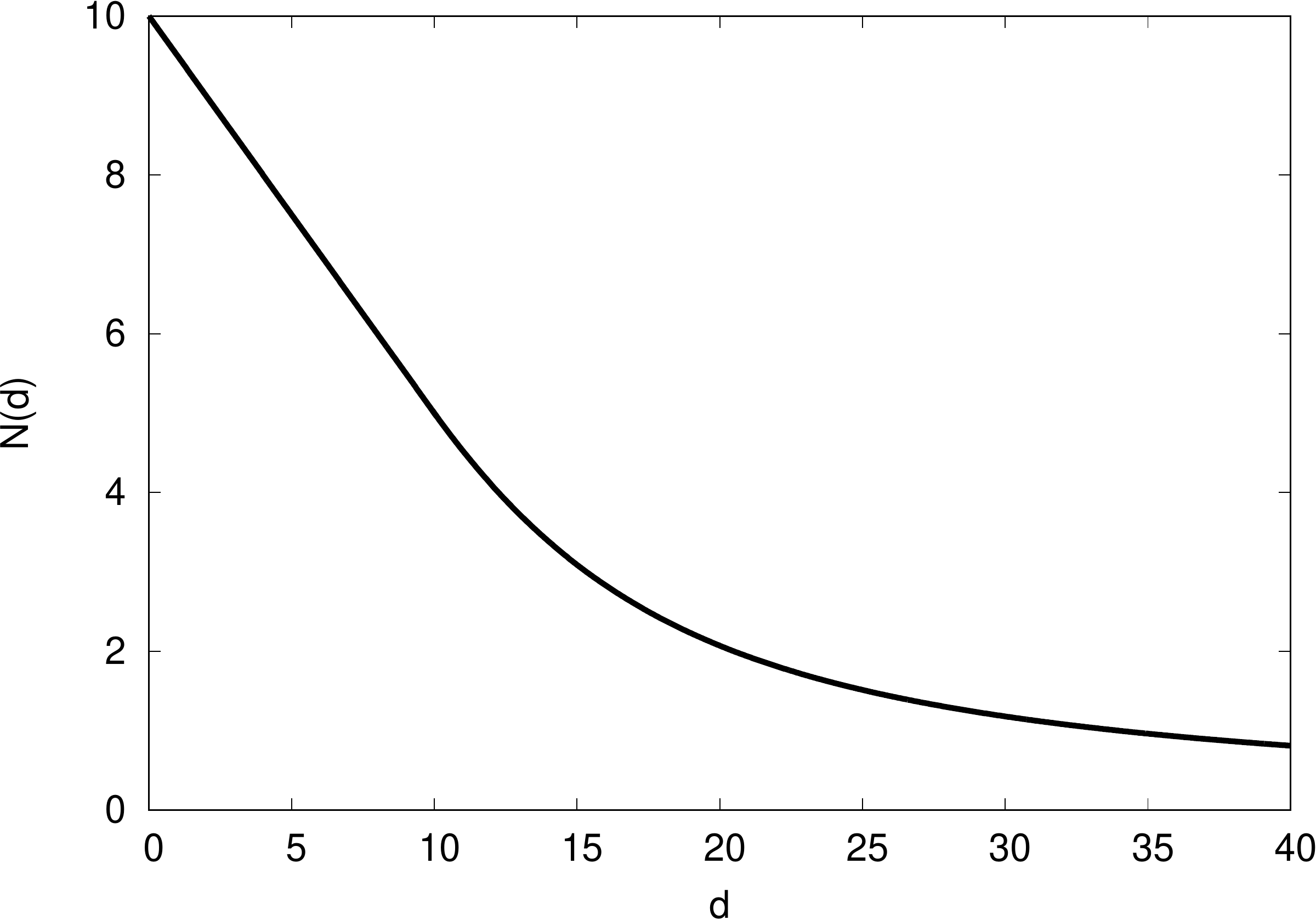}
\caption{Example shape of $N(\tilde{d})$ with $\mu=10$ and $\sigma^2=100$. In this case, $(\mu^2+\sigma^2)/2\mu = 10$.}\label{Natur}
\end{center}
\end{figure}

\begin{lemma}\label{Nlemma}
The function $N(\tilde{d})$ is convex. 
\end{lemma}
\begin{proof}
The first derivative of $N$ with respect to $\tilde{d}$ is
\begin{equation*}
\frac{\partial N}{\partial \tilde{d}} = 
\begin{cases}
\frac{1}{2}\left[ \frac{\tilde{d}-\mu}{\sqrt{(\tilde{d}-\mu)^2 + \sigma^2}}-1\right] & \text{ if }  \tilde{d} > \frac{\mu^2 + \sigma^2}{2\mu} \\
-\frac{\mu^2}{\mu^2 + \sigma^2} & \text{ if } \tilde{d} \le \frac{\mu^2 + \sigma^2}{2\mu} 
\end{cases}
\end{equation*}
and the second derivative is
\begin{equation*}
\frac{\partial^2 N}{\partial \tilde{d}^2} = 
\begin{cases}
\frac{\sigma^2}{2((\tilde{d}-\mu)^2 + \sigma^2)^{\frac{3}{2}}} & \text{ if } \tilde{d} > \frac{\mu^2 + \sigma^2}{2\mu} \\
0 & \text{ if } \tilde{d} \le \frac{\mu^2 + \sigma^2}{2\mu} \end{cases}
\end{equation*}
Note that the first derivative is continuous, as
\begin{align*}
& \frac{1}{2}\left[ \frac{\frac{\mu^2 + \sigma^2}{2\mu} -\mu}{\sqrt{(\frac{\mu^2 + \sigma^2}{2\mu} -\mu)^2 + \sigma^2}}-1\right] =
\frac{1}{2}\left[ \frac{\frac{\sigma^2-\mu^2}{2\mu}}{\sqrt{(\frac{ \mu^2+\sigma^2}{2\mu})^2 -(\mu^2 +\sigma^2)+ \mu^2+ \sigma^2}}-1\right] \\
&= \frac{1}{2}\left[ \frac{\frac{\sigma^2-\mu^2}{2\mu}}{\sqrt{(\frac{\mu^2+\sigma^2}{2\mu})^2}}-1\right]
= \frac{1}{2}\left[ \frac{\frac{\sigma^2-\mu^2}{2\mu}}{\frac{\mu^2+\sigma^2}{2\mu}}-1\right] 
= \frac{1}{2}\left[ \frac{\sigma^2-\mu^2}{\mu^2+\sigma^2}-1\right] 
= \frac{1}{2}\left[ -\frac{2\mu^2}{\mu^2+\sigma^2} \right] = -\frac{\mu^2}{\mu^2+\sigma^2},
\end{align*}
and it is non-decreasing. Hence, the function is convex.
\end{proof}

We now introduce some additional notation. Let $F(\tilde{d})$ be the objective value of the DRSO problem for fixed value of $\tilde{d}$. Then,
\[ F(\tilde{d}) = \min_{(\pmb{x},\pmb{f})\in X(\tilde{d})} \sum_{a\in\cA} c_ax_a + \phi N(\tilde{d}) \]
where as before, $N(\tilde{d})$ is the objective of the adversary problem, and $X(\tilde{d})$ is the set of vectors $\pmb{x}$ and $\pmb{f}$ that give a flow value $\tilde{d}$. We rewrite this to
\[ F(\tilde{d}) = \phi N(\tilde{d}) + G(\tilde{d}) \]
where $G(\tilde{d}) = \min_{(\pmb{x},\pmb{f})\in X(\tilde{d})} \sum_{a\in\cA} c_ax_a$.

\begin{theorem}\label{Fconv}
The function $F(\tilde{d})$ is convex.
\end{theorem}
\begin{proof}
We show that $G(\tilde{d})$ is a convex function in $\tilde{d}$. This, together with the fact that $N(\tilde{d})$ is convex due to Lemma~\ref{Nlemma}, implies that their sum, $F(\tilde{d})$ is also convex.

Consider the function $G'(\tilde{d})$, where
\begin{align}
G'(\tilde{d}) = \min\ &\sum_{a\in \cA} c_a x_a + \psi \cdot \ \left \lvert \tilde{d} - \sum_{a\in \delta^-(t)} f_a + \sum_{a\in \delta^+(t)} f_a  \right \rvert  \\
\text{s.t. } &  \sum_{a\in \delta^-(v)} f_a - \sum_{a\in \delta^+(v)} f_a = 0 \label{Gs}
& \forall  v\in \cV\backslash \{s,t\} \\
& f_a \leq u_a + x_a & \forall   a\in \cA \\
&f_a \ge 0 & \forall a\in\cA \\
& x_a \geq 0 & \forall a\in\cA \label{Ge}
\end{align}
for a large value $\psi \ge \sum_{a\in\cA} c_a$. Let $(\pmb{x}',\pmb{f}')$ be an optimal solution to $G'(\tilde{d})$ and assume that $\Delta := \lvert \tilde{d} - \sum_{a\in \delta^-(t)} f'_a + \sum_{a\in \delta^+(t)} f'_a  \rvert > 0 $. Then we can increase each $x'_a$ by $\Delta$ to find a new solution where there is sufficient capacity to outsource no demand at all. As increasing the capacity this way increases the costs by $\Delta\sum_{a\in\cA} c_a$ and $\psi > \sum_{a\in\cA} c_a$, we have constructed a new solution that is feasible and has no higher objective value that $(\pmb{x}',\pmb{f}')$. Hence there is an optimal solution to $G'(\tilde{d})$ that meet exactly a demand of $\tilde{d}$. Therefore, $G'(\tilde{d})=G(\tilde{d})$.

Recall that  if a function $f_1(x,y)$ is convex in $(x,y)$ and $C$ is a convex set, then
\begin{equation*}
f_2(x)=\inf_{y\in C}f_1(x,y)
\end{equation*} 
is convex as well \cite{Boyd2004}. Therefore, $G(\tilde{d}) = G'(\tilde{d}) = \min_{(\pmb{x},\pmb{f},\tilde{d})\in X} \sum_{a\in \cA} c_a x_a + \psi \cdot \  \lvert \tilde{d} - \sum_{a\in \delta^-(t)} f_a + \sum_{a\in \delta^+(t)} f_a  \rvert $ with $X$ represented by constraints~(\ref{Gs}-\ref{Ge}) is convex.

\end{proof}

We can use \autoref{Fconv} to solve the single-commodity DRSO problem efficiently. For a fixed value $\tilde{d}$, we evaluate $F(\tilde{d})$ by solving $G(\tilde{d})$ as a linear program and $N(\tilde{d})$ using the formula provided in \autoref{fln1_3}.
We can now apply standard convex optimization methods (in our experiments we use the Nelder-Mead method) to solve $\min_{\tilde{d}} F(\tilde{d})$ to optimality.

\subsection{Extension to the Multi-Commodity Case}
\label{sec:ext_MC}

In this setting, we assume that the demand at each of the origin-destination pairs $(s^k,t^k)$ is affected by a different distribution. In particular, we have  mean $\mu_k$ and variance $\sigma^2_k$ for each $k\in\cK$. The nature problem in this case becomes: 
\begin{align}
\max &   \quad \mathbb{E}_{\mathbb{P}}\left[ \sum_{k\in\cK} [d^k - \tilde{d}^k]_+ \right] \label{nat_prob} \\
s.t. & \quad \mathbb{E}_{\mathbb{P}}\big{[} d^k\big{]} = \mu_k &\forall k \in \cK \label{mean_cnsrt}\\
& \quad \mathbb{E}_{\mathbb{P}}\big{[} d^k - \mu_k\big{]}^2 = \sigma_k^2 &\forall k \in \cK \label{vari_c} 
\end{align}
The DRSO problem can now be written as
\begin{align}
\min\ &\sum_{a\in \cA} c_a x_a + \phi \cdot N(\tilde{\pmb{d}}) \\
\text{s.t. } &  \sum_{a\in \delta^-(v)} f^k_a - \sum_{a\in \delta^+(v)} f^k_a = 0
& \forall k\in \cK,  v\in \cV\backslash \{s^k,t^k\} \label{M25}\\
 &  \sum_{a\in \delta^-(t^k)} f^k_a - \sum_{a\in \delta^+(t^k)} f^k_a = \tilde{d}^k
& \forall k\in \cK\\
 &  \sum_{a\in \delta^-(s^k)} f^k_a - \sum_{a\in \delta^+(s^k)} f^k_a = -\tilde{d}^k
& \forall k\in \cK\\
& f^k_a \leq u_a + x_a & \forall  k\in \cK, a\in \cA \\
&f^k_a \ge 0 & \forall k\in \cK, a\in\cA \\
& x_a \geq 0 & \forall a\in\cA \label{M30}
\end{align}
where $N(\tilde{\pmb{d}})$ denotes the multi-dimensional version of nature's problem as defined by equations~\eqref{nat_prob}-\eqref{vari_c}.

\begin{corollary}
The optimal solution for nature's problem defined by equations \eqref{nat_prob}-\eqref{vari_c} can be expressed as
\begin{equation}
N(\tilde{\pmb{d}}) = \sum_{k\in\cK} N(\tilde{d}^k)
\end{equation}
\end{corollary}
\begin{proof}
This extends \cref{Nature1} using the linearity of the expectation on the one hand, and the decomposability of nature's problem on the other hand. We have
\begin{align*}
N(\tilde{\pmb{d}}) &= \max_{\mathbb{P}\in\mathcal{P}} \mathbb{E}_{\mathbb{P}} \left[ \sum_{k\in\cK} [d^k-\tilde{d}^k]_+ \right] \\
&= \max_{\mathbb{P}\in\mathcal{P}} \sum_{k\in\cK} \mathbb{E}_{\mathbb{P}} \left[[d^k-\tilde{d}^k]_+ \right] \\
&=  \sum_{k\in\cK} \max_{\mathbb{P}\in\mathcal{P}} \mathbb{E}_{\mathbb{P}} \left[[d^k-\tilde{d}^k]_+ \right] = \sum_{k\in\cK} N(\tilde{d}^k)
\end{align*}
which proves the result.
\end{proof}

Moreover, it can be easily verified that $N(\pmb{\tilde{d}})$ is convex, as it is a sum of convex functions.

\section{Experiments}
\label{sec:Exp}

\subsection{Setup}

Our models were implemented using a real-world network from the SNDLib library (see~ \cite{Orlowski2010}), the Nobel-US network with 14 nodes and 42 arcs. A sample of $10,000$ demand scenarios was generated to form our reference set, sampling i.i.d from a gamma distribution with shape $4$ and scale $5$ while discarding negative demands and any demand above $50$. Three instances of $20$ source-sink node pairs were randomly generated for the test network; commodity set A, commodity set B and commodity set C. We then sample training sets for the optimization models. Commodities A and B use the same set of 60 scenarios. For commodity C, we sample a separate set of 60 scenarios. These are used as discrete uncertainty sets for the robust model, and to compute the empirical mean and the variance for the DRSO model.

%

The solutions found by the two optimization models are then evaluated with demands from the reference set using 5,000 scenarios. The whole experimental setup is repeated 21 times for commodities A and B, and 20 times for commodity C. We provide an overview on the experimental setup in~\autoref{TabS}.

The cost of capacity allocation to the arc is randomly generated using a normal distribution with mean $40$ and variance $36$, while the penalty of unsatisfied or outsourced demand was set to 130 using $10(N-1)$, where $N$ is the number of arcs.

\begin{table}[htbp]\footnotesize
\centering
\caption{Experimental setup.}
\begin{tabular}{lcccc}
\toprule
\textbf{Experiment}&\textbf{\ Commodity} & \textbf{\# Repetitions} &\textbf{\# Evaluation per Solution} &\textbf{Total \# Evaluation} \\
\toprule 
$\text{DRSO Model}$& A, B& 21 &    5,000&210,000 \\
$\text{Robust Model}$& A, B&21 &   5,000&210,000  \\
$\text{DRSO Model}$& C& 20 &   5,000&100,000 \\
$\text{Robust Model}$& C&20 &   5,000&100,000  \\
\bottomrule
\end{tabular}
\label{TabS}
\end{table}

The results of the two models are recorded as the in-sample results where the first-stage investment cost of the objective value (Cap. Inv) is the cost of deploying capacity and O/S demand is the outsourced (unsatisfied) demand, which when multiplied by the unit penalty cost ($\phi$) gives the second-stage outsourcing costs of the objective value.
The performance of the evaluation model is reported in terms of mean outsourced demand (E[O/S]), expected maximum outsourced demand(E[max O/S]), average outsourced demand over the worst 5\% values (CVaR95), average outsourced demand over the worst 25\% values (CVaR75) and the mean satisfied demand E[$\tilde{d}$]. The first of these metrics, the E[O/S], is a low risk measure while the rest three are high risk measures. The two models were implemented using Julia and Gurobi version 7.5 on a Lenovo desktop machine with 8GB RAM and Intel Core i5-65 CPU with 2.50GHz using Windows 10 (64-bit) OS.

The results of the two models, DRSO and Robust, are evaluated in-sample and out-sample using the below model. This seeks to calculate an optimal flow for a fixed scenario $\pmb{d}$ while minimizing the expected outsourced demand $\tau^k$ in each commodity $k\in\cK$, due to lack of adequate capacity. As the first-stage investment is already fixed, we fix the $\pmb{x}$ solution in this model.
\begin{align*}
\min\ & \sum_{k\in \cK} \tau^k \\
\text{s.t. } &  \sum_{a\in \delta^-(v)} f^k_a - \sum_{a\in \delta^+(v)} f^k_a = 0
& \forall k\in \cK, v\in \cV\backslash \{s^k,t^k\} \\
&  \sum_{a\in \delta^-(t^k)} f^k_a - \sum_{a\in \delta^+(t^k)} f^k_a = \tilde{d}^k
& \forall k\in \cK\\
&  \sum_{a\in \delta^-(s^k)} f^k_a - \sum_{a\in \delta^+(s^k)} f^k_a = -\tilde{d}^k
& \forall k\in \cK\\
& \tau^k \geq d^k - \tilde{d}^k & \forall k\in\cK\\
& f^k_a \leq u_a + x_a & \forall k\in \cK, a\in \cA \\
&f^k_a \ge 0 & \forall k\in \cK, a\in\cA \\
&\tilde{d}^k \geq 0 & \forall k\in \cK \\
&\tau^k \geq 0 & \forall k\in \cK
\end{align*}
For each $\pmb{x}$ solution, $5,000$ evaluations are done and the outsourced demand together with satisfied demand are recorded for each evaluation.

\subsection{Computational Results}

\autoref{Tab4} presents a high level summary of the experimental results. Recall that three instances of $20$ S-T pairs were used, which are denoted as commodities A, B and C in the table. Each row of results under commodities A and B  in the table is the average of $21$ instances while the ones under commodity C is the average of $20$ instances of different $60$ demand samples. The first three columns results are the in-sample optimization result while the next two are the out-of-sample evaluation result. The average (E[O/S]), the maximum (E[Max O/S]), average of the largest $5\%$ (CVaR95) and average of largest $25\%$ (CVaR75) are calculated from the $5,000$ evaluation results for the outsourced demand as the out-of-sample results. While for the satisfied demand, only the average (E[$\tilde{d}$]) is calculated.

\begin{table}[htbp]\footnotesize
\centering
\caption{Comparing the two models under two commodities.}
\begin{tabular}{lrrrrrrr}
\toprule

\textbf{Commodity A}	& 	 Tot. Inv 	& Cap. Inv&O/S Dem&E[O/S]&E[$\tilde{d}$]	& Cap Add&  Unit Cost \\
\toprule 
Robust &44,679.05 &	 32,641.54 &92.60&70.91	& 372.02 &	 842.09 	& 38.76	\\ 
DRSO&	 41,830.77 &	18.671.69 &178.15&	159.75& 267.28 & 484.99 &	 38.50 	\\

\midrule
\textbf{Commodity B}&&&&&&&\\	 
\midrule									
Robust&	 43,548.63 	& 28,685.06&114.33& 86.22&355.32&	 730.25&	 39.29	\\
DRSO&	41,081.28 	& 19,255.14 &167.89& 150.49&280.37 &  489.35 &	 39.35	\\
\midrule
\textbf{Commodity C}&&&&&&&\\	 
\midrule									
Robust &50,878.38 &	 29,226.48 &166.55&	132.37 &281.68 &705.39 	& 41.43	\\ 
DRSO&	 47,290.23 &	 11,734.21 &273.51&261.79 &141.53 & 283.55 &	 41.38 	\\
\bottomrule
\end{tabular}
\label{Tab4}
\end{table}

In the following, we focus on the evaluation for commodity type A. Results for commodity types B and C can be found in Appendix~\ref{AppendA} and Appendix~\ref{AppendB}, respectively.

From \autoref{Tab4}, it is observed that DRSO solutions build less capacity compared to the robust solutions for all demand instances and hence a lower capital investment, both in terms of total investment and capacity investment. The capacity investment is the cost of adding capacity, the first term of the objective function, while the total investment is the objective value of the optimization problem. The observation seems valid irrespective of the commodity and data set used. For commodity A, the robust solution builds approximately $74\%$ more capacity than the DRSO solution ($50\%$ and $150\%$ more in case of commodity B and C, respectively).
Though this result is the average over all demand instances, it is also true for each single demand instance, see \autoref{fig3d}.

\begin{table}[htbp]\footnotesize
	\centering
	\caption{Robust model results for commodity type A.}
	\begin{tabular}{rrrrrrrrr}
\toprule 
&\multicolumn{2}{c}{In Sample} 
&\multicolumn{5}{c}{Out of Sample}&\\	
\cmidrule (r) {2-3} \cmidrule (l) {4-8}
Inst. &    Cap. Inv.&	O/S Demand&	E[O/S]& E[Max O/S]&	CVaR95&	CVaR75&	E[$\tilde{d}$]&	CapAdd \\
\midrule
 1	 & 	 32,589.64 	 & 	 97.48 	 & 	 72.91 	 & 	 562.76 	 & 	 407.99 	 & 	 231.40 	 & 	 378.76 	 & 	 832.84 	 \\ 
 2	 & 	 34,600.41 	 & 	 77.58 	 & 	 59.55 	 & 	 525.62 	 & 	 369.93 	 & 	 199.51 	 & 	 385.59 	 & 	 894.49 	 \\ 
 3	 & 	 33,855.56 	 & 	 78.89 	 & 	 65.59 	 & 	 558.41 	 & 	 402.72 	 & 	 219.72 	 & 	 375.00 	 & 	 873.30 	 \\ 
 4	 & 	 34,166.13 	 & 	 75.71 	 & 	 60.50 	 & 	 525.91 	 & 	 378.00 	 & 	 204.94 	 & 	 376.94 	 & 	 885.85 	 \\ 
 5	 & 	 34,445.25 	 & 	 88.87 	 & 	 64.67 	 & 	 545.86 	 & 	 391.37 	 & 	 215.37 	 & 	 370.99 	 & 	 888.71 	 \\ 
 6	 & 	 33,818.93 	 & 	 90.58 	 & 	 65.47 	 & 	 542.46 	 & 	 387.26 	 & 	 215.31 	 & 	 387.39 	 & 	 870.78 	 \\ 
 7	 & 	 32,527.15 	 & 	 95.14 	 & 	 75.24 	 & 	 561.43 	 & 	 411.77 	 & 	 239.97 	 & 	 364.20 	 & 	 838.04 	 \\ 
 8	 & 	 31,603.34 	 & 	 106.24 	 & 	 74.96 	 & 	 561.52 	 & 	 407.47 	 & 	 232.43 	 & 	 372.81 	 & 	 809.27 	 \\ 
 9	 & 	 32,173.50 	 & 	 103.65 	 & 	 77.82 	 & 	 561.57 	 & 	 405.88 	 & 	 231.07 	 & 	 370.11 	 & 	 834.15 	 \\ 
 10	 & 	 29,153.00 	 & 	 99.19 	 & 	 91.97 	 & 	 585.70 	 & 	 437.79 	 & 	 264.62 	 & 	 355.20 	 & 	 751.48 	 \\ 
 11	 & 	 29,642.85 	 & 	 118.30 	 & 	 82.57 	 & 	 578.90 	 & 	 423.41 	 & 	 250.96 	 & 	 350.37 	 & 	 768.80 	 \\ 
 12	 & 	 29,848.79 	 & 	 124.09 	 & 	 86.77 	 & 	 579.68 	 & 	 423.98 	 & 	 245.93 	 & 	 354.20 	 & 	 773.85 	 \\ 
 13	 & 	 31,146.65 	 & 	 106.54 	 & 	 75.19 	 & 	 561.82 	 & 	 410.03 	 & 	 239.60 	 & 	 363.83 	 & 	 805.23 	 \\ 
 14	 & 	 32,925.94 	 & 	 80.73 	 & 	 67.86 	 & 	 551.43 	 & 	 396.82 	 & 	 222.27 	 & 	 363.21 	 & 	 853.06 	 \\ 
 15	 & 	 33,431.34 	 & 	 90.17 	 & 	 67.58 	 & 	 555.29 	 & 	 403.51 	 & 	 228.90 	 & 	 374.53 	 & 	 858.36 	 \\ 
 16	 & 	 33,160.00 	 & 	 90.34 	 & 	 68.89 	 & 	 544.42 	 & 	 390.52 	 & 	 216.68 	 & 	 369.76 	 & 	 854.61 	 \\ 
 17	 & 	 30,910.42 	 & 	 93.55 	 & 	 75.41 	 & 	 579.02 	 & 	 424.54 	 & 	 246.26 	 & 	 352.29 	 & 	 799.72 	 \\ 
 18	 & 	 35,969.65 	 & 	 74.46 	 & 	 52.75 	 & 	 524.08 	 & 	 368.39 	 & 	 190.02 	 & 	 395.79 	 & 	 925.42 	 \\ 
 19	 & 	 33,415.01 	 & 	 84.78 	 & 	 67.08 	 & 	 540.52 	 & 	 387.26 	 & 	 212.55 	 & 	 384.07 	 & 	 860.78 	 \\ 
 20	 & 	 30,853.02 	 & 	 93.20 	 & 	 78.69 	 & 	 574.63 	 & 	 425.13 	 & 	 248.80 	 & 	 358.67 	 & 	 795.68 	 \\ 
 21	 & 	 35,235.78 	 & 	 75.03 	 & 	 57.68 	 & 	 515.23 	 & 	 363.10 	 & 	 200.72 	 & 	 408.72 	 & 	 909.52 	 \\ 
	\bottomrule
\end{tabular}
\label{Tab2}
\end{table}

\begin{table}[htbp]\footnotesize
\centering
\caption{DRSO model results for commodity type A.}
\begin{tabular}{rrrcrcrrrr}
\toprule 
&&\multicolumn{2}{l}{In Sample} 
&\multicolumn{5}{c}{Out of Sample}&\\ 
\cmidrule (r) {2-4} \cmidrule (l) {5-9}
Inst. &    Cap. Inv.& Nature& $\tilde{d}$ &E[O/S]& E[ Max O/S]& CVaR95& CVaR75& E[$\tilde{d}$]& CapAdd \\
\midrule
1	&	 19,066.17 	 & 	 176.23 	 & 	 300.70 	 & 	 154.02 	 & 	 698.87 	 & 	 543.18 	 & 	 356.25 	 & 	 278.11 	 & 	 495.16 	\\
2	&	 19,577.97 	 & 	 171.08 	 & 	 297.69 	 & 	 150.88 	 & 	 688.22 	 & 	 532.53 	 & 	 350.41 	 & 	 279.93 	 & 	 506.00 	\\
3	&	 16,785.07 	 & 	 185.68 	 & 	 267.71 	 & 	 173.22 	 & 	 726.25 	 & 	 570.56 	 & 	 383.89 	 & 	 254.52 	 & 	 437.08 	\\
4	&	 20,276.51 	 & 	 162.84 	 & 	 307.05 	 & 	 142.09 	 & 	 677.09 	 & 	 521.40 	 & 	 339.98 	 & 	 286.83 	 & 	 529.02 	\\
5	&	 17,372.72 	 & 	 192.30 	 & 	 271.26 	 & 	 176.06 	 & 	 728.31 	 & 	 572.62 	 & 	 385.63 	 & 	 246.74 	 & 	 452.77 	\\
6	&	 17,882.24 	 & 	 183.66 	 & 	 283.42 	 & 	 168.07 	 & 	 716.15 	 & 	 560.46 	 & 	 373.48 	 & 	 261.94 	 & 	 463.40 	\\
7	&	 19,125.03 	 & 	 170.55 	 & 	 292.95 	 & 	 154.98 	 & 	 704.27 	 & 	 548.58 	 & 	 361.60 	 & 	 271.78 	 & 	 495.25 	\\
8	&	 17,633.05 	 & 	 197.36 	 & 	 275.90 	 & 	 169.89 	 & 	 717.59 	 & 	 561.90 	 & 	 376.29 	 & 	 258.67 	 & 	 459.89 	\\
9	&	 17,282.82 	 & 	 195.63 	 & 	 261.69 	 & 	 180.33 	 & 	 737.88 	 & 	 582.19 	 & 	 395.31 	 & 	 237.93 	 & 	 442.80 	\\
10	&	 19,688.38 	 & 	 161.49 	 & 	 292.87 	 & 	 151.62 	 & 	 700.52 	 & 	 544.83 	 & 	 359.57 	 & 	 274.50 	 & 	 507.90 	\\
11	&	 19,001.74 	 & 	 172.63 	 & 	 298.09 	 & 	 152.80 	 & 	 697.06 	 & 	 541.37 	 & 	 354.93 	 & 	 274.36 	 & 	 495.37 	\\
12	&	 21,012.66 	 & 	 162.36 	 & 	 317.59 	 & 	 136.45 	 & 	 670.18 	 & 	 514.49 	 & 	 332.53 	 & 	 292.79 	 & 	 547.64 	\\
13	&	 16,316.88 	 & 	 202.74 	 & 	 264.19 	 & 	 184.32 	 & 	 735.38 	 & 	 579.69 	 & 	 392.82 	 & 	 246.14 	 & 	 424.59 	\\
14	&	 18,998.62 	 & 	 172.22 	 & 	 291.80 	 & 	 152.49 	 & 	 704.37 	 & 	 548.68 	 & 	 361.77 	 & 	 268.54 	 & 	 497.03 	\\
15	&	 20,174.18 	 & 	 166.03 	 & 	 304.22 	 & 	 144.64 	 & 	 695.35 	 & 	 539.66 	 & 	 352.67 	 & 	 282.83 	 & 	 523.39 	\\
16	&	 22,007.06 	 & 	 144.70 	 & 	 318.79 	 & 	 132.10 	 & 	 680.78 	 & 	 525.08 	 & 	 338.10 	 & 	 298.47 	 & 	 569.06 	\\
17	&	 19,207.57 	 & 	 172.70 	 & 	 296.34 	 & 	 151.61 	 & 	 700.80 	 & 	 545.11 	 & 	 358.26 	 & 	 269.02 	 & 	 503.15 	\\
18	&	 19,184.25 	 & 	 176.21 	 & 	 292.80 	 & 	 156.95 	 & 	 706.77 	 & 	 551.08 	 & 	 364.10 	 & 	 261.23 	 & 	 500.53 	\\
19	&	 19,023.42 	 & 	 173.25 	 & 	 296.49 	 & 	 153.29 	 & 	 688.88 	 & 	 533.19 	 & 	 352.03 	 & 	 276.72 	 & 	 493.69 	\\
20	&	 15,879.57 	 & 	 198.35 	 & 	 257.87 	 & 	 185.17 	 & 	 741.70 	 & 	 586.01 	 & 	 399.03 	 & 	 243.59 	 & 	 412.61 	\\
21	&	 16,609.68 	 & 	 203.07 	 & 	 260.61 	 & 	 183.82 	 & 	 738.96 	 & 	 583.26 	 & 	 396.28 	 & 	 248.14 	 & 	 428.43 	\\
\bottomrule
\end{tabular}
\label{Tab3}
\end{table}

\autoref{Tab2} and \autoref{Tab3} present the results for each demand instance for the robust and DRSO solutions respectively. For the same demand instance, the DRSO solution invests less in capacity, which can be expected since it takes the distribution information of the random variable into consideration unlike the robust model. The DRSO solution is less conservative in this regards, whereas the robust solution plans for the worst observed realization of the random variable. The DRSO solution is therefore cost efficient, building only the needed capacity based on the distribution information which lowers the investment cost, while the robust solution seems to take the more pessimistic route. Comparing each demand instance, the penalty cost of outsourced demand is lower for the robust solution, which is attributed to the fact that it builds for the worst case.

Evaluating the performance of the optimization solution, which is reported by the out-sample performance in \autoref{Tab2} and \autoref{Tab3}, the expected unsatisfied demand and expected satisfied demand are lower compared to the DRSO solution. All other high-risk metrics are higher for the DRSO solution. The first four metrics are derived from the expected unsatisfied demand which explains this observation while the satisfied demand on the other hand is a function of the capacity already installed for which the DRSO solution is less conservative. However, this simple comparison may not fully explain the performance of the models, hence we will rely on the charts presented in \autoref{fig1} to \autoref{fig3}. The capacity investment, on the horizontal axis, is compared with the four expected unsatisfied demand metrics in \autoref{fig1a} to \autoref{fig1d}.

\begin{figure}[htb]
	\begin{center}
		\begin{subfigure}{.45\textwidth}
			\includegraphics[width=1.1\linewidth]{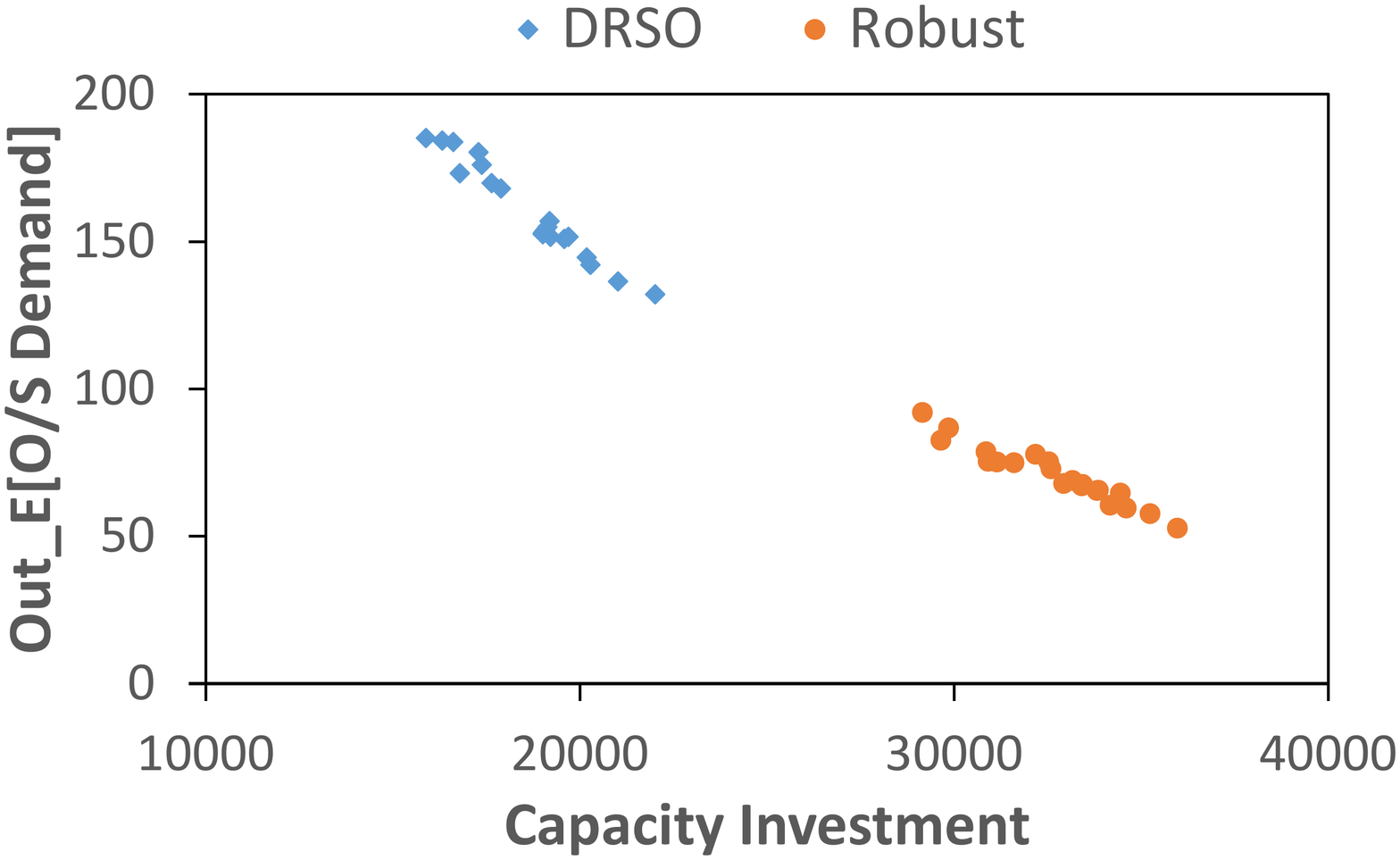}
			\caption{Expected unsatisfied demand.}\label{fig1a}
		\end{subfigure}
		\hspace{1cm}
		\begin{subfigure}{.45\textwidth}
			\includegraphics[width=1.1\linewidth]{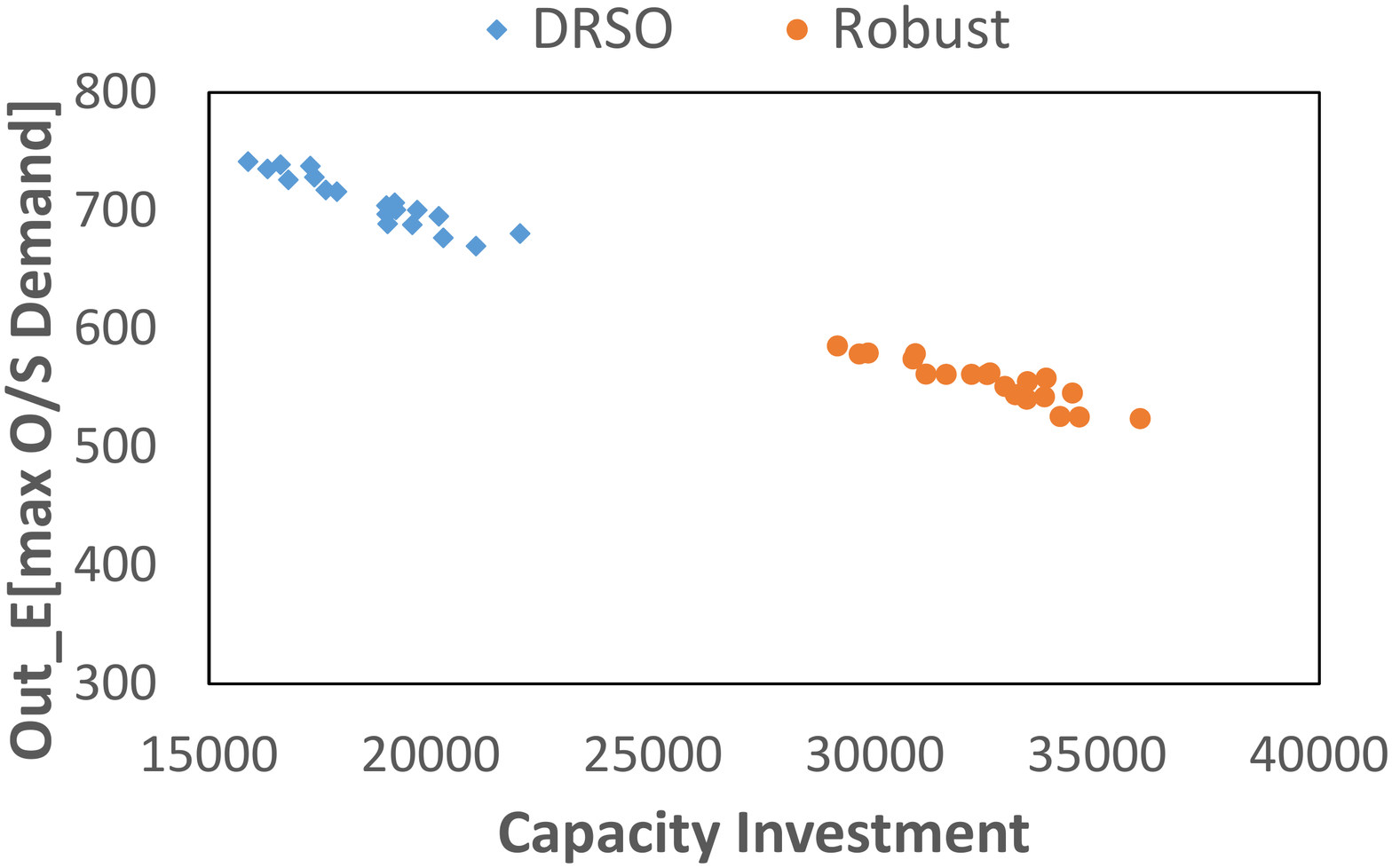}
			\caption{Expected maximum unsatisfied demand.}\label{fig1b}
		\end{subfigure}
		\vspace*{5mm}
		\begin{subfigure}{.45\textwidth}
			\includegraphics[width=1.1\linewidth]{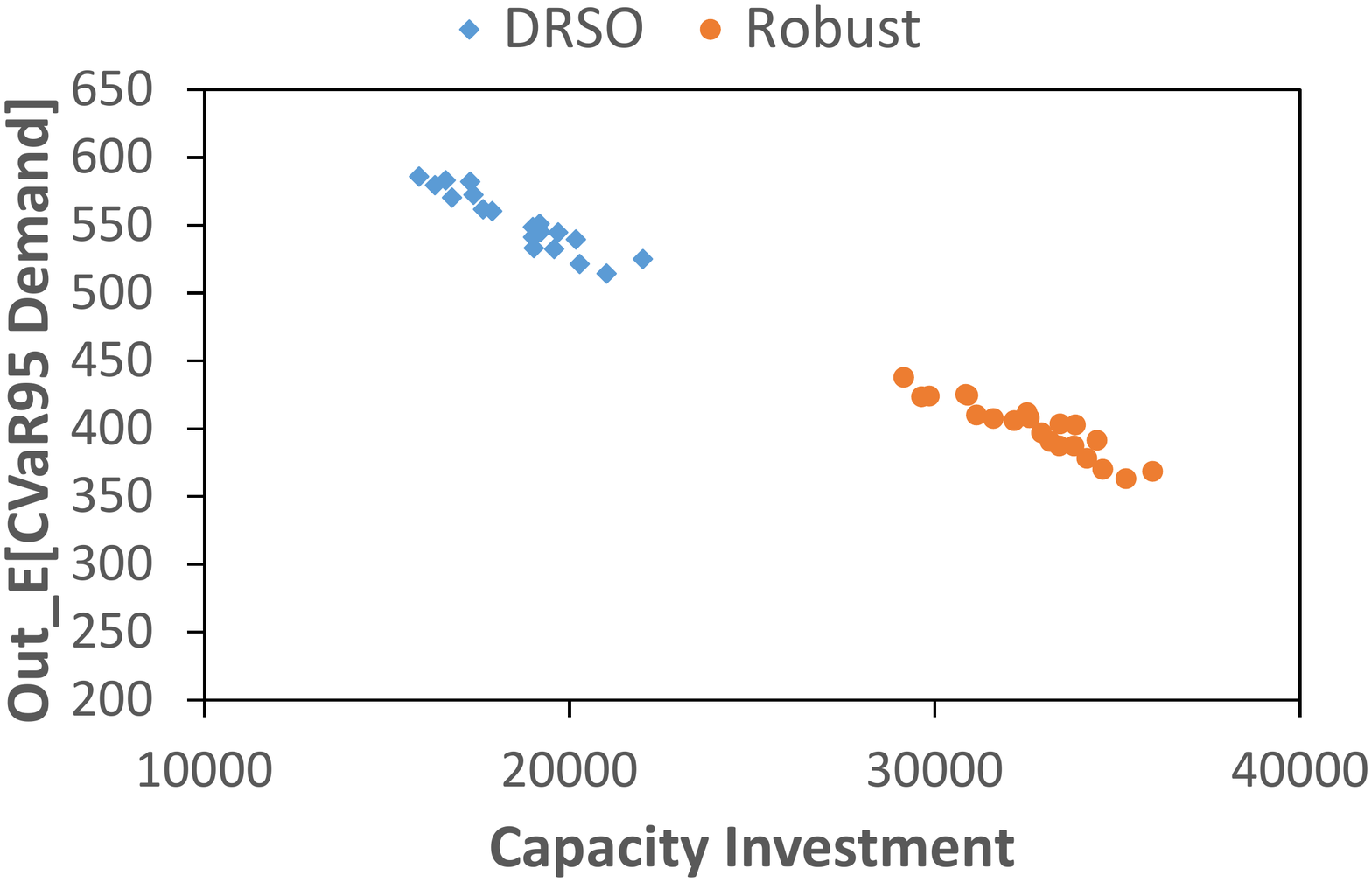}
			\caption{CVaR95 unsatisfied demand.}\label{fig1c}
		\end{subfigure}
		\hspace{1cm}
		\begin{subfigure}{.45\textwidth}
			\includegraphics[width=1.1\linewidth]{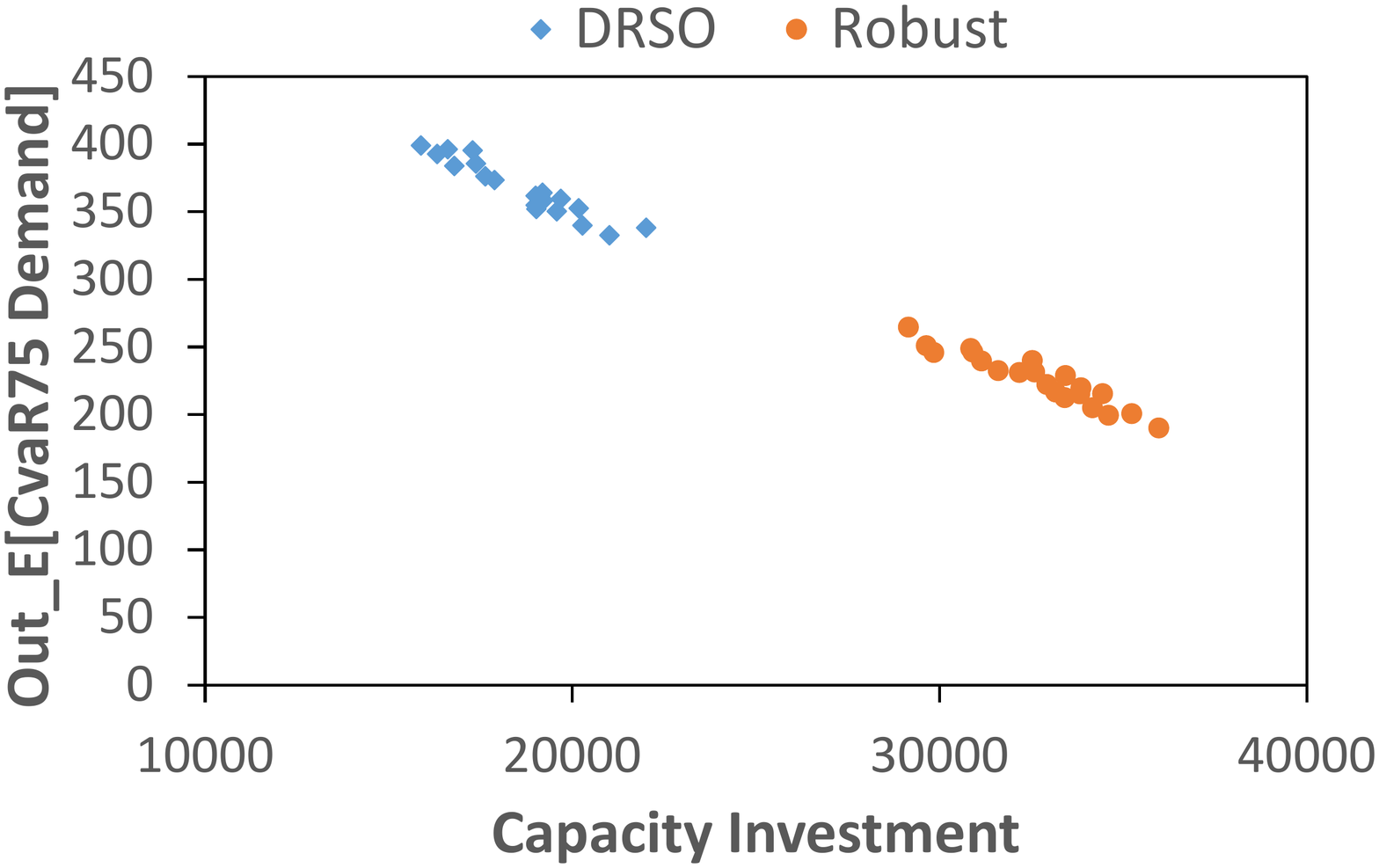}
			\caption{CVaR75 unsatisfied demand.}\label{fig1d}
		\end{subfigure}
		\caption{Expected unsatisfied demand mean and risk measures (commodity A).}\label{fig1}
	\end{center}
\end{figure}

Each of these four charts gives the same observation: Solutions based on the robust model,at higher investment cost, are revealed to be generally more robust than solutions based on DRSO model, which are at lower investment region. For an unexpected surge in traffic, the solution based on robust model will be more able to accommodate this surge compared to the DRSO solution, as its expected outsourced demand is lower. However, in order to allow for a fair comparison of these two models  and to be sure this observation is consistent for both models in all  the investment regions, we scale a robust solution up towards the DRSO solutions area while scaling down a DRSO solution towards the robust solutions area using expected unsatisfied demand that is presented in \autoref{fig1a}. This scaling, which was also used in \cite{Garuba2019a}, is carried out using a representative data point for each model. We scale the $\pmb{x}$ solution in the direction of interest and re-evaluate with the out-of-sample set. In personal experience, this is also common in the industry, where an optimal solution is scaled up or down during planning iteration, thus allowing for a comprehensive comparison of the two models in all the investment regions. The DRSO $\pmb{x}$ solution is multiplied by a factor of $\lambda=1.0$ to $\lambda=1.8$, where $\lambda $ is the scaling factor, with a scale interval of $0.08$, to scale up the solution towards the high investment area while the robust solution is multiplied by a factor of $\lambda=1.0$ to $\lambda=0.5$, with a scale interval of $0.05$, to scale down the solution towards the low investment area. The result of this scaling is as shown in \autoref{fig5} which to the contrary shows that for highly risk-averse metrics (maximum and CVaRs of expected unsatisfied demand), solutions based on the DRSO model are in fact better, having a higher degree of robustness with increasing capacity investment even for high investment. However, for the less risk-averse metric (expected unsatisfied demand), the robust model gives a better solution at higher investment region but with comparable performance for lower investment cost. This observation is consistent for the other two commodities B and C, see \autoref{fig5B} and \autoref{fig9}  in the appendix.

\begin{figure}[htb]
\begin{center}
\begin{subfigure}{.45\textwidth}
\includegraphics[width=1.1\linewidth]{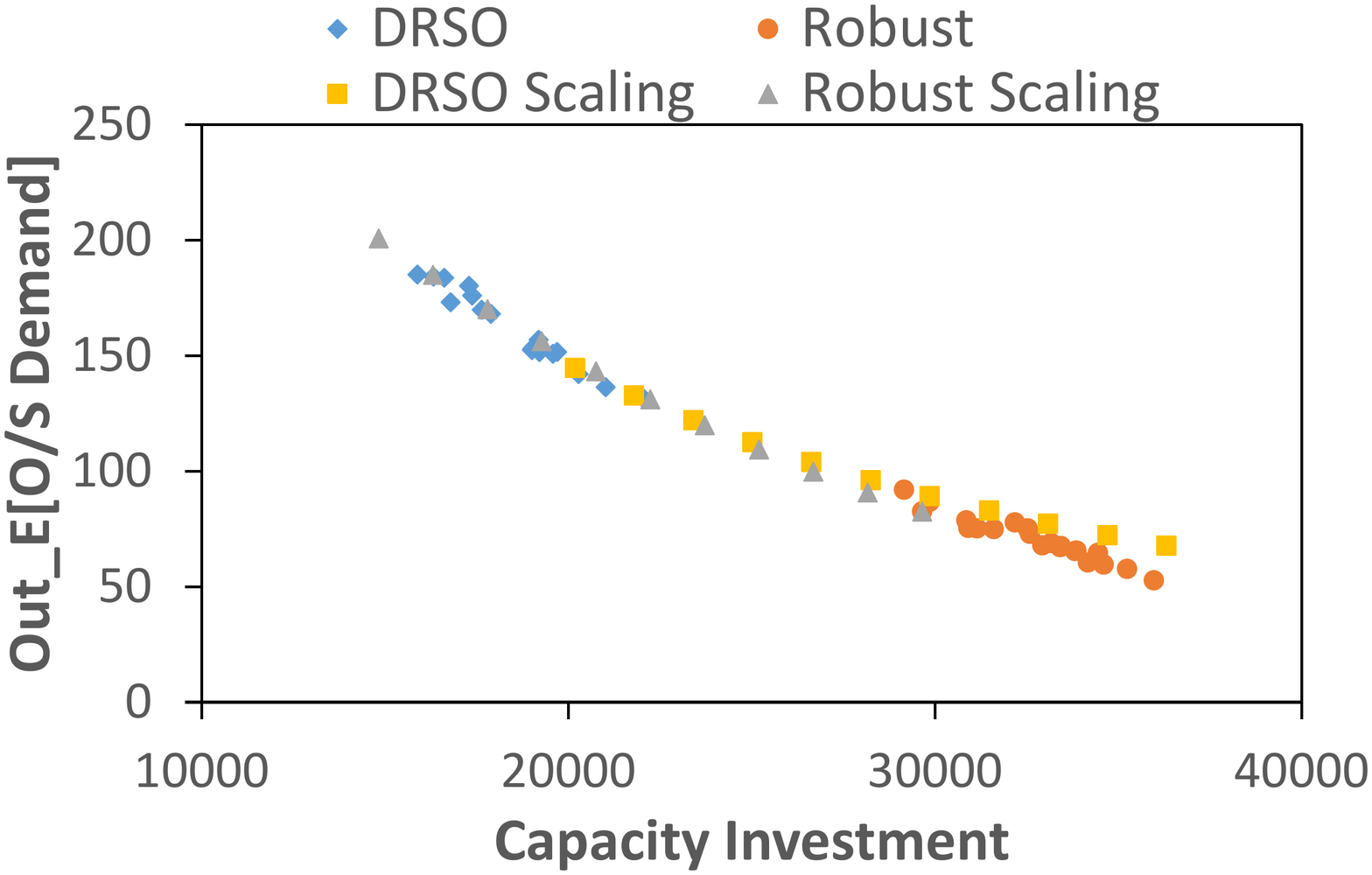}
\caption{Expected unsatisfied demand.}\label{fig5a}

\end{subfigure}
\hspace{1cm}
\begin{subfigure}{.45\textwidth}
\includegraphics[width=1.1\linewidth]{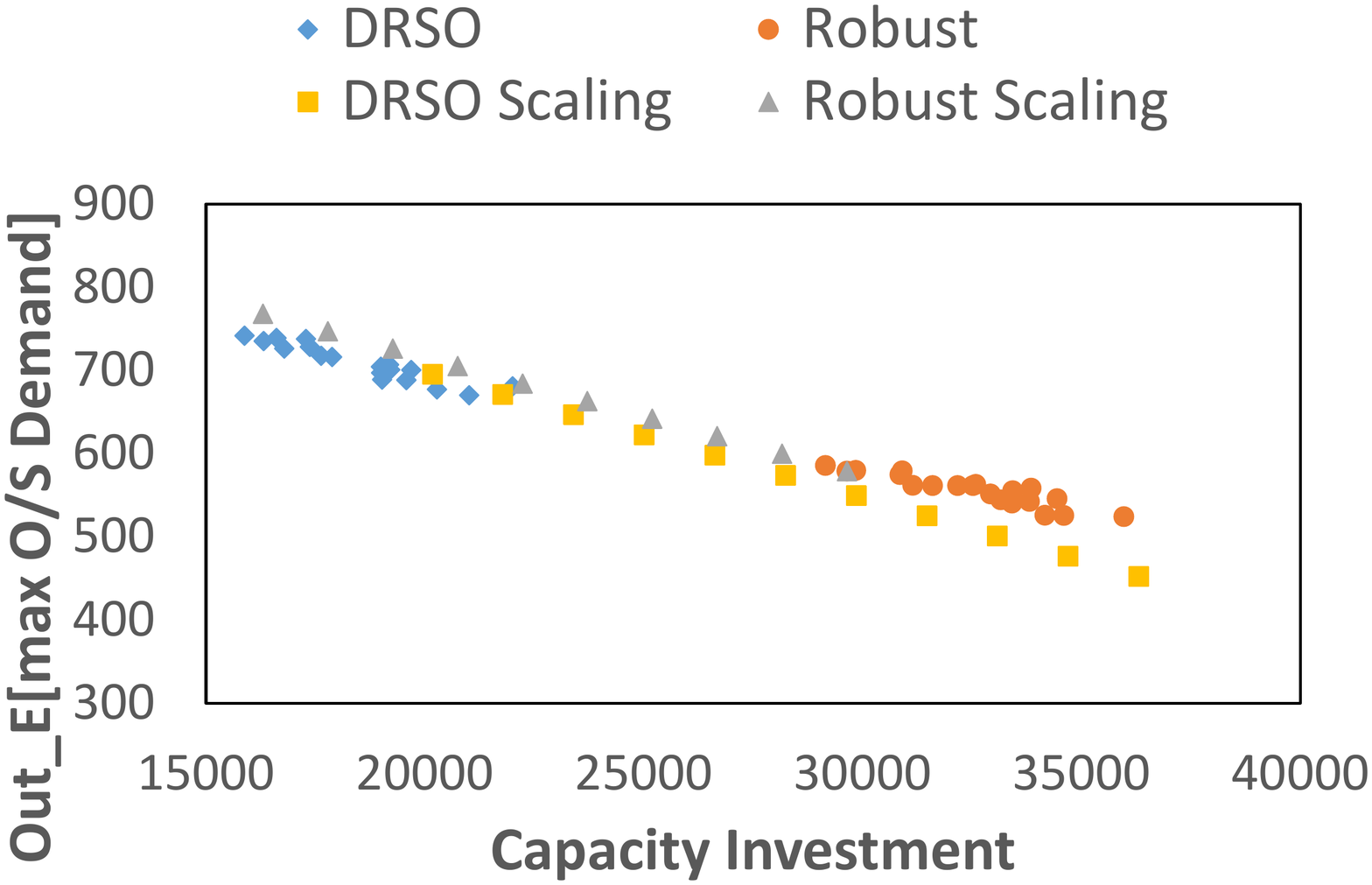}
\caption{Expected maximum unsatisfied demand.}\label{fig5b}
\end{subfigure}
\vspace*{5mm}
\begin{subfigure}{.45\textwidth}
\includegraphics[width=1.1\linewidth]{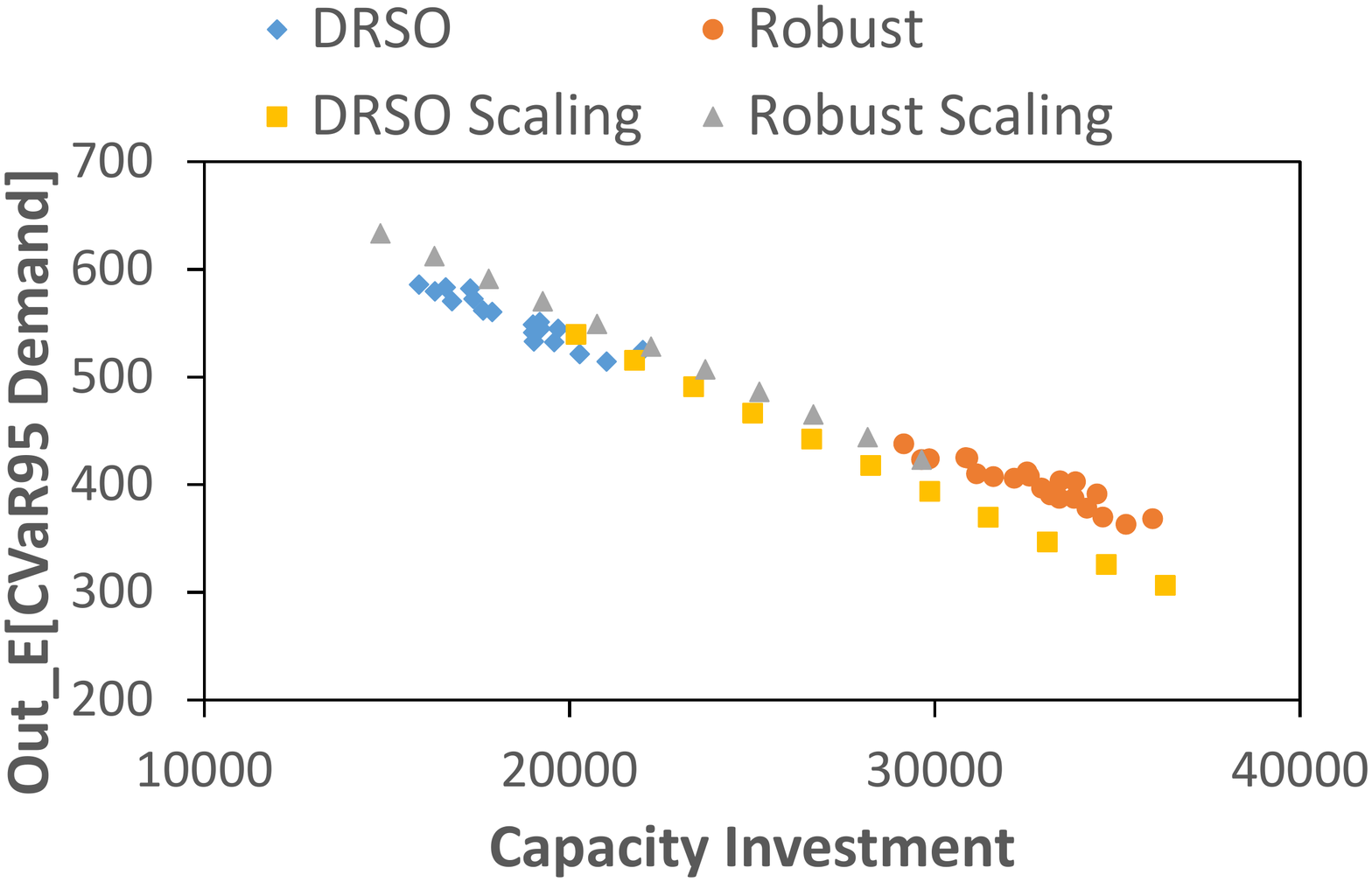}
\caption{CVaR95 unsatisfied demand}\label{fig5c}
\end{subfigure}
\hspace{1cm}
\begin{subfigure}{.45\textwidth}
\includegraphics[width=1.1\linewidth]{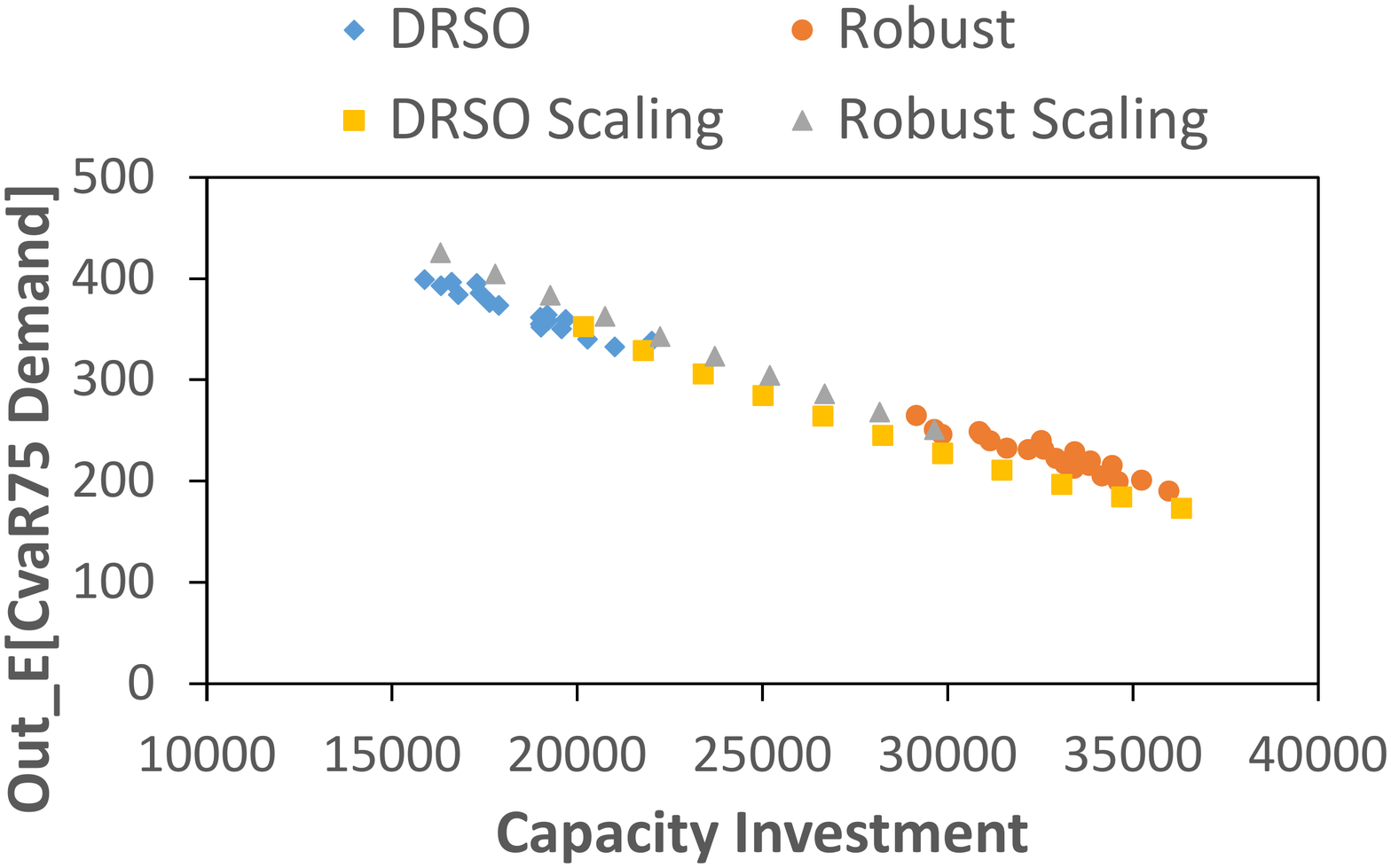}
\caption{CVaR75 unsatisfied demand.}\label{fig5d}
\end{subfigure}
\caption{Performance metric scaling (commodity A).}\label{fig5}
\end{center}
\end{figure}

Next we compare the unsatisfied demand (in-sample) to the expected outsourced demand (out-of-sample) for these two models to see which of these gives a better estimate. The charts in \autoref{fig2a} to \autoref{fig2e} present these results and they show that the DRSO results produce a far better estimate and hence a better predictor of the input variables under data uncertainty. On the average the input/output ratio of the unsatisfied demand (in-sample) to the expected unsatisfied demand (out-of-sample) is around $10.3\%$ (commodities B and C respectively are $10.3\%$ and $4.29\%$) for the DRSO model, while for the robust model, this is as high as $23.43\%$ (commodities B and C respectively are $24.59\%$ and $20.52\%$). Also, regression analysis in \autoref{fig2a} shows that $93\%$ variation in the estimate is explained by the in-sample result for the DRSO model while for the robust model this is approximately $71\%$, see \autoref{fig2b}. The result follows the same pattern for the maximum expected unsatisfied demand with respect to the unsatisfied demand in \autoref{fig2c} for DRSO with $R^2=0.8121$ and \autoref{fig2d} for the robust model with $R^2=0.6048$.

\begin{figure}[!htb]
\begin{center}
\begin{subfigure}{.45\textwidth}
\includegraphics[width=1.1\linewidth]{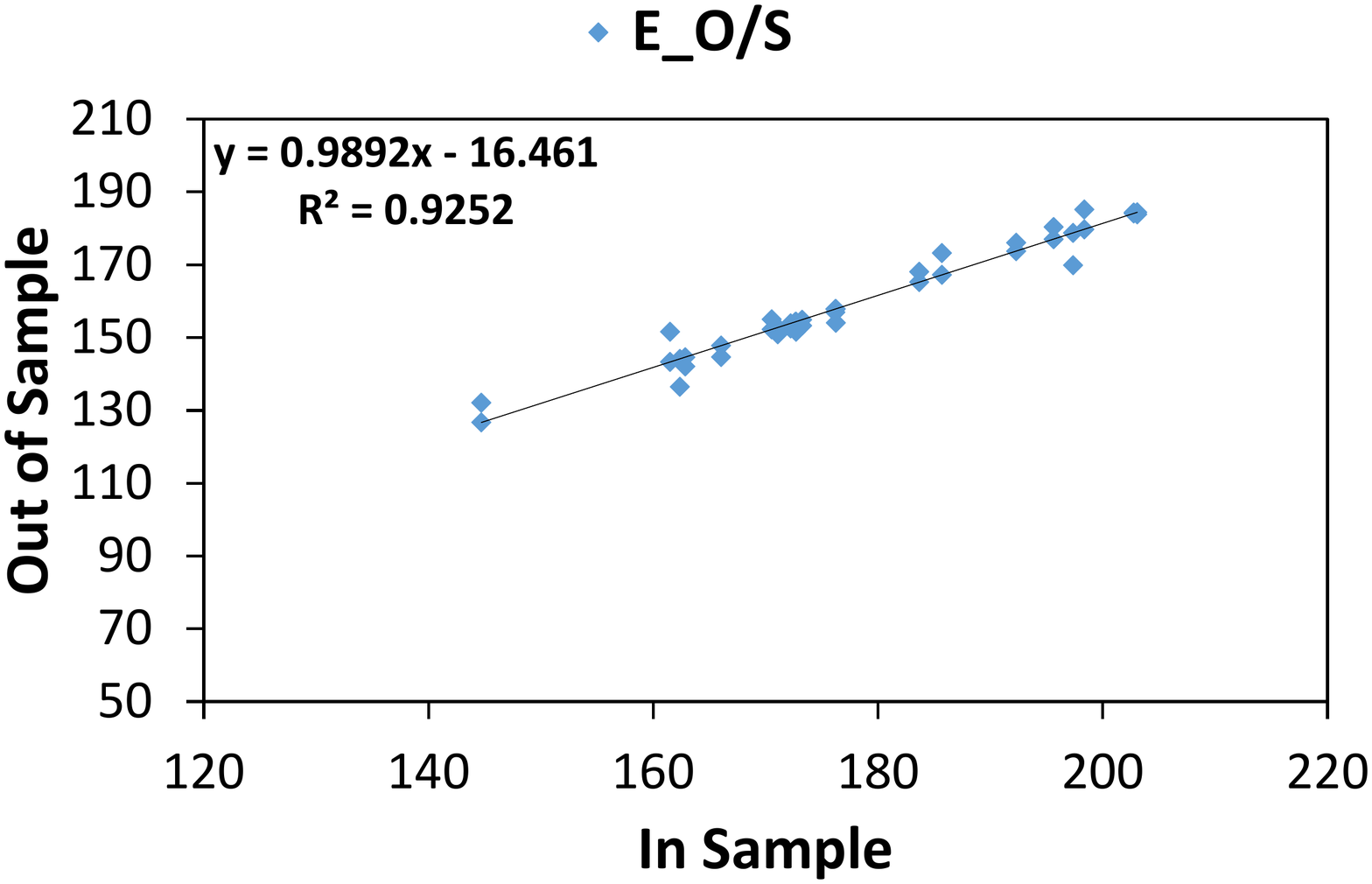}
\caption{DRSO solutions, expected demand gap.}\label{fig2a}
\end{subfigure}
\hspace{1cm}
\begin{subfigure}{.45\textwidth}
\includegraphics[width=1.1\linewidth]{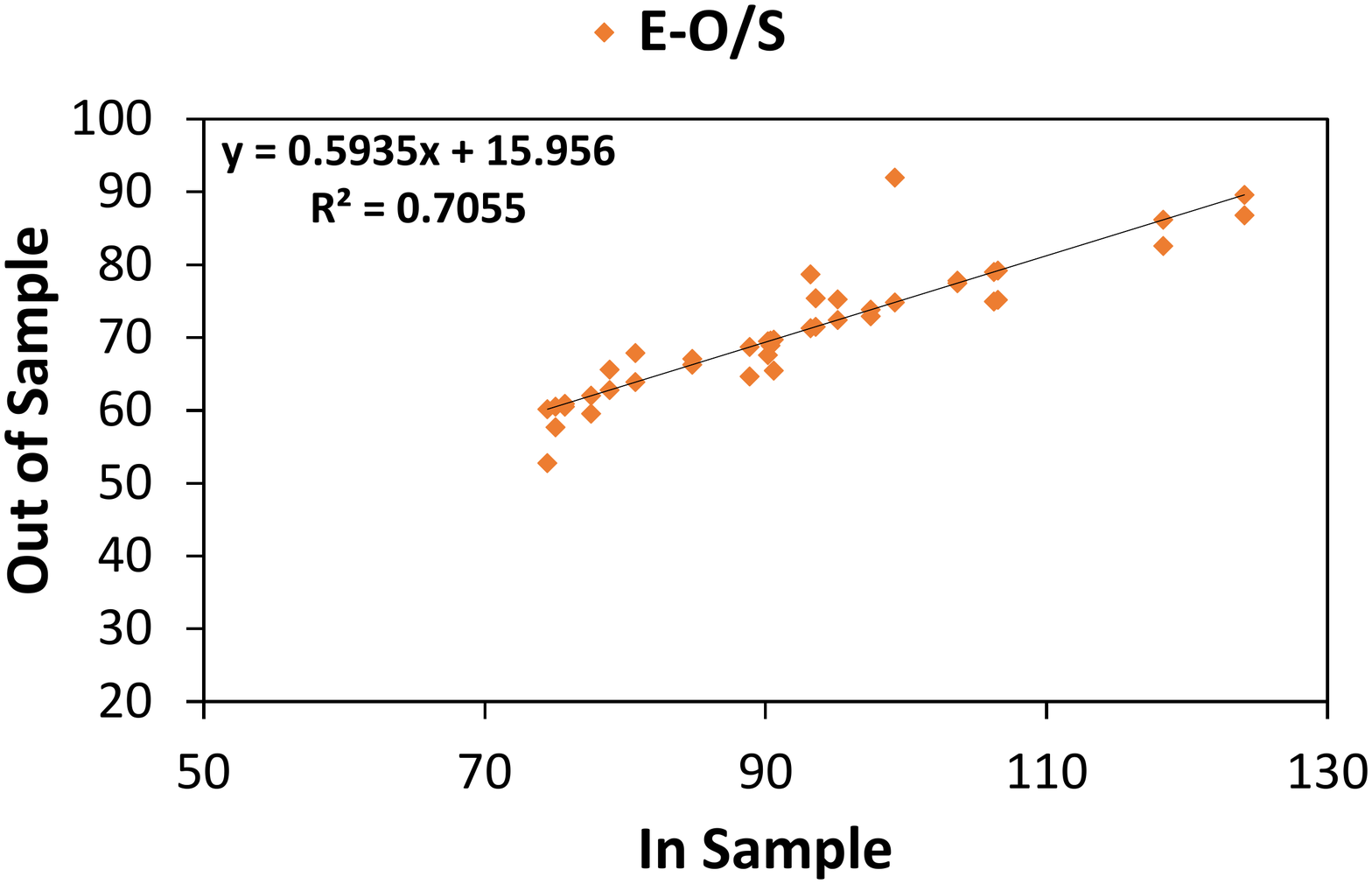}
\caption{Robust solutions, expected demand gap.}\label{fig2b}
\end{subfigure}
\vspace*{5mm}
\begin{subfigure}{.45\textwidth}
\includegraphics[width=1.1\linewidth]{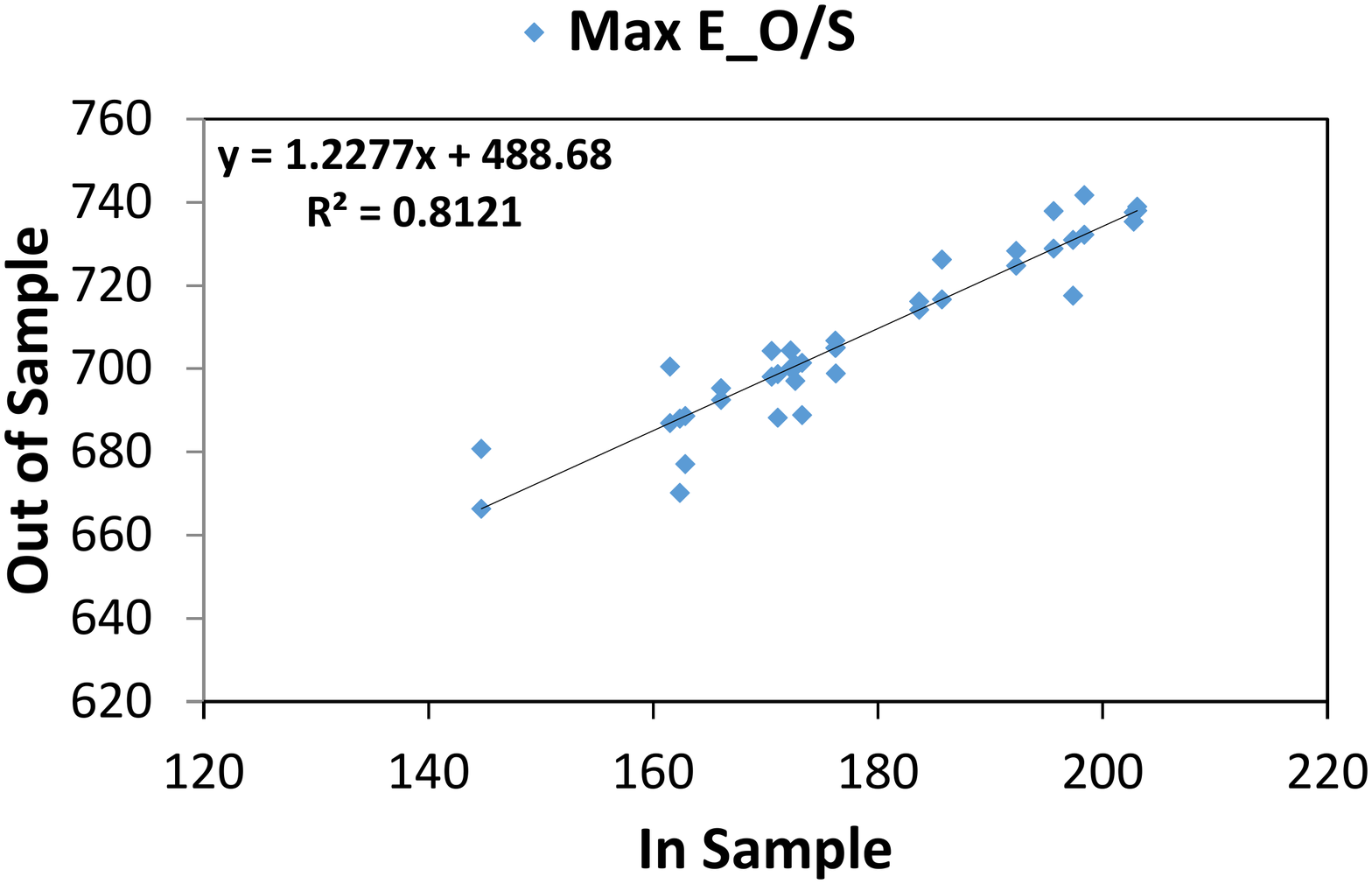}
\caption{DRSO solutions, expected max demand gap.}\label{fig2c}
\end{subfigure}
\hspace{1cm}
\begin{subfigure}{.45\textwidth}
\includegraphics[width=1.1\linewidth]{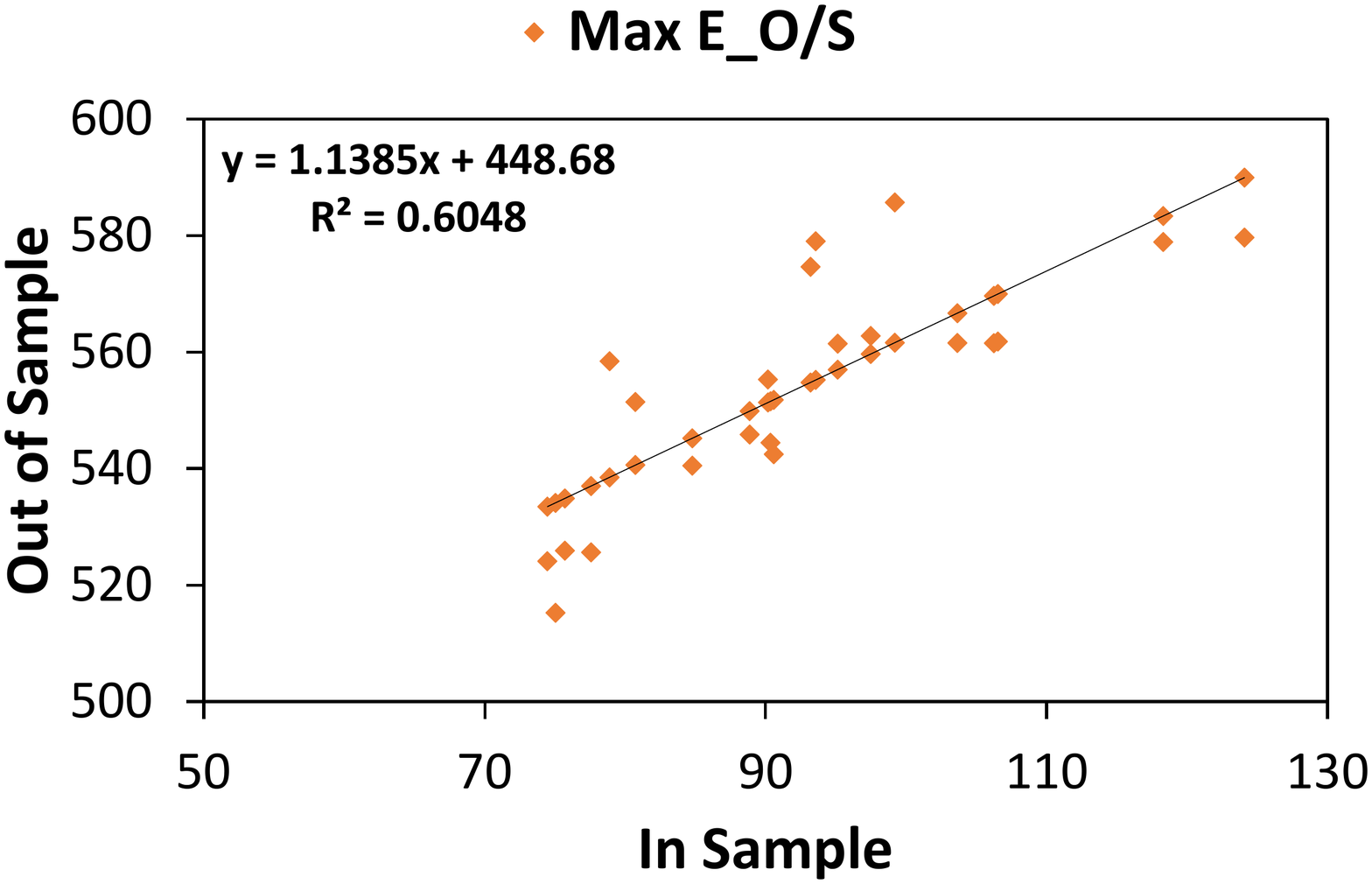}
\caption{Robust solutions, expected max demand gap.}\label{fig2d}
\end{subfigure}
%
%
\begin{subfigure}{.45\textwidth}
\includegraphics[width=1.1\linewidth]{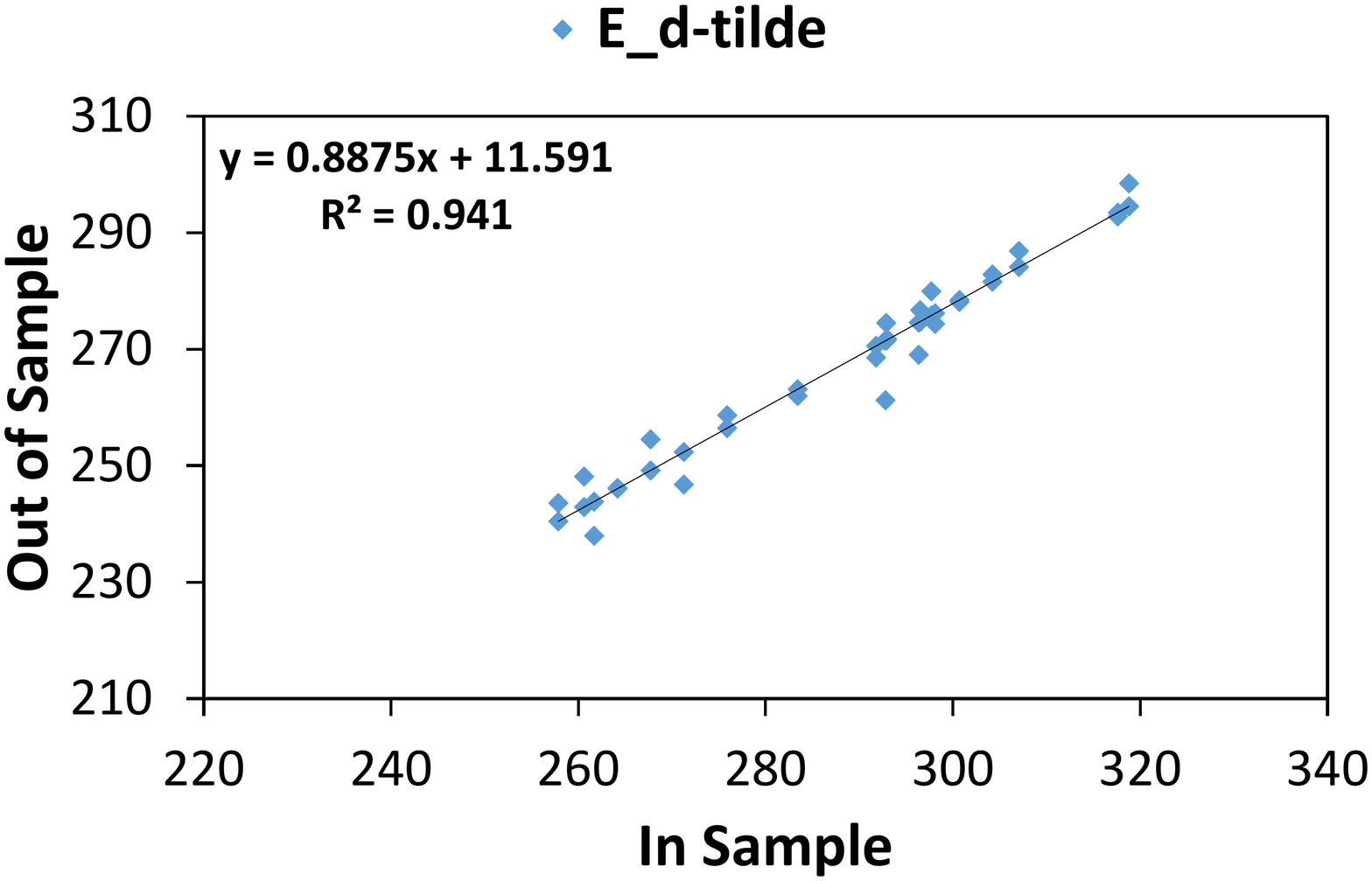}
\caption{DRSO solutions, expected satisfied demand gap.}\label{fig2e}
\end{subfigure}
\caption{Results of out-of-sample prediction (commodity A).}\label{fig2}
\end{center}
\end{figure}

A similar trend is also observed for the satisfied demand ($\tilde{d}$) metric with an I/O gap of $7.35\%$ (commodities B and C respectively are $7.23\%$ and $4.20\%$) while the regression result in \autoref{fig2e} shows that $94.10\%$ (commodities B and C respectively are $94.78\%$ and $98.83\%$) variation in the expected satisfied demand (E[$\tilde{d}$]) is explained by the in-sample result, which means that $5.90\%$ variation in the expected satisfied demand is not due to the in-sample satisfied demand.

Although the robust solutions follow a pessimistic route and build more capacity for the same demand instances, there is no observed significant difference in the average unit cost of capacity for these two models irrespective of commodity type and data set, see \autoref{Tab4}. For commodity A, for instance, with average unit capacity cost of $38.50$ for the DRSO solutions and $38.76$ for the robust solutions.

Additional insight on this is provided by \autoref{fig3b}, which compares the unit cost per instance and by \autoref{fig3c}, which shows similar linear relationship between capacity and investment for the two models. However, if the capacity installed is compared to the total investment cost, the unit cost of robust solutions becomes cheaper. 

%

\begin{figure}[!htb]
\begin{center}
\begin{subfigure}{.45\textwidth}
        \includegraphics[width=1.1\linewidth]{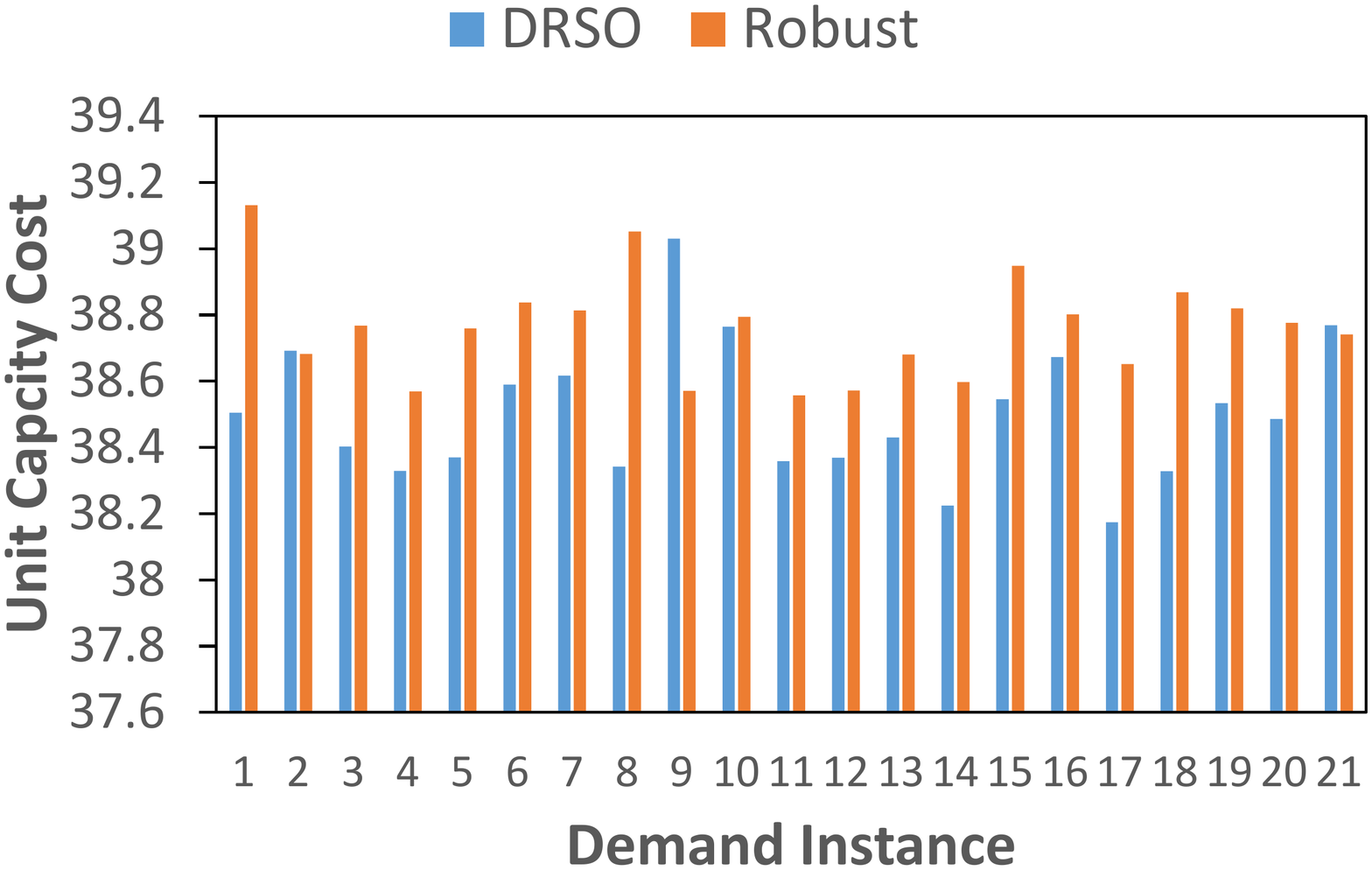}
        \caption{Unit cost of capacity.}\label{fig3b}
\end{subfigure}
\vspace*{5mm}
\begin{subfigure}{.45\textwidth}
\includegraphics[width=1.1\linewidth]{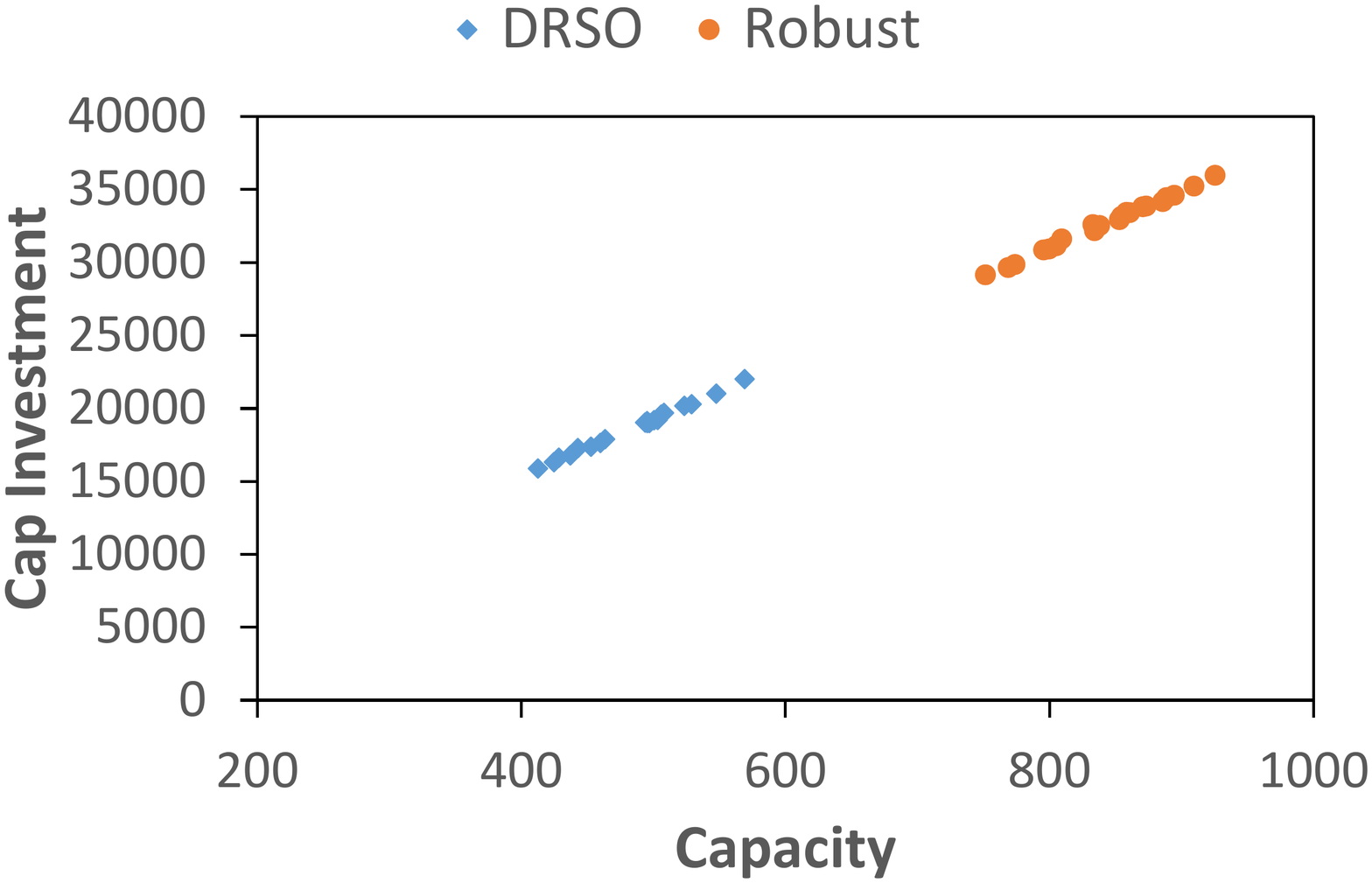}
\caption{Capacity investment.}\label{fig3c}
\end{subfigure}
\hspace{1cm}
\begin{subfigure}{.45\textwidth}
\includegraphics[width=1.1\linewidth]{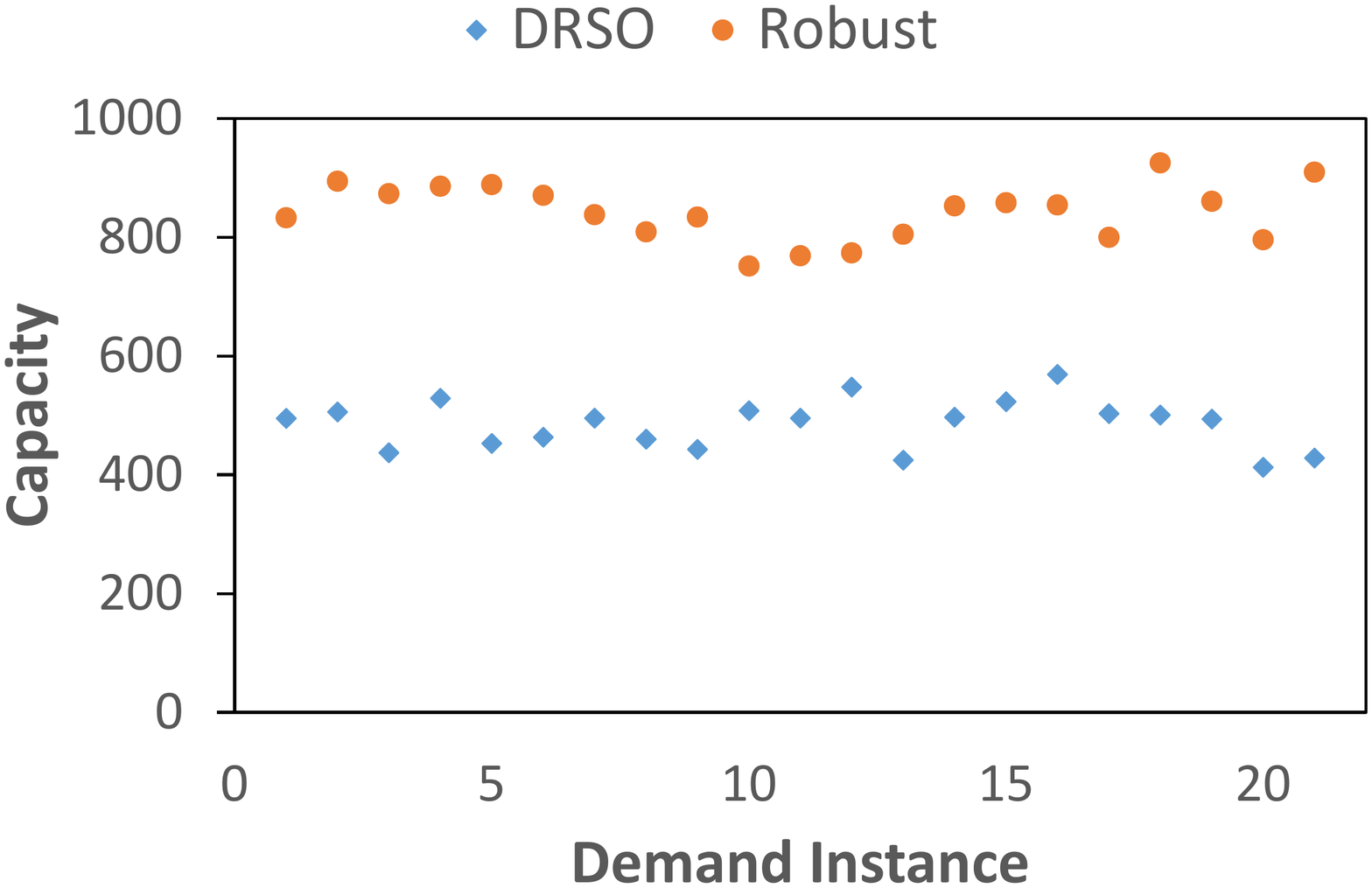}
\caption{Comparison of capacity deployed.}\label{fig3d}
\end{subfigure}
\caption{Investment efficiency (commodity A).}\label{fig3}
\end{center}
\end{figure}

\section{Conclusions}
\label{sec:conclude}

In this paper an efficient approach to distributionally robust network capacity planning under demand ambiguity was proposed. In this approach, we formulate the problem as bilevel optimization, where the worst-case distribution can be characterized by a two-point distribution. This allows us to reformulate the problem as a convex optimization problem, where we need to search over the demand $\tilde{\pmb{d}}$ we intent to satisfy.
We then solve this new model using Nelder-Mead algorithm, a convex optimization method.

In order to evaluate the quality of our new approach, the resulting model was compared with the robust approach model on the Nobel-US network taken from the SNDlib, \cite{Orlowski2010}, database on a number of performance metrics. Our computational result show that solutions from the DRSO model outperform those from the robust model on all high risk-averse performance metrics. Even in the area of solution robustness and quality where the later is generally of a higher robustness, the result scaling shows that solutions from DRSO outperform the robust model in this area on the high risk measures. The robust, however, performs better on the low risk-averse metric.

One interesting result which was also reported earlier by \cite{Nakao2017} using a different metric is the prediction accuracy of the DRSO model with over $90\%$ expected result variability explained by model result whereas the Robust model cannot be relied upon having a prediction accuracy of approximately $57\%$ and higher. It was also noted that despite the performance difference, the actual unit cost of capacity for this two model is not significantly different.

Moreover, the solutions based on the DRSO were found to be less conservative when compared to Robust model, irrespective of the observed demand instance, data set used and the commodity type, with lower total and capacity investment.

\begin{appendices}
\section{Proof of \cref{Nature1}}
\label{Nat_proof}
\begin{proof}
First suppose that $\tilde{d} \le \frac{\mu^2 + \sigma^2}{2\mu} $.\\	Here, Nature is characterized by a two point distribution defined by a one-sided Chebyshev inequality below;
	\begin{equation}
	T=
	\begin{cases}
	0 \quad &\mbox{w.p.} \hspace{0.3pc}  \  \frac{\sigma^2}{\sigma^2 +\mu^2} \\
	\frac{\sigma^2+\mu^2}{\mu} \quad &\mbox{w.p.} \hspace{0.3pc} \frac{\mu^2}{\sigma^2 + \mu^2} 
	\end{cases} 
	\end{equation}
	with mass $\frac{\sigma^2}{\sigma^2 +\mu^2} $ at $0$ and vice-versa.
	Hence, Nature becomes,
	
		\begin{equation*}
		N(\tilde{d},\chi_1) = (\chi_1-\tilde{d})\left(\frac{\mu^2}{\sigma^2 + \mu^2}\right)\label{Nt1}\\
		\end{equation*}
		where $\chi_1$ is the upper support, hence;
		\begin{align*}
		N(\tilde{d})& = \left(\frac{\sigma^2+\mu^2}{\mu} -\tilde{d}\right)\left(\frac{\mu^2}{\sigma^2 + \mu^2}\right)\\
		N(\tilde{d})&= \mu-\tilde{d}\left(\frac{\mu^2}{\sigma^2 + \mu^2}\right)
		\end{align*}
	\\
	Suppose now that  $\tilde{d} > \frac{\mu^2 + \sigma^2}{2\mu} $. Let $\chi_2$ be the upper support point of nature's distribution. Nature objective function can be written as:
	\begin{equation}
	N(\tilde{d},\chi_2) = (\chi_2-\tilde{d})(1-q),\label{Nt}
	\end{equation}
	where $q$ is probability mass on lower support point. 
	\smallskip
	We can express $q$ in terms of $\chi_2$ using the fact that 
	\begin{equation*}
	\chi_2 = \mu + \sigma \sqrt{\frac{q}{1-q}}, 
	\end{equation*}
	which gives
	\begin{equation*}
	q = \frac{(\chi_2-\mu)^2}{\sigma^2+(\chi_2-\mu)^2}.
	\end{equation*}
	which can be derived from the following two equations, where $\alpha$ is the lower support point ;
	\begin{align*}
	p\alpha+(1-p)\chi_2& = \mu \\ 
	p\alpha^2+(1-p)\chi_2^2&=\mu^2+\sigma^2 
	\end{align*}
	
	Then, \autoref{Nt} becomes;
	\begin{equation}
	N(\tilde{d},\chi_2)=(\chi_2-\tilde{d})\left (1-\frac{(\chi_2-\mu)^2}{(\chi_2-\mu)^2 + \sigma^2} \right ) \label{Nt3}
	\end{equation}
	
	The value of $\chi$ at the root maximizes the above function. To this end, the  first derivative of $N$ w.r.t $\chi$ is 
	\begin{equation*}
	\frac{\partial N}{\partial \chi_2} =  -\frac{\sigma^2(\chi_2^2 - 2\tilde{d}\chi_2 + 2\mu \tilde{d} - \mu^2 - \sigma^2)}{(\chi_2^2-2\mu \chi_2 + \sigma^2 + \mu^2)^2}, 
	\end{equation*} 
	\smallskip
	Setting $\frac{\partial N}{\partial \chi_2} =0$, produces a root at $\chi_2 = \tilde{d}+\sqrt{(\tilde{d}-\mu)^2 + \sigma^2}$ which when substituted in \autoref{Nt3} gives
	\begin{equation*}
	N(\tilde{d})=\left(\tilde{d}+\sqrt{(\tilde{d}-\mu)^2 + \sigma^2}-\tilde{d}\right)\left (1-\frac{\left(\tilde{d}+\sqrt{(\tilde{d}-\mu)^2 + \sigma^2}-\mu\right)^2}{\left(\tilde{d}+\sqrt{(\tilde{d}-\mu)^2 + \sigma^2}-\mu\right)^2 + \sigma^2} \right ) \label{Nt2}
	\end{equation*}
	Simplifying the above equation and re-arranging terms result in the below;
	\begin{equation*}
	N(\tilde{d})=1/2 \left(\mu-\tilde{d}+\sqrt{(\tilde{d}-\mu)^2 +\sigma^2}\right)
	\end{equation*}
	and this completes the proof of \autoref{fln1_3}.
	
	\begin{equation*}
	N(\tilde{d})=
	\begin{cases}
	1/2 \left(\mu-\tilde{d}+\sqrt{(\tilde{d}-\mu)^2 +\sigma^2}\right) \quad &\mbox{when} \hspace{0.3pc}  \ \tilde{d} > \frac{\mu^2 + \sigma^2}{2\mu} \label{fln1_2}\\
	\mu - \tilde{d}\left(\frac{\mu^2}{\mu^2 + \sigma^2}\right) \quad &\mbox{when} \hspace{0.3pc} \tilde{d} \le \frac{\mu^2 + \sigma^2}{2\mu} 
	\end{cases} 
	\end{equation*}
\end{proof}\\

\section{Results for commodity type B.}
\label{AppendA}

We present additional results for commodity type B in a similar way to the presentation of results for commodity type A in the main text. \autoref{Tab2B} and \autoref{Tab3B} show key metrics for the 21 repetitions using the robust and the DRSO model, respectively.

\begin{table}[htbp]\footnotesize
	\centering
	\caption{Robust model results for commodity type B.}
	\begin{tabular}{rrcrcrrrr}
		\toprule 
		&\multicolumn{2}{c}{In Sample} 
		&\multicolumn{5}{c}{Out of Sample}&\\	
		\cmidrule (r) {2-3} \cmidrule (l) {4-8}
		Inst. &    Cap. Inv.&	O/S Demand&	E[O/S]& E[Max O/S]&	CVaR95&	CVaR75&	E[$\tilde{d}$]&	CapAdd \\
		\midrule
	1	&	 26,955.23 	&	 138.21 	 & 	 94.47 	 & 	 614.68 	 & 	 458.99 	 & 	 272.58 	 & 	 342.28 	 & 	 682.99 	\\
	2	&	 27,340.83 	&	 122.58 	 & 	 95.06 	 & 	 592.60 	 & 	 436.90 	 & 	 259.98 	 & 	 358.40 	 & 	 690.01 	\\
	3	&	 27,966.59 	&	 109.42 	 & 	 88.49 	 & 	 601.61 	 & 	 445.92 	 & 	 262.78 	 & 	 349.14 	 & 	 698.86 	\\
	4	&	 30,966.21 	&	 96.20 	 & 	 73.17 	 & 	 569.72 	 & 	 414.03 	 & 	 232.73 	 & 	 356.59 	 & 	 787.31 	\\
	5	&	 27,081.86 	&	 135.60 	 & 	 93.40 	 & 	 600.58 	 & 	 444.89 	 & 	 262.61 	 & 	 342.51 	 & 	 698.33 	\\
	6	&	 32,108.77 	&	 98.27 	 & 	 71.99 	 & 	 540.65 	 & 	 384.96 	 & 	 212.76 	 & 	 384.38 	 & 	 824.14 	\\
	7	&	 25,174.15 	&	 144.43 	 & 	 109.13 	 & 	 615.47 	 & 	 459.78 	 & 	 280.33 	 & 	 340.72 	 & 	 627.89 	\\
	8	&	 28,144.28 	&	 125.35 	 & 	 85.90 	 & 	 588.84 	 & 	 433.15 	 & 	 250.75 	 & 	 363.18 	 & 	 717.50 	\\
	9	&	 33,905.19 	&	 84.59 	 & 	 62.63 	 & 	 531.42 	 & 	 375.74 	 & 	 206.53 	 & 	 394.00 	 & 	 871.49 	\\
	10	&	 25,598.99 	&	 127.12 	 & 	 101.55 	 & 	 622.91 	 & 	 467.21 	 & 	 284.64 	 & 	 329.08 	 & 	 656.61 	\\
	11	&	 30,395.00 	&	 100.40 	 & 	 73.78 	 & 	 554.45 	 & 	 399.09 	 & 	 223.45 	 & 	 379.72 	 & 	 774.90 	\\
	12	&	 29,084.16 	&	 112.90 	 & 	 84.14 	 & 	 583.00 	 & 	 427.31 	 & 	 244.61 	 & 	 358.90 	 & 	 743.51 	\\
	13	&	 30,291.89 	&	 109.15 	 & 	 80.03 	 & 	 567.86 	 & 	 412.16 	 & 	 235.45 	 & 	 352.00 	 & 	 764.24 	\\
	14	&	 27,374.82 	&	 111.42 	 & 	 87.46 	 & 	 593.87 	 & 	 438.18 	 & 	 254.00 	 & 	 356.60 	 & 	 701.42 	\\
	15	&	 28,590.93 	&	 112.90 	 & 	 82.91 	 & 	 593.14 	 & 	 437.45 	 & 	 251.73 	 & 	 351.84 	 & 	 731.89 	\\
	16	&	 30,027.87 	&	 95.84 	 & 	 75.90 	 & 	 560.51 	 & 	 404.82 	 & 	 226.01 	 & 	 367.25 	 & 	 764.01 	\\
	17	&	 27,506.54 	&	 111.69 	 & 	 92.84 	 & 	 603.45 	 & 	 447.76 	 & 	 266.75 	 & 	 337.79 	 & 	 703.11 	\\
	18	&	 27,813.21 	&	 121.74 	 & 	 89.64 	 & 	 593.74 	 & 	 438.04 	 & 	 253.97 	 & 	 354.84 	 & 	 709.47 	\\
	19	&	 31,422.87 	&	 100.67 	 & 	 78.99 	 & 	 550.96 	 & 	 395.27 	 & 	 218.36 	 & 	 357.05 	 & 	 800.68 	\\
	20	&	 26,886.16 	&	 118.71 	 & 	 95.24 	 & 	 619.57 	 & 	 463.88 	 & 	 278.46 	 & 	 341.75 	 & 	 685.73 	\\
	21	&	 27,750.82 	&	 123.86 	 & 	 93.98 	 & 	 594.73 	 & 	 439.04 	 & 	 260.02 	 & 	 343.73 	 & 	 701.20 	\\
		\bottomrule
	\end{tabular}
	\label{Tab2B}
\end{table}

\begin{table}[htbp]\footnotesize
	\centering
	\caption{DRSO model results for commodity type B.}
	\begin{tabular}{rrrcrcrrrr}
		\toprule 
		&&\multicolumn{2}{l}{In Sample} 
		&\multicolumn{5}{c}{Out of Sample}&\\ 
		\cmidrule (r) {2-4} \cmidrule (l) {5-9}
		Inst. &    Cap. Inv.& Nature& $\tilde{d}$ &E[O/S]& E[ Max O/S]& CVaR95& CVaR75& E[$\tilde{d}$]& CapAdd \\
		\midrule
	1	&	 17,917.30 	 & 	 197.00 	 & 	 272.84 	 & 	 169.65 	 & 	 726.73 	 & 	 571.04 	 & 	 384.06 	 & 	 252.57 	 & 	 450.22 	\\
	2	&	 20,356.05 	 & 	 156.43 	 & 	 316.57 	 & 	 140.34 	 & 	 683.00 	 & 	 527.31 	 & 	 340.33 	 & 	 293.09 	 & 	 513.17 	\\
	3	&	 19,387.54 	 & 	 164.77 	 & 	 292.85 	 & 	 153.23 	 & 	 706.72 	 & 	 551.03 	 & 	 364.05 	 & 	 273.34 	 & 	 491.20 	\\
	4	&	 18,059.27 	 & 	 191.98 	 & 	 271.09 	 & 	 173.10 	 & 	 728.48 	 & 	 572.79 	 & 	 386.11 	 & 	 241.03 	 & 	 460.61 	\\
	5	&	 19,949.98 	 & 	 166.55 	 & 	 306.07 	 & 	 144.27 	 & 	 693.50 	 & 	 537.81 	 & 	 350.83 	 & 	 279.57 	 & 	 515.17 	\\
	6	&	 21,042.16 	 & 	 153.28 	 & 	 319.69 	 & 	 132.98 	 & 	 679.88 	 & 	 524.19 	 & 	 337.21 	 & 	 289.33 	 & 	 538.45 	\\
	7	&	 20,756.93 	 & 	 149.48 	 & 	 323.17 	 & 	 133.46 	 & 	 676.40 	 & 	 520.71 	 & 	 333.73 	 & 	 299.24 	 & 	 527.19 	\\
	8	&	 19,641.00 	 & 	 171.67 	 & 	 310.38 	 & 	 146.93 	 & 	 689.19 	 & 	 533.50 	 & 	 346.52 	 & 	 292.60 	 & 	 494.15 	\\
	9	&	 19,625.72 	 & 	 161.37 	 & 	 310.06 	 & 	 143.92 	 & 	 689.51 	 & 	 533.82 	 & 	 346.83 	 & 	 291.45 	 & 	 494.48 	\\
	10	&	 19,318.96 	 & 	 153.30 	 & 	 305.69 	 & 	 145.59 	 & 	 693.88 	 & 	 538.19 	 & 	 351.21 	 & 	 283.81 	 & 	 493.99 	\\
	11	&	 18,554.09 	 & 	 179.70 	 & 	 291.15 	 & 	 158.88 	 & 	 708.42 	 & 	 552.73 	 & 	 366.19 	 & 	 263.01 	 & 	 471.08 	\\
	12	&	 15,903.68 	 & 	 203.29 	 & 	 267.32 	 & 	 186.02 	 & 	 732.25 	 & 	 576.56 	 & 	 389.58 	 & 	 253.73 	 & 	 401.86 	\\
	13	&	 16,550.21 	 & 	 197.09 	 & 	 275.00 	 & 	 175.14 	 & 	 724.57 	 & 	 568.88 	 & 	 382.16 	 & 	 258.62 	 & 	 420.55 	\\
	14	&	 17,326.99 	 & 	 183.18 	 & 	 280.58 	 & 	 166.19 	 & 	 718.99 	 & 	 563.29 	 & 	 376.65 	 & 	 254.37 	 & 	 448.43 	\\
	15	&	 18,206.14 	 & 	 177.24 	 & 	 290.24 	 & 	 161.64 	 & 	 709.33 	 & 	 553.64 	 & 	 366.66 	 & 	 268.19 	 & 	 469.62 	\\
	16	&	 20,826.06 	 & 	 139.73 	 & 	 329.30 	 & 	 131.07 	 & 	 670.27 	 & 	 514.58 	 & 	 327.60 	 & 	 303.62 	 & 	 529.05 	\\
	17	&	 20,943.80 	 & 	 152.76 	 & 	 323.21 	 & 	 133.57 	 & 	 676.36 	 & 	 520.66 	 & 	 333.68 	 & 	 301.89 	 & 	 526.64 	\\
	18	&	 21,479.18 	 & 	 146.50 	 & 	 332.86 	 & 	 127.91 	 & 	 666.71 	 & 	 511.02 	 & 	 324.06 	 & 	 309.97 	 & 	 542.32 	\\
	19	&	 19,553.38 	 & 	 159.83 	 & 	 314.98 	 & 	 142.59 	 & 	 684.59 	 & 	 528.89 	 & 	 342.11 	 & 	 296.98 	 & 	 498.68 	\\
	20	&	 20,173.40 	 & 	 156.14 	 & 	 312.15 	 & 	 142.32 	 & 	 687.42 	 & 	 531.73 	 & 	 344.75 	 & 	 287.68 	 & 	 510.07 	\\
	21	&	 18,786.10 	 & 	 164.48 	 & 	 309.88 	 & 	 151.53 	 & 	 689.69 	 & 	 534.00 	 & 	 347.02 	 & 	 293.73 	 & 	 479.52 	\\
		\bottomrule
	\end{tabular}
	\label{Tab3B}
\end{table}

\autoref{fig1B}, \autoref{fig5B} and \autoref{figB2} correspond to \autoref{fig1}, \autoref{fig5} and \autoref{fig2} using commodity type B instead of A.

\begin{figure}[htbp]
	\begin{center}
		\begin{subfigure}{.45\textwidth}
			\includegraphics[width=1.1\linewidth]{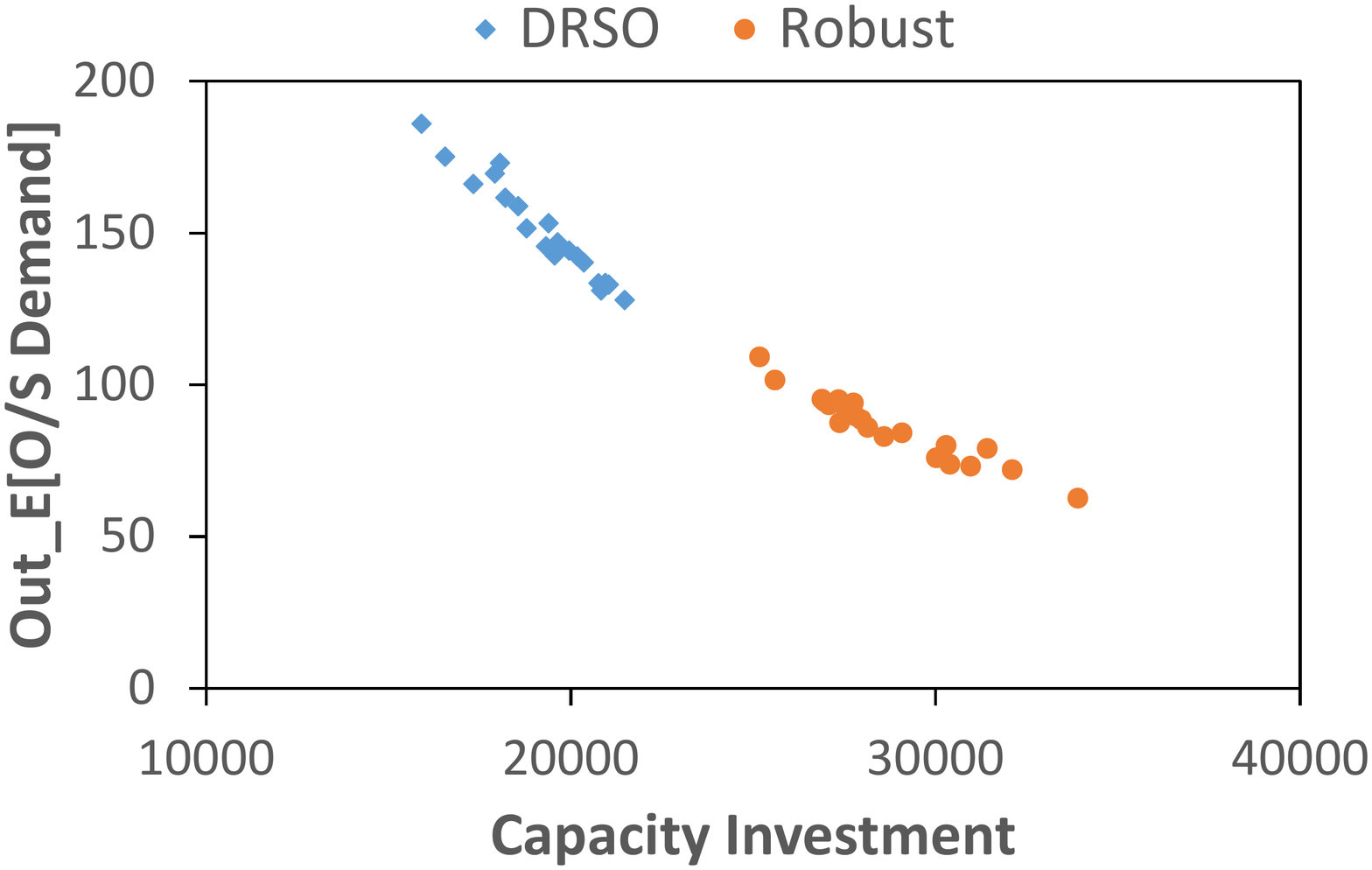}
			\caption{Expected unsatisfied demand.}\label{fig1Ba}
		\end{subfigure}
		\hspace{1cm}
		\begin{subfigure}{.45\textwidth}
			\includegraphics[width=1.1\linewidth]{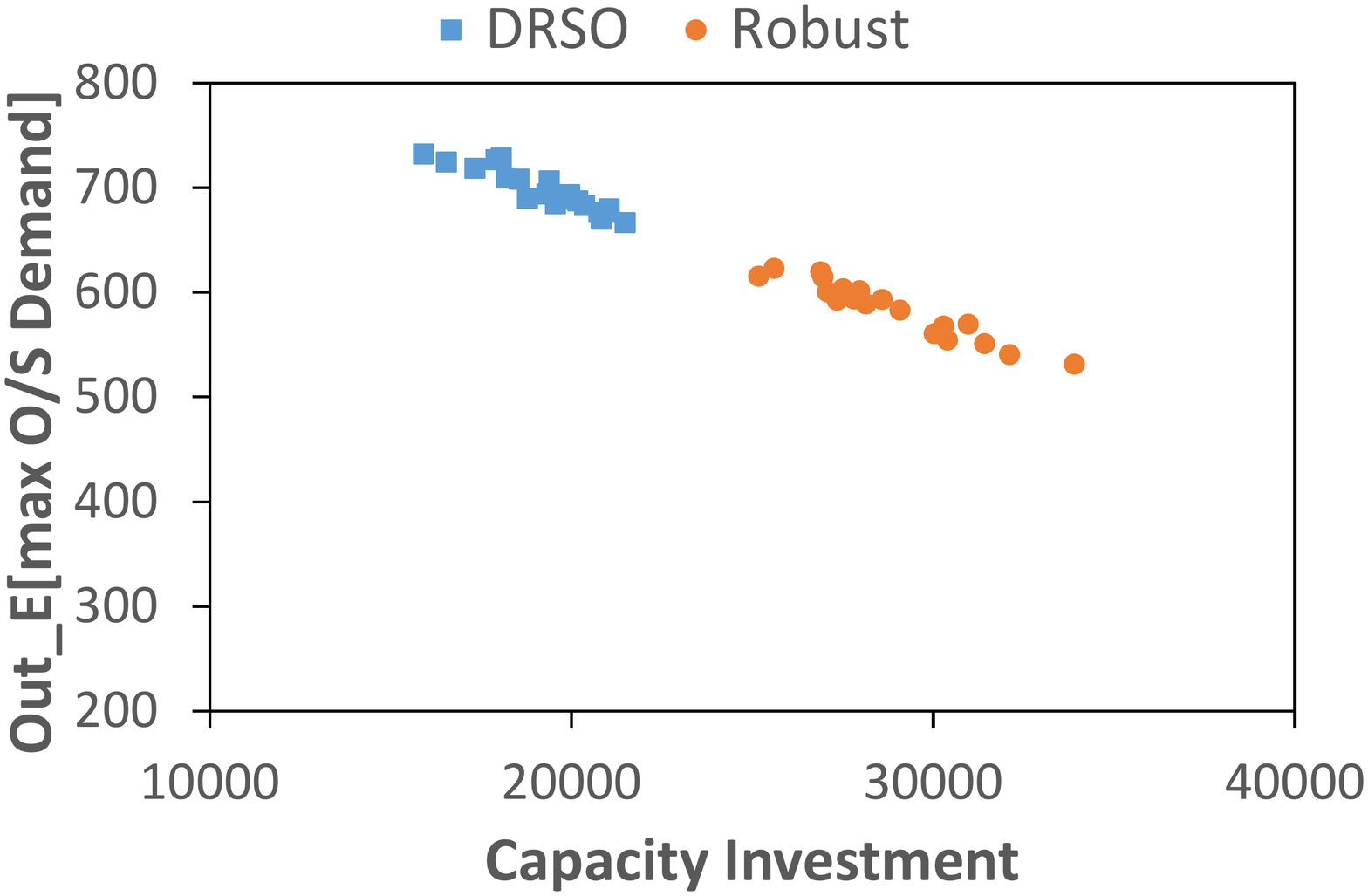}
			\caption{Expected max unsatisfied demand.}\label{fig1Bb}
		\end{subfigure}
		\vspace*{5mm}
		\begin{subfigure}{.45\textwidth}
			\includegraphics[width=1.1\linewidth]{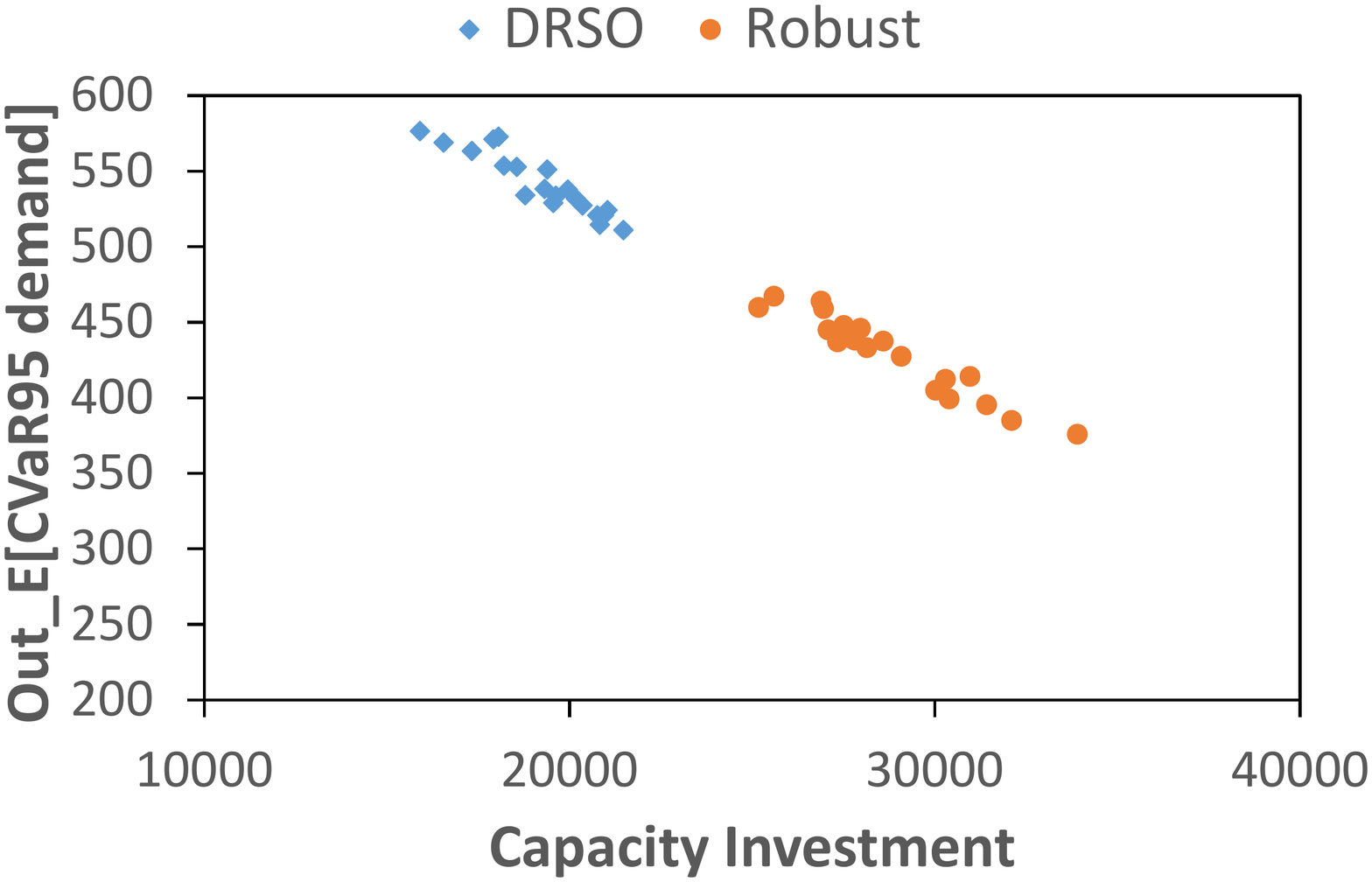}
			\caption{CVaR95 unsatisfied demand.}\label{figB1c}
		\end{subfigure}
		\hspace{1cm}
		\begin{subfigure}{.45\textwidth}
			\includegraphics[width=1.1\linewidth]{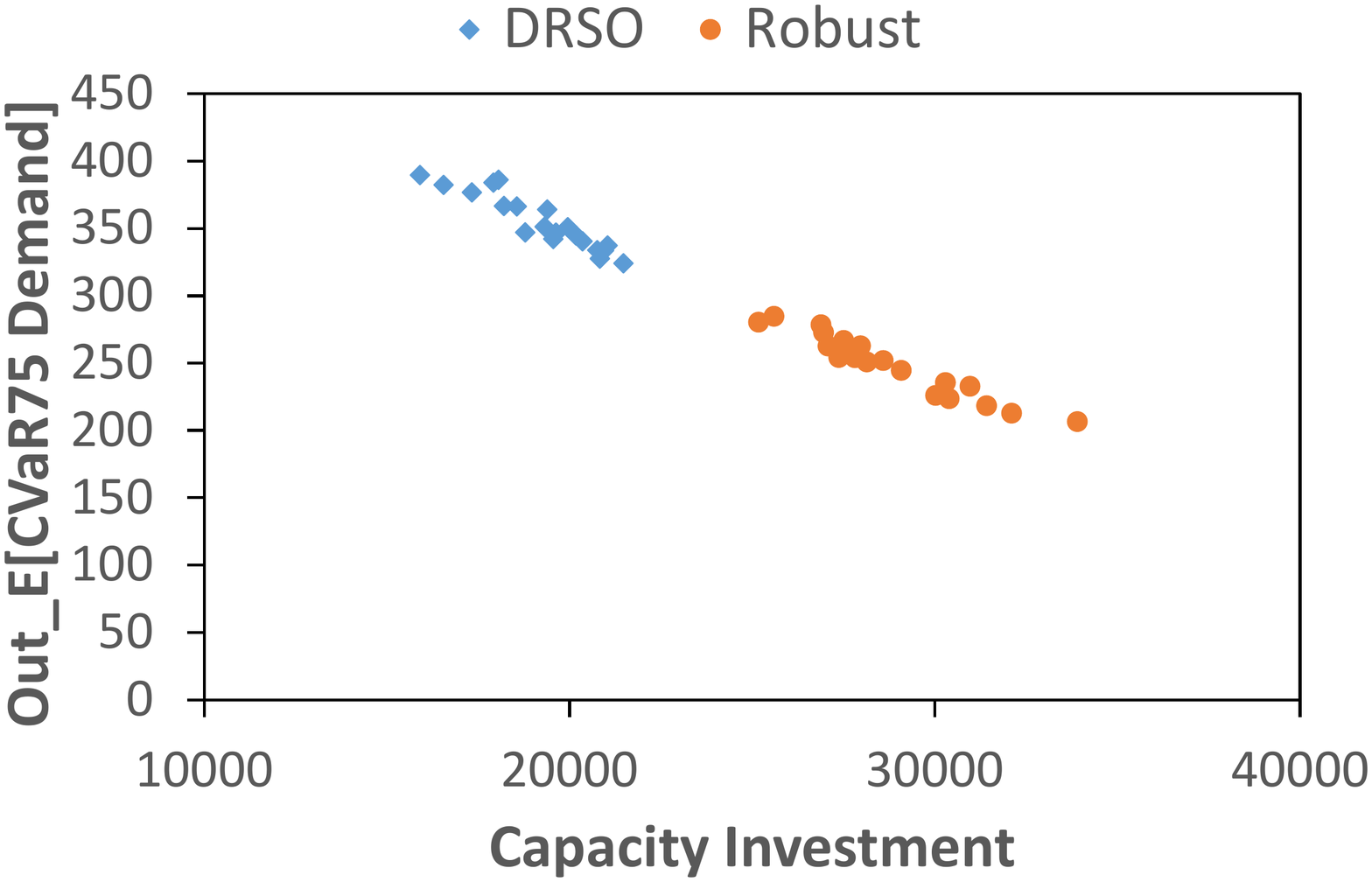}
			\caption{CVaR75 unsatisfied demand.}\label{fig1Bd}
		\end{subfigure}
		\caption{Expected unsatisfied demand mean and risk measures (commodity B).}\label{fig1B}
	\end{center}
\end{figure}

\begin{figure}[htbp]
	\begin{center}
		\begin{subfigure}{.45\textwidth}
			\includegraphics[width=1.1\linewidth]{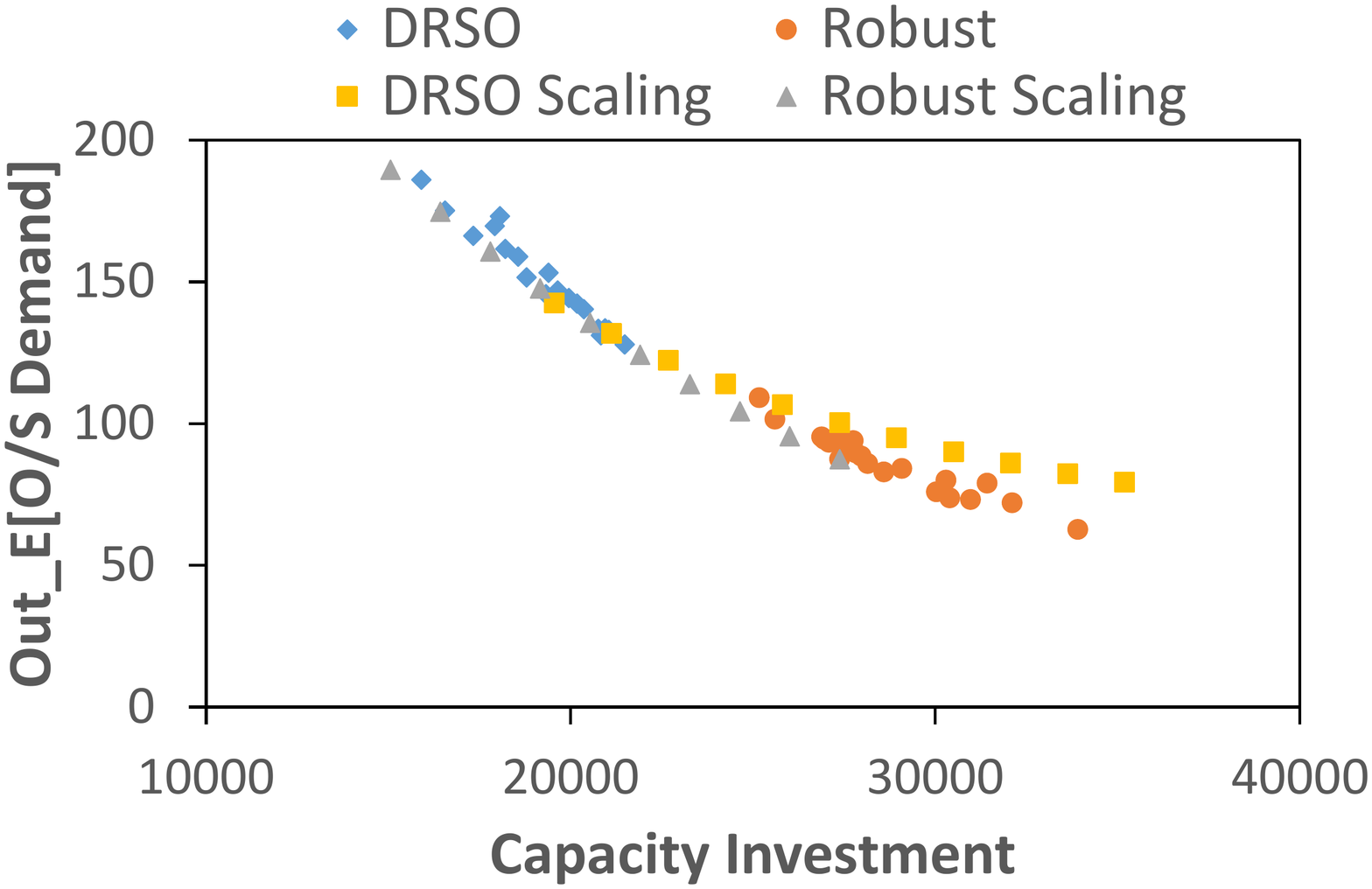}
			\caption{Expected unsatisfied demand.}\label{fig5Ba}
		\end{subfigure}
		\hspace{1cm}
		\begin{subfigure}{.45\textwidth}
			\includegraphics[width=1.1\linewidth]{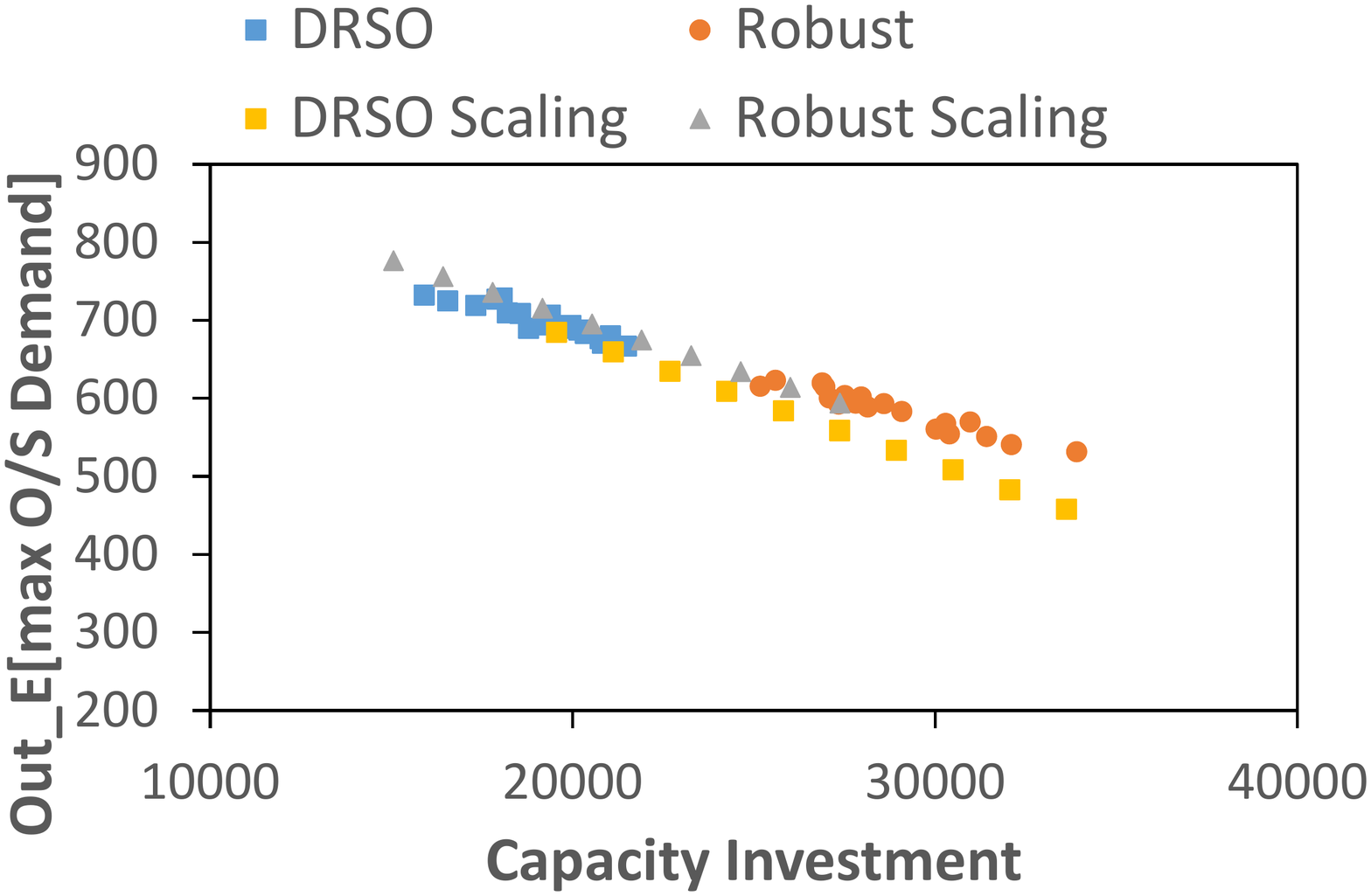}
			\caption{Expected maximum unsatisfied demand.}\label{fig5Bb}
		\end{subfigure}
		\vspace*{5mm}
		\begin{subfigure}{.45\textwidth}
			\includegraphics[width=1.1\linewidth]{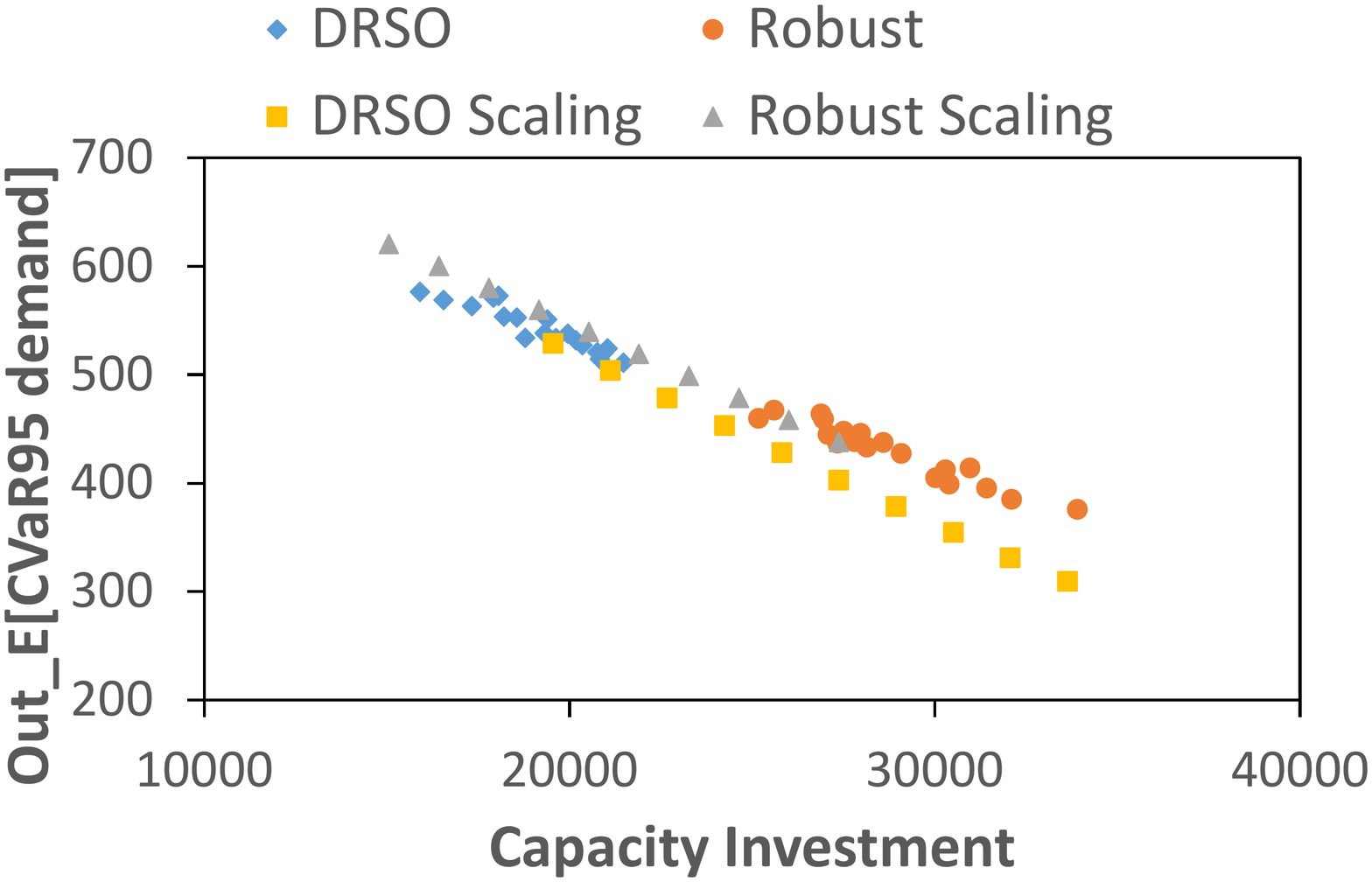}
			\caption{CVaR95 unsatisfied demand}\label{fig5Bc}
		\end{subfigure}
		\hspace{1cm}
		\begin{subfigure}{.45\textwidth}
			\includegraphics[width=1.1\linewidth]{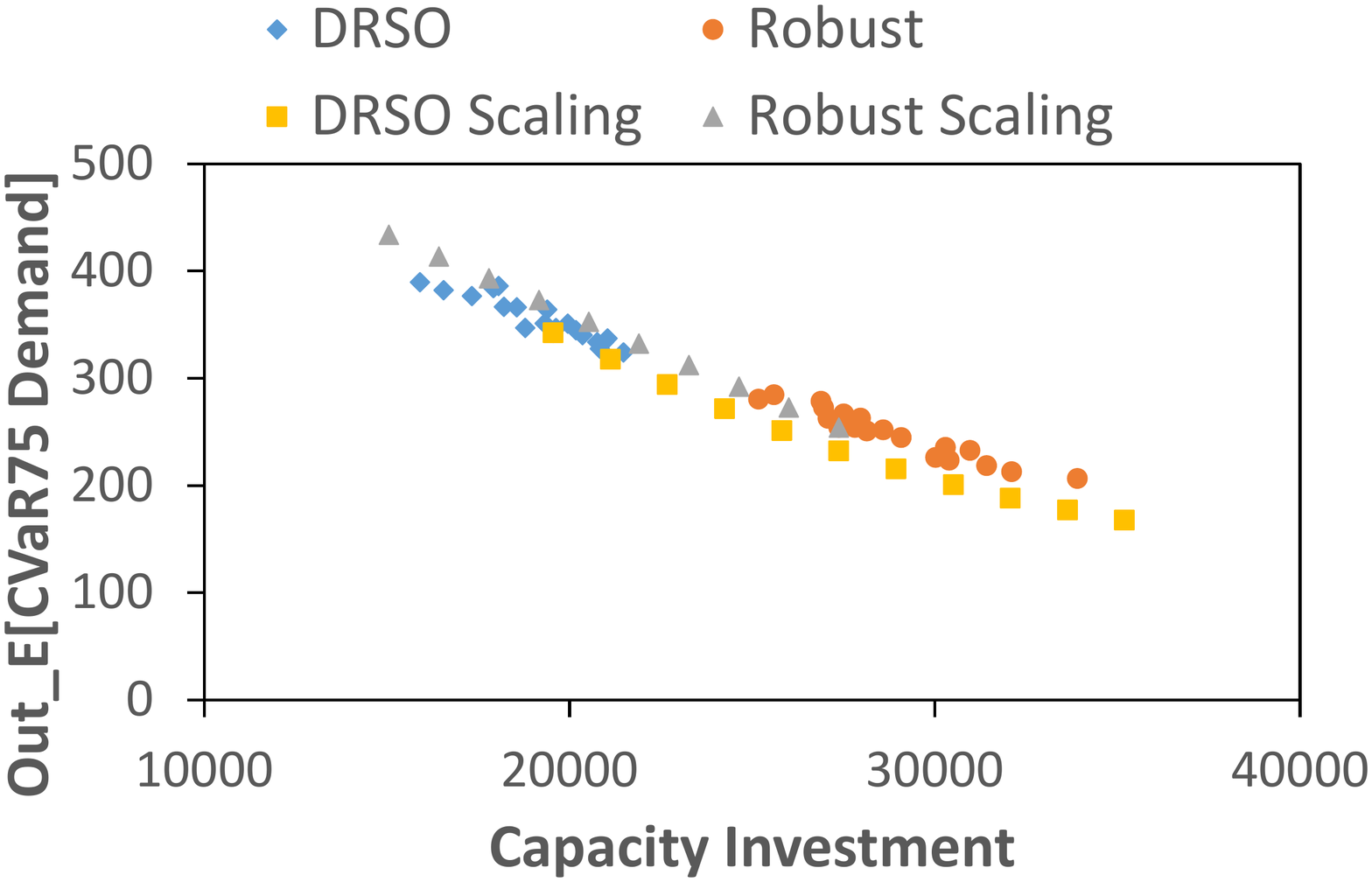}
			\caption{CVaR75 unsatisfied demand.}\label{fig5Bd}
		\end{subfigure}
		\caption{Performance metric scaling (commodity B).}\label{fig5B}
	\end{center}
\end{figure}

\begin{figure}[htbp]
	\begin{center}
		\begin{subfigure}{.45\textwidth}
			\includegraphics[width=1.1\linewidth]{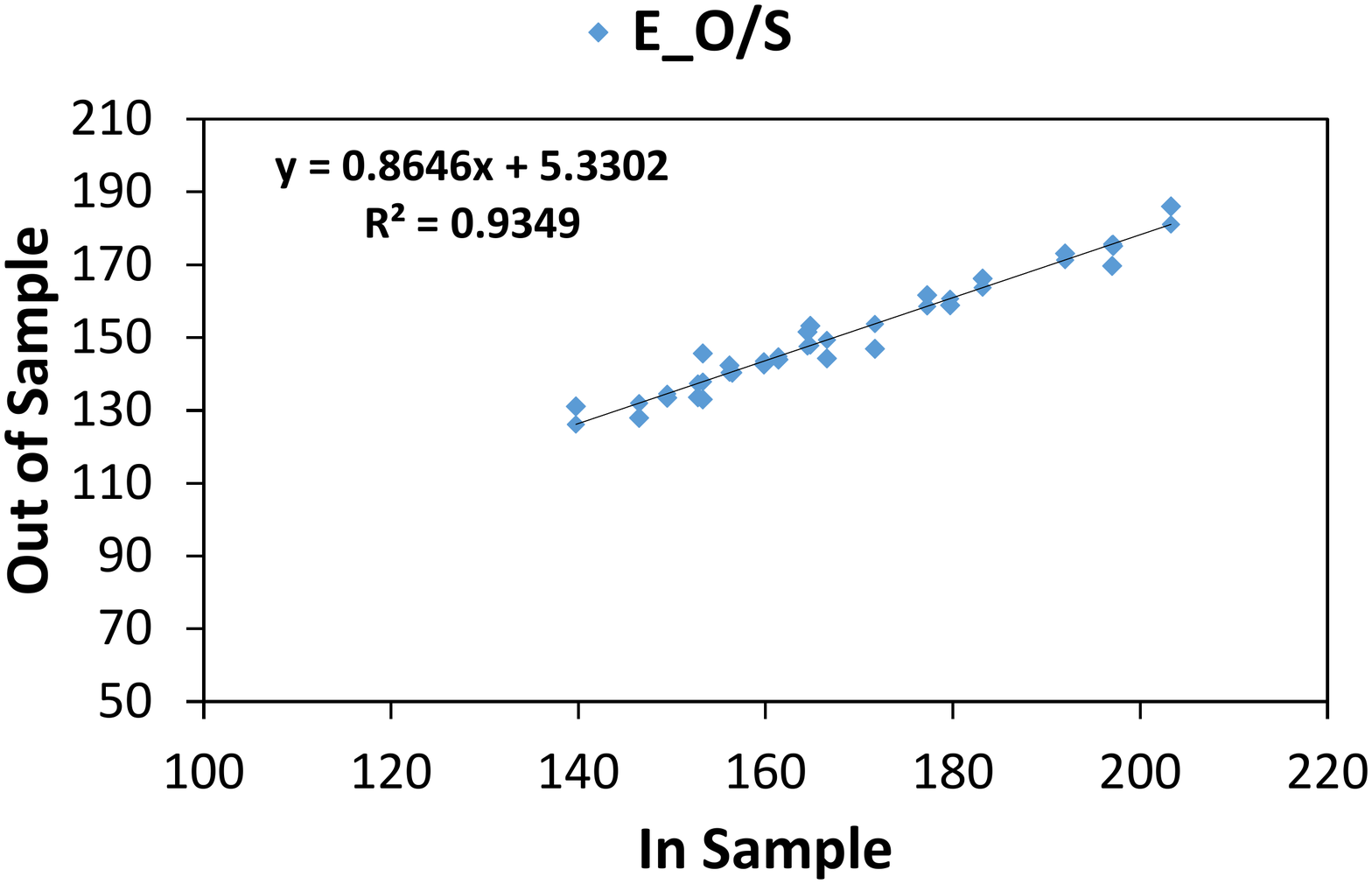}
			\caption{DRSO solutions, expected demand gap.}\label{fig2Ba}
		\end{subfigure}
		\hspace{1cm}
		\begin{subfigure}{.45\textwidth}
			\includegraphics[width=1.1\linewidth]{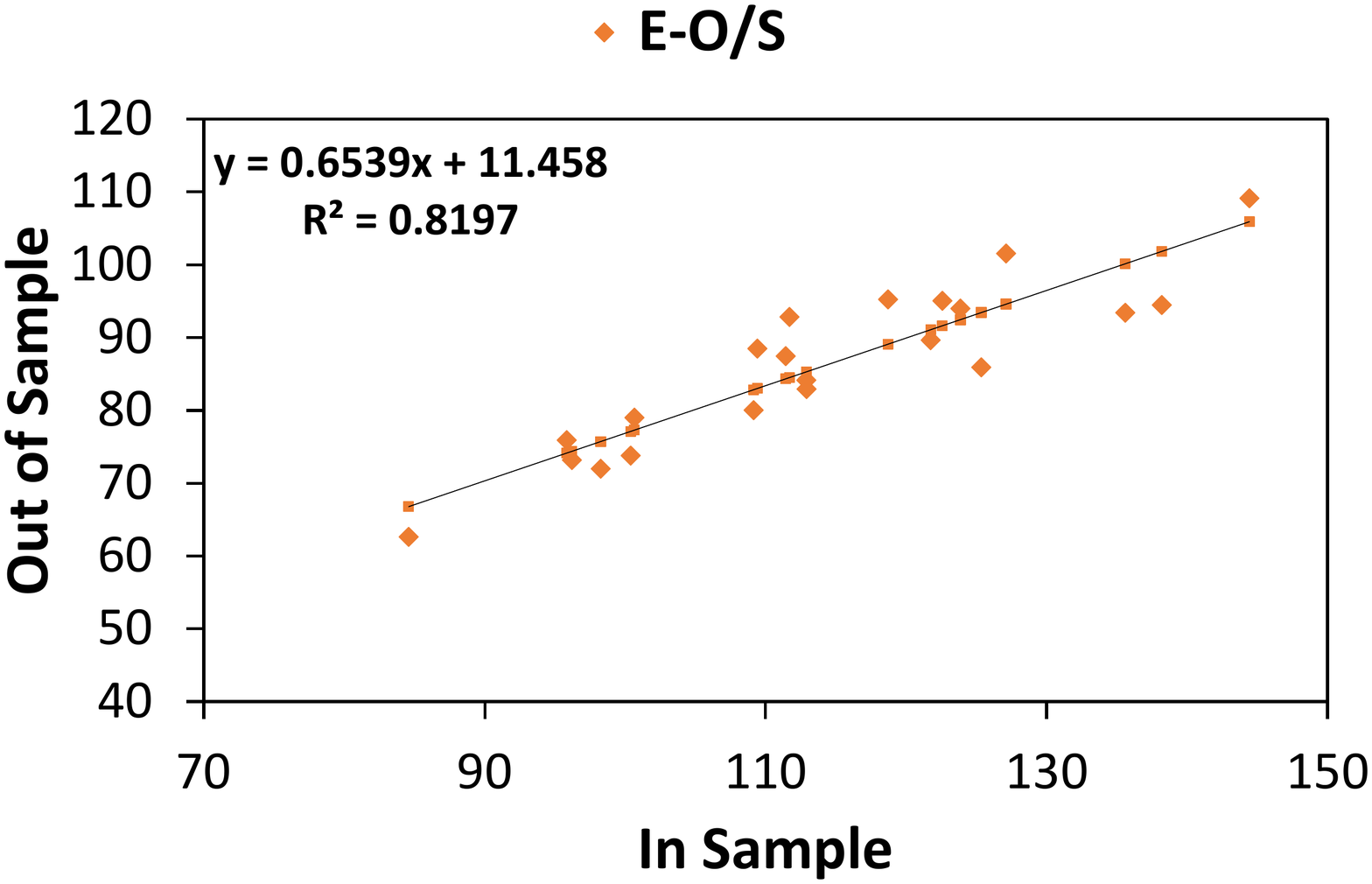}
			\caption{Robust solutions, expected demand gap.}\label{fig2Bb}
		\end{subfigure}
		\vspace*{5mm}
		\begin{subfigure}{.45\textwidth}
			\includegraphics[width=1.1\linewidth]{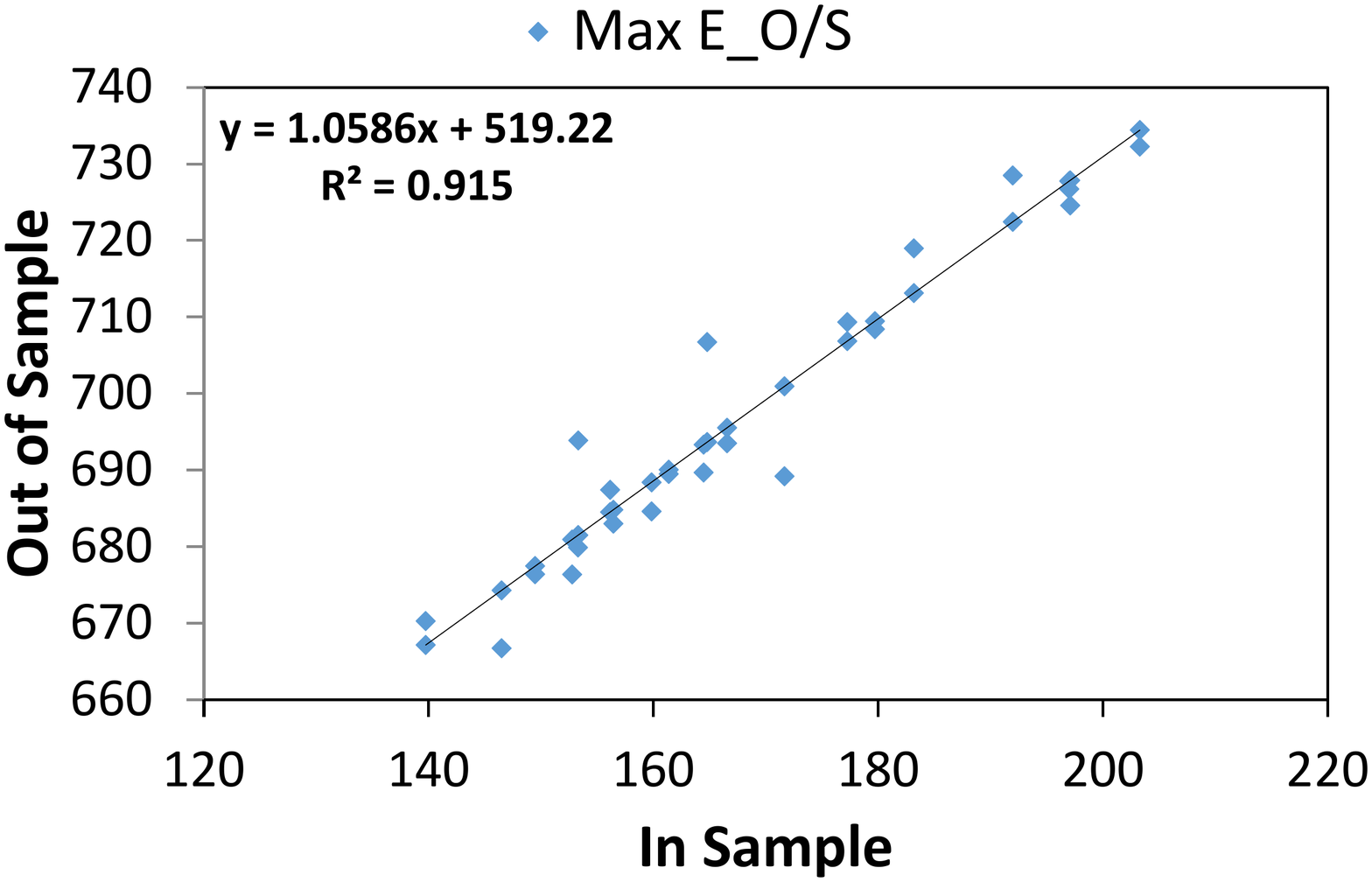}
			\caption{DRSO solutions, expected max demand gap.}\label{fig2Bc}
		\end{subfigure}
		\hspace{1cm}
		\begin{subfigure}{.45\textwidth}
			\includegraphics[width=1.1\linewidth]{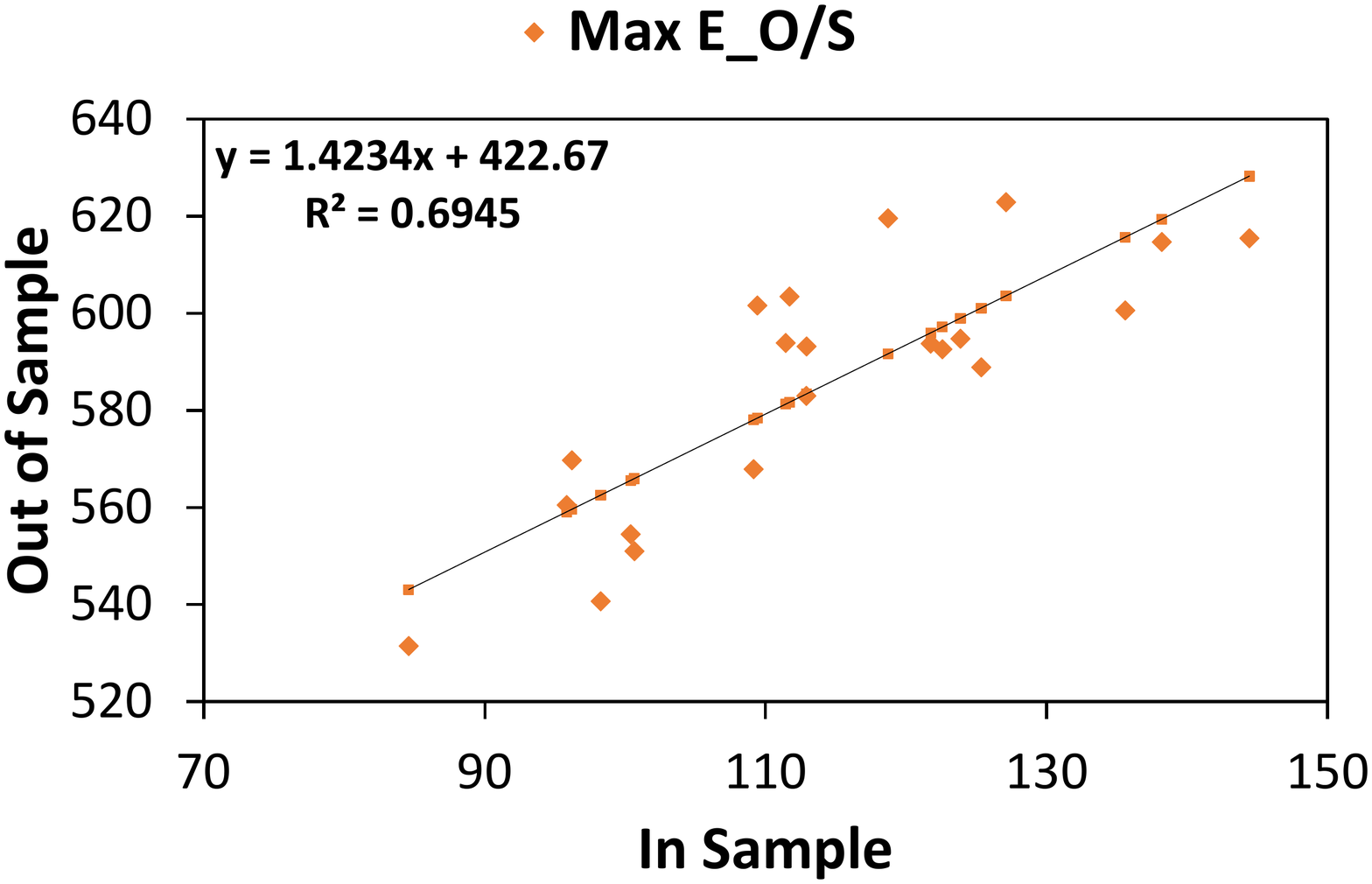}
			\caption{Robust solutions, expected max demand gap.}\label{fig2Bd}
		\end{subfigure}
		\vspace*{5mm}
		\begin{subfigure}{.45\textwidth}
			\includegraphics[width=1.1\linewidth]{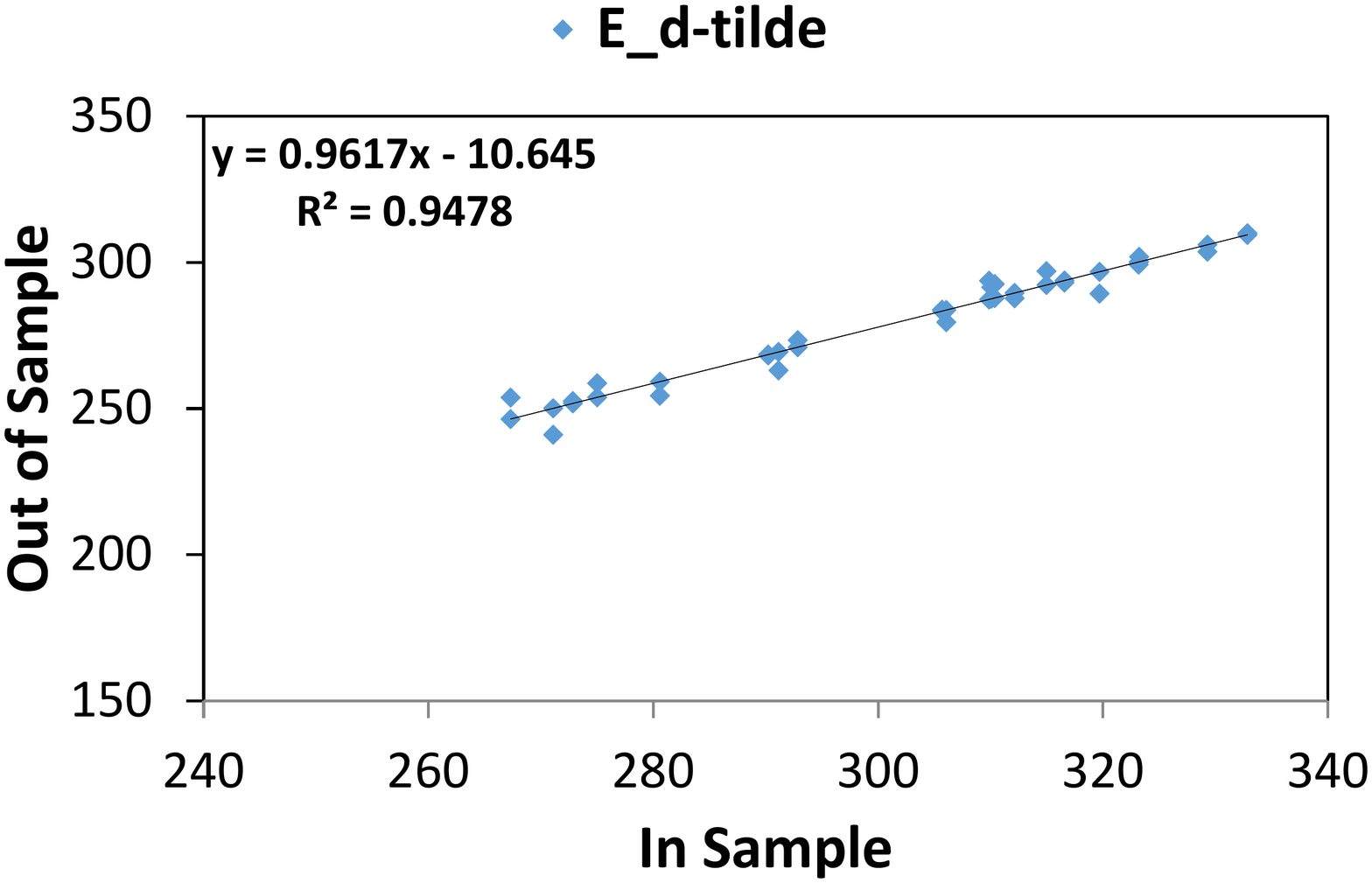}
			\caption{DRSO solutions, expected satisfied demand gap.}\label{fig2Be}
		\end{subfigure}
		\hspace{1cm}
		\caption{Results of out-of-sample prediction (commodity B).}\label{figB2}
	\end{center}
\end{figure}

Results indicate an overall similarity to the results reported in Section~\ref{sec:Exp}.

\section{Results for commodity type C.}
\label{AppendB}

We present additional results for commodity type C in a similar way to the presentation of results for commodity type A in the main text. \autoref{Tab5} and \autoref{Tab6} show key metrics for the 20 repetitions using the robust and the DRSO model, respectively.

\begin{table}[htbp]\footnotesize
	\centering
	\caption{Robust model results for commodity type C.}
	\begin{tabular}{rrcrcrrrr}
		\toprule 
		&\multicolumn{2}{c}{In Sample} 
		&\multicolumn{5}{c}{Out of Sample}&\\	
		\cmidrule (r) {2-3} \cmidrule (l) {4-8}
		Inst. &    Cap. Inv.&	O/S Demand&	E[O/S]& E[Max O/S]&	CVaR95&	CVaR75&	E[$\tilde{d}$]&	CapAdd \\
		\midrule
		
		1	&	 26,956.58 	 & 	 179.10 	 & 	 136.73 	 & 	 684.56 	 & 	 529.60 	 & 	 339.64 	 & 	 276.26 	 & 	 647.25 	\\
		2	&	 32,347.16 	 & 	 161.78 	 & 	 130.64 	 & 	 634.36 	 & 	 479.54 	 & 	 296.79 	 & 	 279.20 	 & 	 794.22 	\\
		3	&	 30,172.88 	 & 	 149.01 	 & 	 119.10 	 & 	 659.10 	 & 	 504.14 	 & 	 313.73 	 & 	 282.00 	 & 	 725.67 	\\
		4	&	 28,334.37 	 & 	 172.10 	 & 	 137.63 	 & 	 667.46 	 & 	 512.50 	 & 	 325.54 	 & 	 273.43 	 & 	 686.66 	\\
		5	&	 35,453.89 	 & 	 120.88 	 & 	 98.17 	 & 	 594.81 	 & 	 439.86 	 & 	 257.32 	 & 	 320.85 	 & 	 851.52 	\\
		6	&	 30,982.87 	 & 	 152.36 	 & 	 119.19 	 & 	 664.96 	 & 	 510.01 	 & 	 320.93 	 & 	 293.67 	 & 	 749.94 	\\
		7	&	 25,591.28 	 & 	 195.04 	 & 	 159.70 	 & 	 690.12 	 & 	 535.33 	 & 	 353.75 	 & 	 264.98 	 & 	 616.11 	\\
		8	&	 25,347.46 	 & 	 209.66 	 & 	 156.05 	 & 	 706.55 	 & 	 551.59 	 & 	 363.42 	 & 	 246.95 	 & 	 609.26 	\\
		9	&	 30,205.32 	 & 	 184.27 	 & 	 132.04 	 & 	 651.36 	 & 	 496.41 	 & 	 317.62 	 & 	 275.19 	 & 	 727.79 	\\
		10	&	 27,130.79 	 & 	 166.38 	 & 	 143.88 	 & 	 698.51 	 & 	 543.55 	 & 	 353.33 	 & 	 266.32 	 & 	 656.16 	\\
		11	&	 29,966.64 	 & 	 150.23 	 & 	 126.10 	 & 	 649.42 	 & 	 494.47 	 & 	 309.78 	 & 	 285.65 	 & 	 722.97 	\\
		12	&	 30,664.61 	 & 	 146.39 	 & 	 123.74 	 & 	 648.96 	 & 	 494.00 	 & 	 305.39 	 & 	 291.40 	 & 	 742.28 	\\
		13	&	 30,595.65 	 & 	 154.07 	 & 	 129.73 	 & 	 641.56 	 & 	 486.60 	 & 	 302.37 	 & 	 296.50 	 & 	 739.28 	\\
		14	&	 32,529.85 	 & 	 149.01 	 & 	 109.24 	 & 	 628.82 	 & 	 475.08 	 & 	 292.77 	 & 	 310.60 	 & 	 781.49 	\\
		15	&	 28,710.56 	 & 	 150.83 	 & 	 129.56 	 & 	 674.99 	 & 	 520.03 	 & 	 329.21 	 & 	 279.38 	 & 	 692.88 	\\
		16	&	 28,970.35 	 & 	 167.93 	 & 	 133.26 	 & 	 655.04 	 & 	 500.08 	 & 	 316.23 	 & 	 308.97 	 & 	 703.85 	\\
		17	&	 29,124.96 	 & 	 169.75 	 & 	 134.52 	 & 	 658.81 	 & 	 504.96 	 & 	 323.99 	 & 	 283.71 	 & 	 698.58 	\\
		18	&	 24,216.81 	 & 	 191.10 	 & 	 163.21 	 & 	 717.45 	 & 	 562.50 	 & 	 371.64 	 & 	 249.18 	 & 	 590.91 	\\
		19	&	 31,721.36 	 & 	 171.99 	 & 	 119.38 	 & 	 627.87 	 & 	 472.94 	 & 	 292.73 	 & 	 281.83 	 & 	 759.90 	\\
		20	&	 25,506.27 	 & 	 189.19 	 & 	 145.45 	 & 	 706.33 	 & 	 551.37 	 & 	 360.61 	 & 	 267.52 	 & 	 610.99 	\\
		\bottomrule
	\end{tabular}
	\label{Tab5}
\end{table}

\begin{table}[htbp]\footnotesize
	\centering
	\caption{DRSO model results for commodity type C.}
	\begin{tabular}{rrrcrcrrrr}
		\toprule 
		&&\multicolumn{2}{l}{In Sample} 
		&\multicolumn{5}{c}{Out of Sample}&\\ 
		\cmidrule (r) {2-4} \cmidrule (l) {5-9}
		Inst. &    Cap. Inv.& Nature& $\tilde{d}$ &E[O/S]& E[ Max O/S]& CVaR95& CVaR75& E[$\tilde{d}$]& CapAdd \\
		\midrule
		1	&	 12,363.99 	 & 	 268.29 	 & 	 152.98 	 & 	258.35	 & 	 842.53 	 & 	 687.58 	 & 	 496.72 	 & 	146.54	 & 	 294.04 	\\
		2	&	 5,571.67 	 & 	 345.44 	 & 	 59.95 	 & 	332.11	 & 	 929.18 	 & 	 774.23 	 & 	 583.37 	 & 	65.39	 & 	 136.01 	\\
		3	&	 13,259.41 	 & 	 259.92 	 & 	 165.58 	 & 	247.00	 & 	 829.93 	 & 	 674.97 	 & 	 484.12 	 & 	157.41	 & 	 321.55 	\\
		4	&	 7,897.34 	 & 	 309.18 	 & 	 105.22 	 & 	301.23	 & 	 888.37 	 & 	 733.42 	 & 	 542.56 	 & 	101.94	 & 	 186.84 	\\
		5	&	 12,396.05 	 & 	 285.00 	 & 	 146.38 	 & 	258.36	 & 	 839.87 	 & 	 684.91 	 & 	 494.06 	 & 	142.81	 & 	 299.13 	\\
		6	&	 11,576.99 	 & 	 263.34 	 & 	 151.77 	 & 	261.03	 & 	 843.09 	 & 	 688.14 	 & 	 497.28 	 & 	147.41	 & 	 283.31 	\\
		7	&	 11,152.00 	 & 	 270.66 	 & 	 148.23 	 & 	267.12	 & 	 847.28 	 & 	 692.33 	 & 	 501.47 	 & 	136.14	 & 	 266.95 	\\
		8	&	 14,153.71 	 & 	 249.28 	 & 	 174.53 	 & 	237.78	 & 	 820.98 	 & 	 666.03 	 & 	 475.17 	 & 	162.31	 & 	 344.49 	\\
		9	&	 9,627.59 	 & 	 282.76 	 & 	 134.52 	 & 	276.63	 & 	 858.54 	 & 	 703.58 	 & 	 512.73 	 & 	129.57	 & 	 234.24 	\\
		10	&	 10,393.15 	 & 	 287.02 	 & 	 126.55 	 & 	272.23	 & 	 861.08 	 & 	 706.13 	 & 	 515.27 	 & 	126.03	 & 	 249.93 	\\
		11	&	 10,305.45 	 & 	 284.76 	 & 	 131.37 	 & 	275.49	 & 	 861.32 	 & 	 706.36 	 & 	 515.51 	 & 	130.22	 & 	 252.57 	\\
		12	&	 8,650.82 	 & 	 311.90 	 & 	 105.36 	 & 	295.17	 & 	 882.67 	 & 	 727.72 	 & 	 536.86 	 & 	106.36	 & 	 214.13 	\\
		13	&	 14,760.43 	 & 	 242.42 	 & 	 189.52 	 & 	229.55	 & 	 805.94 	 & 	 650.98 	 & 	 460.13 	 & 	178.94	 & 	 355.27 	\\
		14	&	 10,207.56 	 & 	 295.88 	 & 	 132.09 	 & 	276.41	 & 	 860.33 	 & 	 705.37 	 & 	 514.52 	 & 	125.09	 & 	 248.92 	\\
		15	&	 15,725.45 	 & 	 240.47 	 & 	 176.62 	 & 	231.30	 & 	 813.68 	 & 	 658.73 	 & 	 467.87 	 & 	173.52	 & 	 381.58 	\\
		16	&	 12,809.80 	 & 	 261.85 	 & 	 160.46 	 & 	253.39	 & 	 835.05 	 & 	 680.10 	 & 	 489.24 	 & 	154.27	 & 	 312.61 	\\
		17	&	 14,626.78 	 & 	 243.12 	 & 	 186.55 	 & 	230.40	 & 	 808.96 	 & 	 654.00 	 & 	 463.15 	 & 	171.96	 & 	 353.18 	\\
		18	&	 13,876.73 	 & 	 245.18 	 & 	 173.35 	 & 	241.13	 & 	 822.12 	 & 	 667.17 	 & 	 476.31 	 & 	159.20	 & 	 336.01 	\\
		19	&	 13,628.69 	 & 	 262.00 	 & 	 175.23 	 & 	236.30	 & 	 815.98 	 & 	 661.03 	 & 	 470.17 	 & 	165.27	 & 	 327.33 	\\
		20	&	 11,700.59 	 & 	 261.69 	 & 	 158.25 	 & 	254.83	 & 	 836.88 	 & 	 681.93 	 & 	 491.07 	 & 	150.23	 & 	 272.91 	\\
		
		\bottomrule
	\end{tabular}
	\label{Tab6}
\end{table}

\autoref{fig6}, \autoref{fig9} and \autoref{fig7} 
correspond to \autoref{fig1}, \autoref{fig5} and \autoref{fig2} using commodity type B instead of A.

\begin{figure}[htbp]
	\begin{center}
		\begin{subfigure}{.45\textwidth}
			\includegraphics[width=1.1\linewidth]{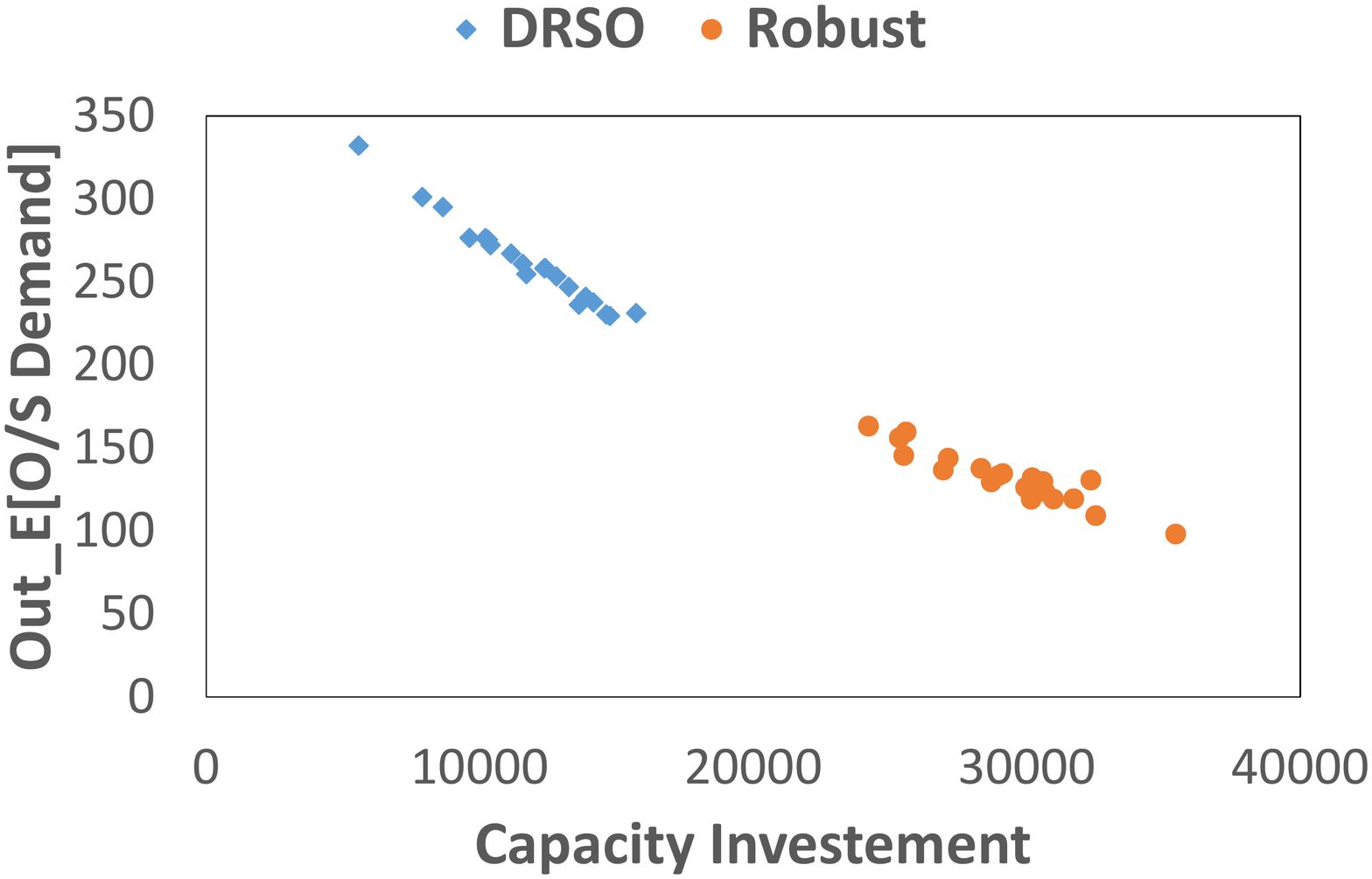}
			\caption{Expected unsatisfied demand.}\label{fig6a}
		\end{subfigure}
		\hspace{1cm}
		\begin{subfigure}{.45\textwidth}
			\includegraphics[width=1.1\linewidth]{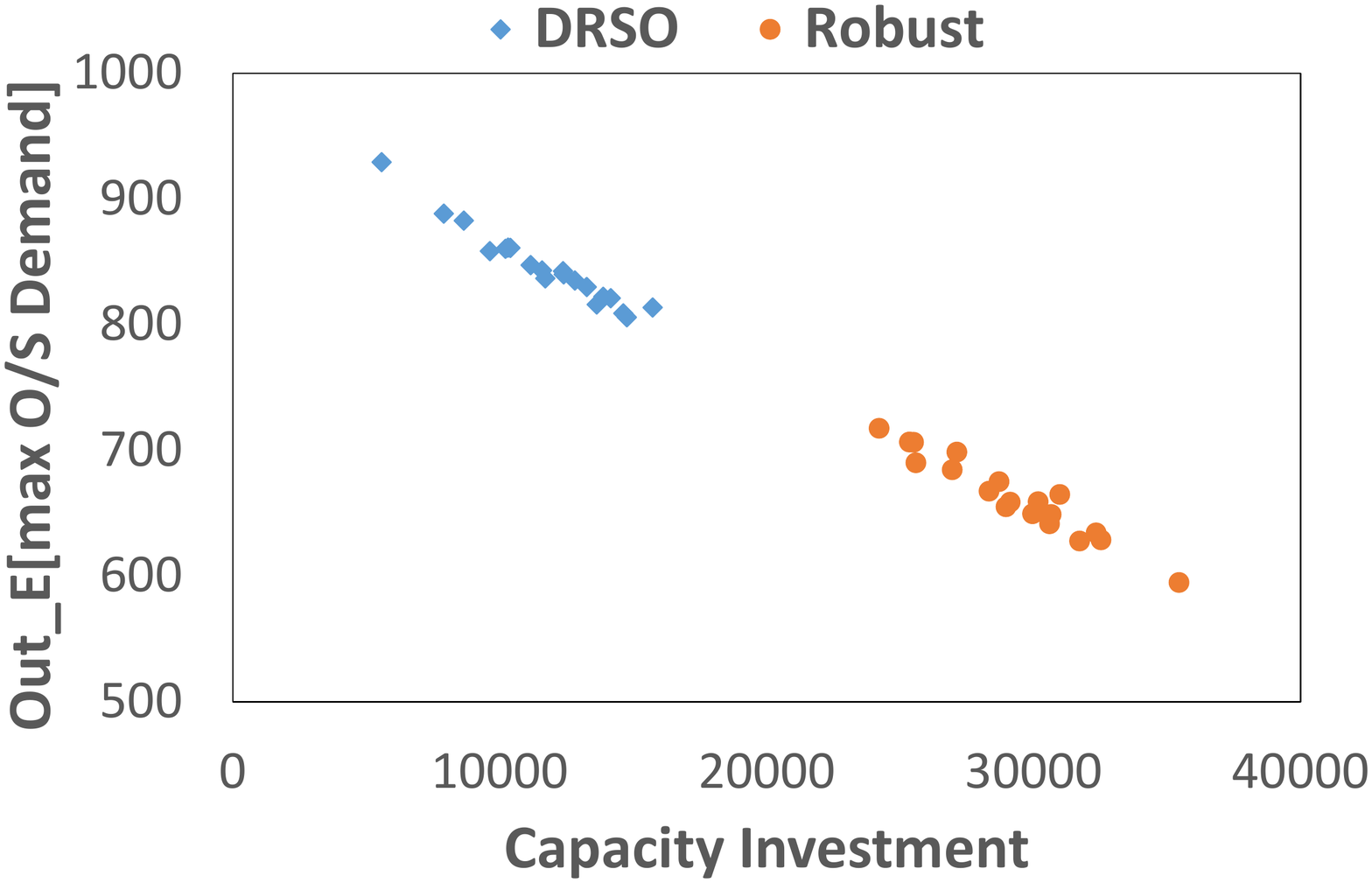}
			\caption{Expected max unsatisfied demand.}\label{fig6b}
		\end{subfigure}
		\vspace*{5mm}
		\begin{subfigure}{.45\textwidth}
			\includegraphics[width=1.1\linewidth]{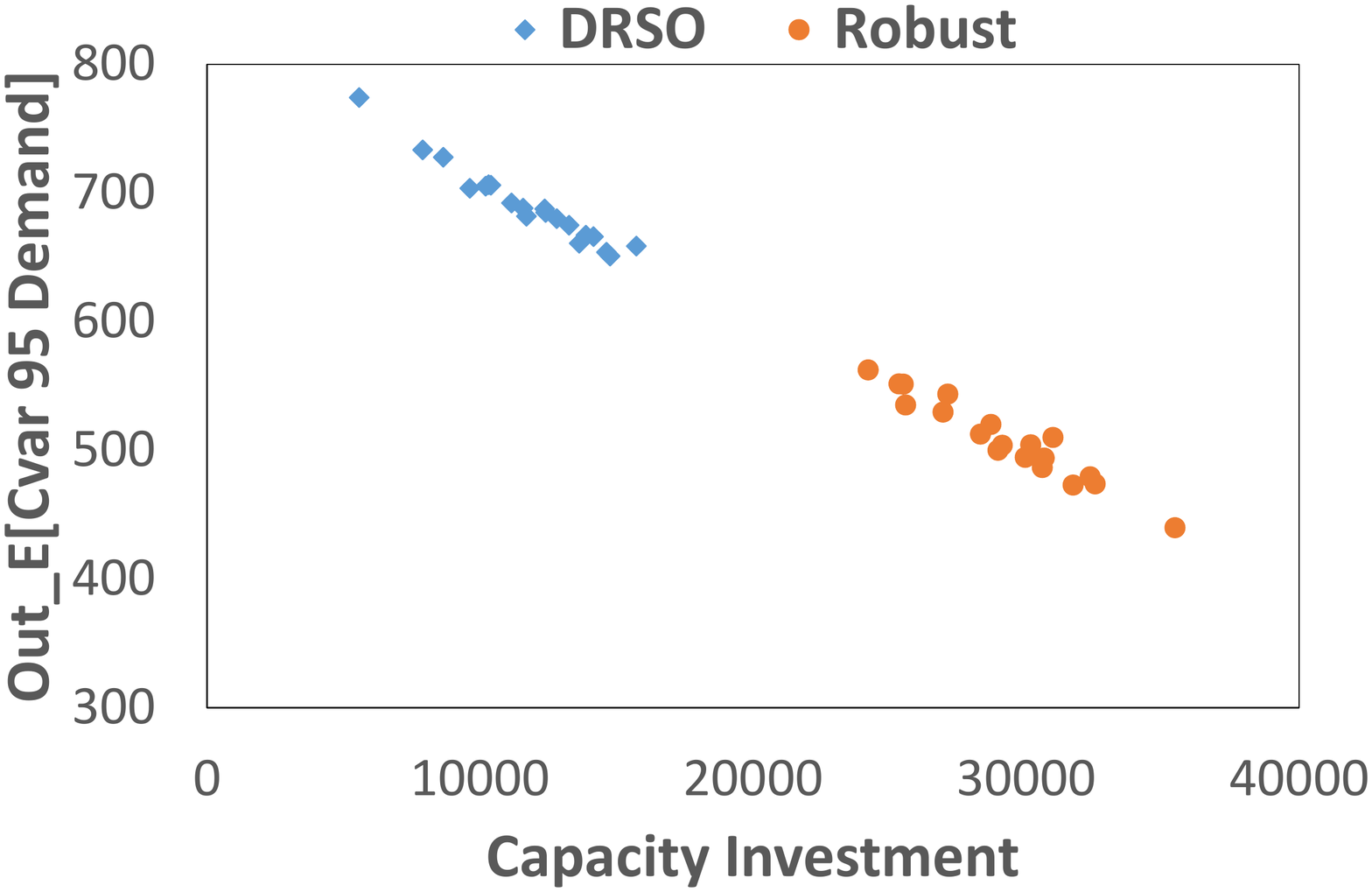}
			\caption{CVaR95 unsatisfied demand.}\label{fig6c}
		\end{subfigure}
		\hspace{1cm}
		\begin{subfigure}{.45\textwidth}
			\includegraphics[width=1.1\linewidth]{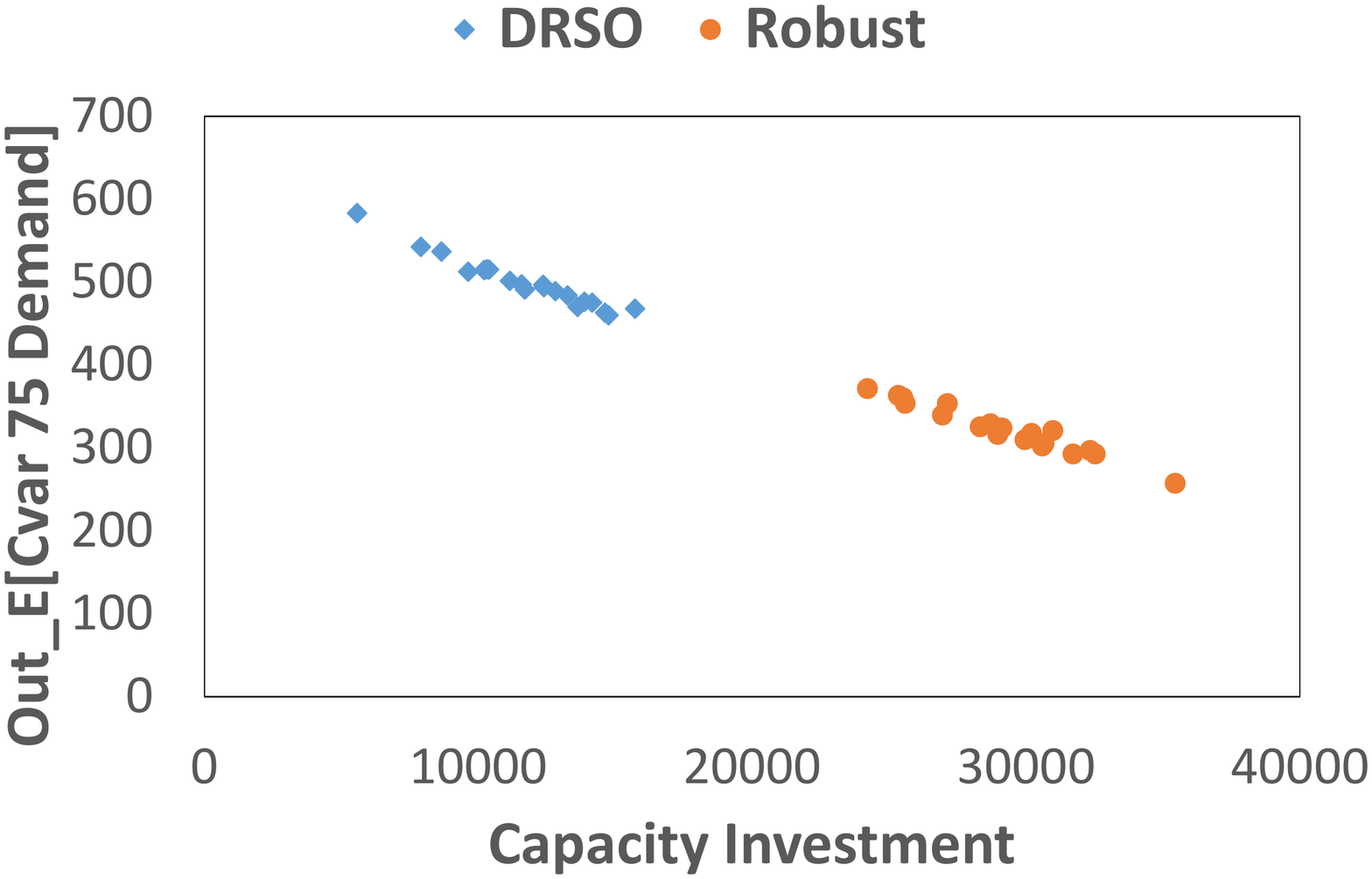}
			\caption{CVaR75 unsatisfied demand.}\label{fig6d}
		\end{subfigure}
		\caption{Expected unsatisfied demand mean and risk measures (commodity C).}\label{fig6}
	\end{center}
\end{figure}	

\begin{figure}[htbp]
	\begin{center}
		\begin{subfigure}{.45\textwidth}
			\includegraphics[width=1.1\linewidth]{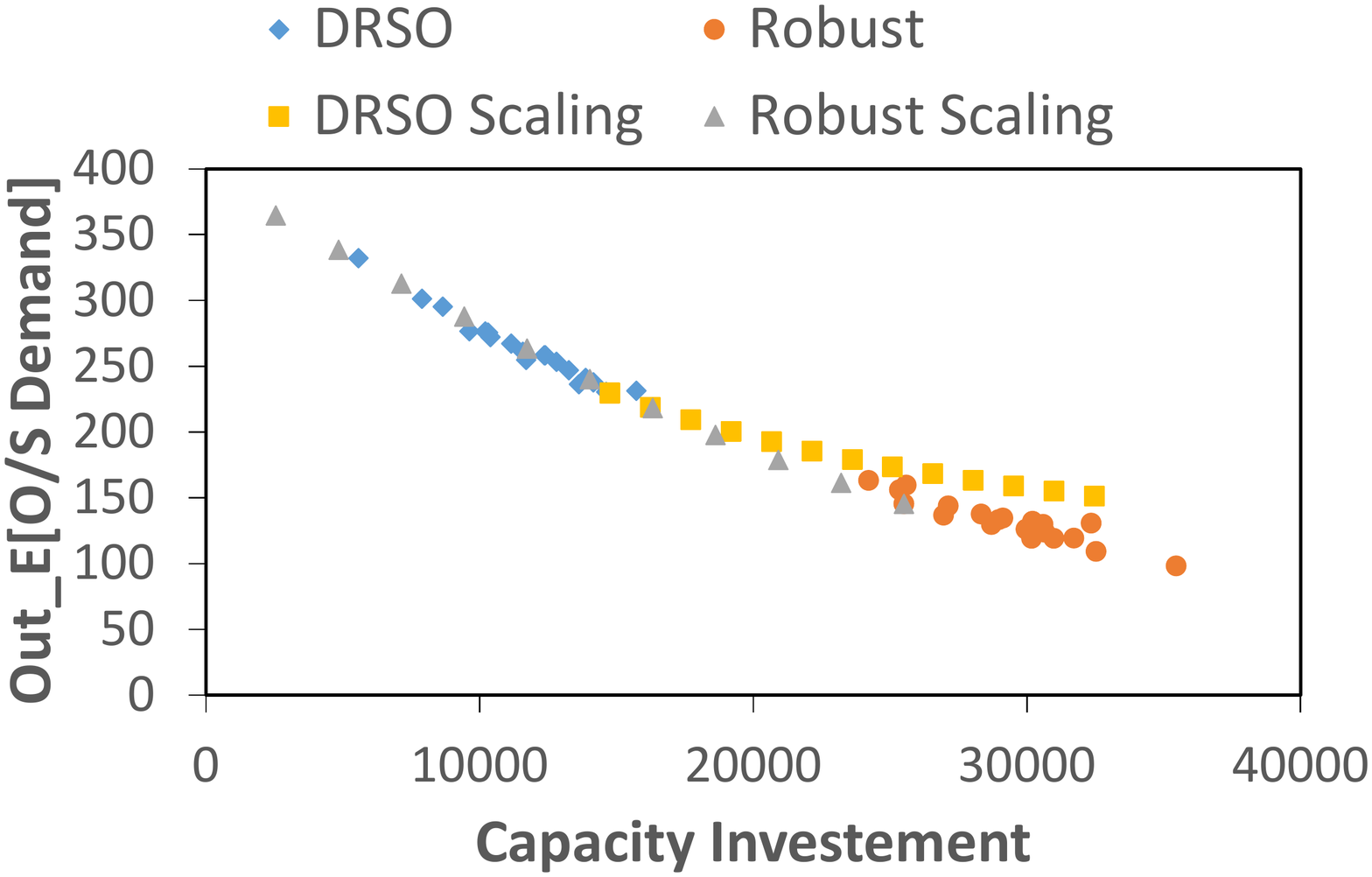}
			\caption{Expected unsatisfied demand.}\label{fig9a}
		\end{subfigure}
		\hspace{1cm}
		\begin{subfigure}{.45\textwidth}
			\includegraphics[width=1.1\linewidth]{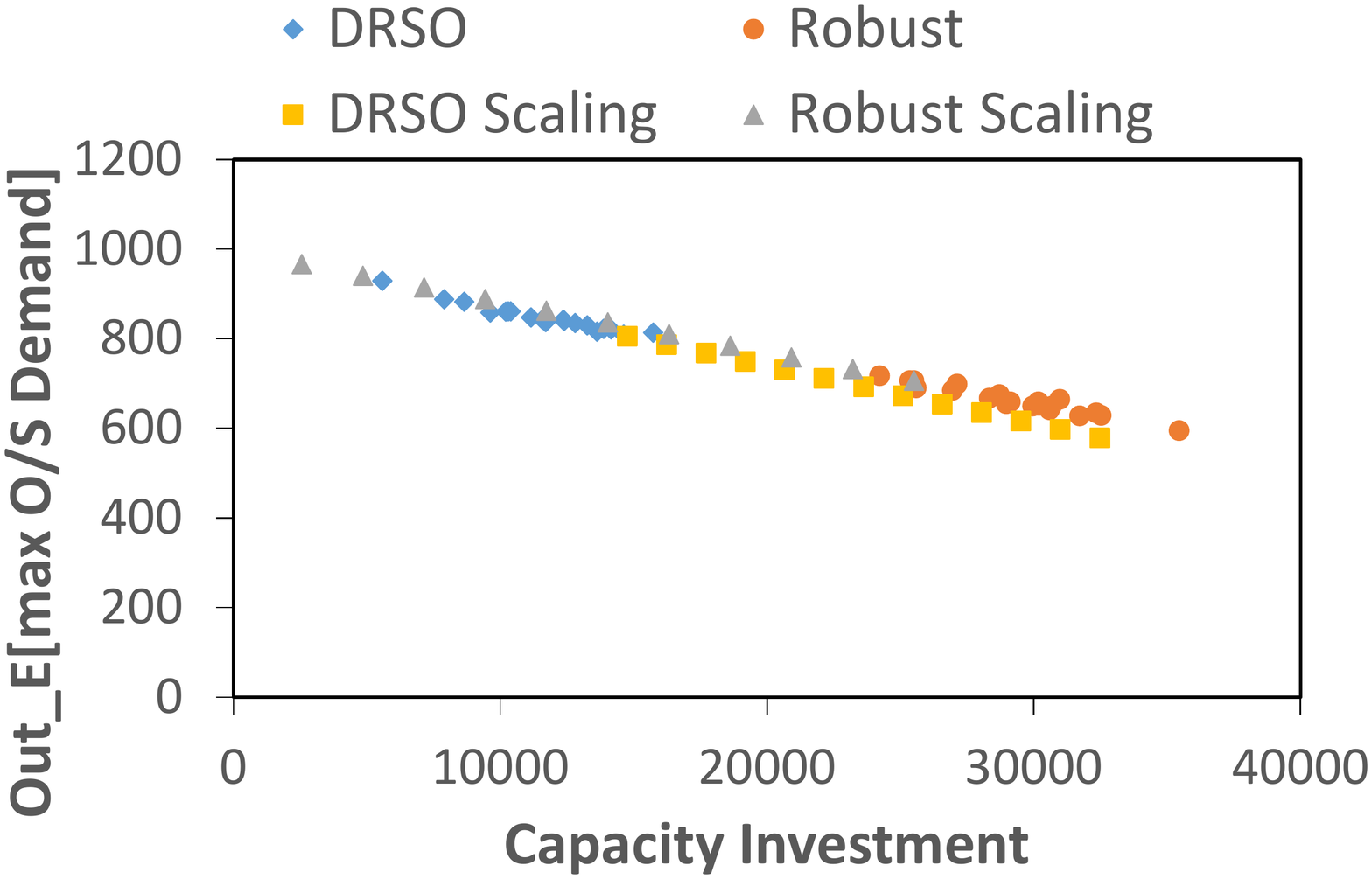}
			\caption{Expected maximum unsatisfied demand.}\label{fig9b}
		\end{subfigure}
		\vspace*{5mm}
		\begin{subfigure}{.45\textwidth}
			\includegraphics[width=1.1\linewidth]{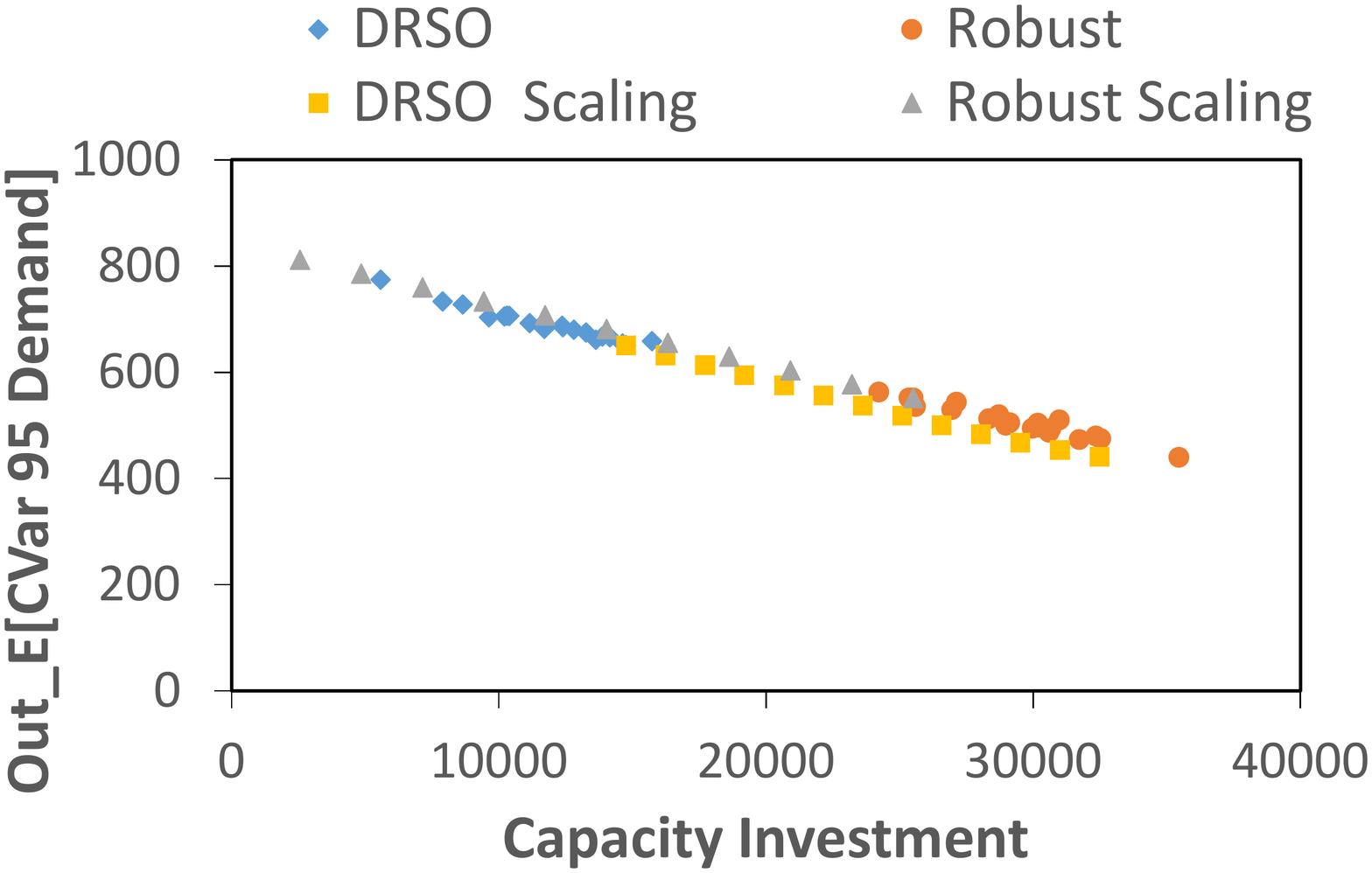}
			\caption{CVaR95 unsatisfied demand}\label{fig9c}
		\end{subfigure}
		\hspace{1cm}
		\begin{subfigure}{.45\textwidth}
			\includegraphics[width=1.1\linewidth]{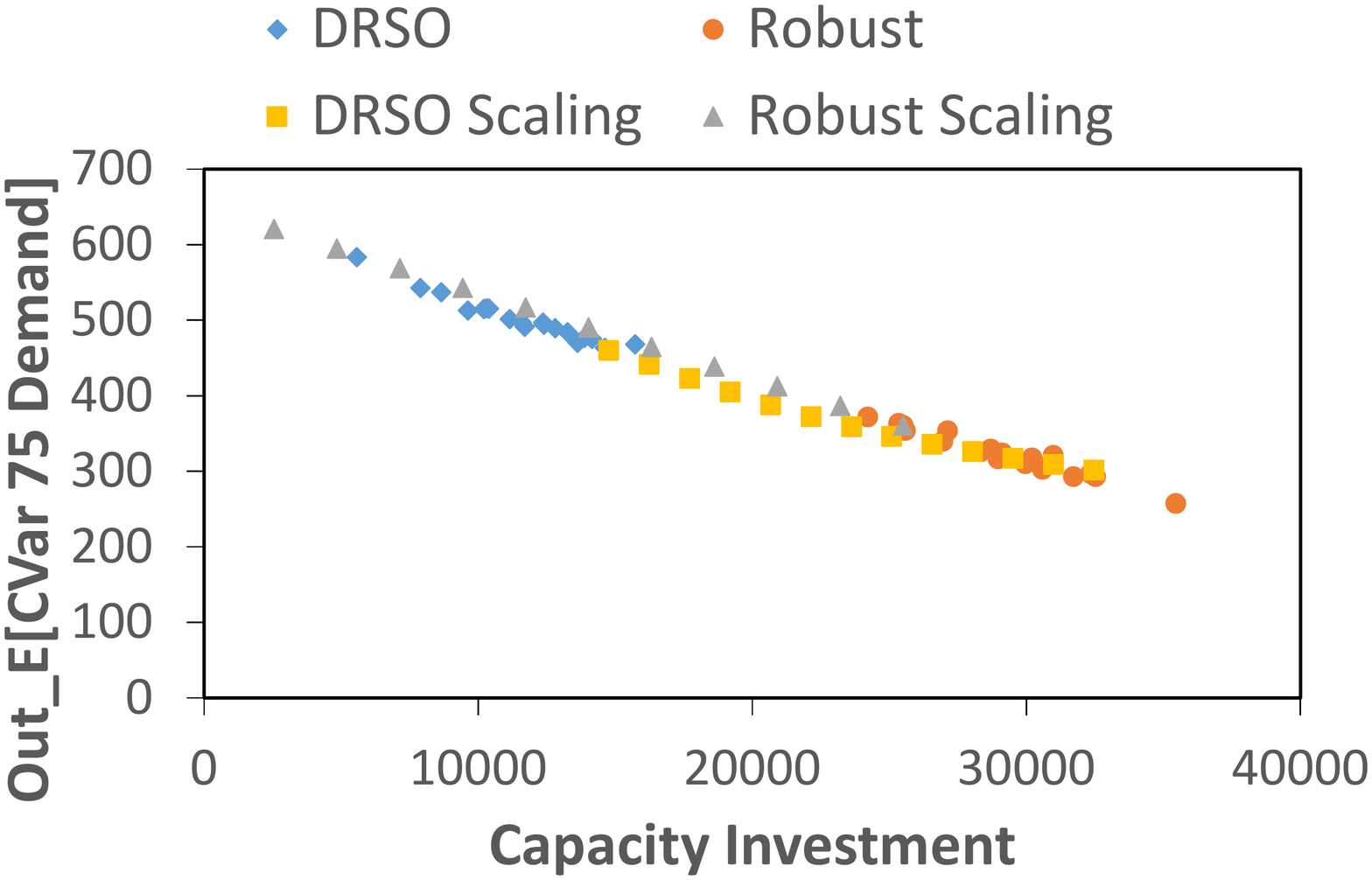}
			\caption{CVaR75 Unsatisfied Demand.}\label{fig9d}
		\end{subfigure}
		\caption{Performance metric scaling (commodity C).}\label{fig9}
	\end{center}
\end{figure}

\begin{figure}[htbp]
	\begin{center}
		\begin{subfigure}{.45\textwidth}
			\includegraphics[width=1.1\linewidth]{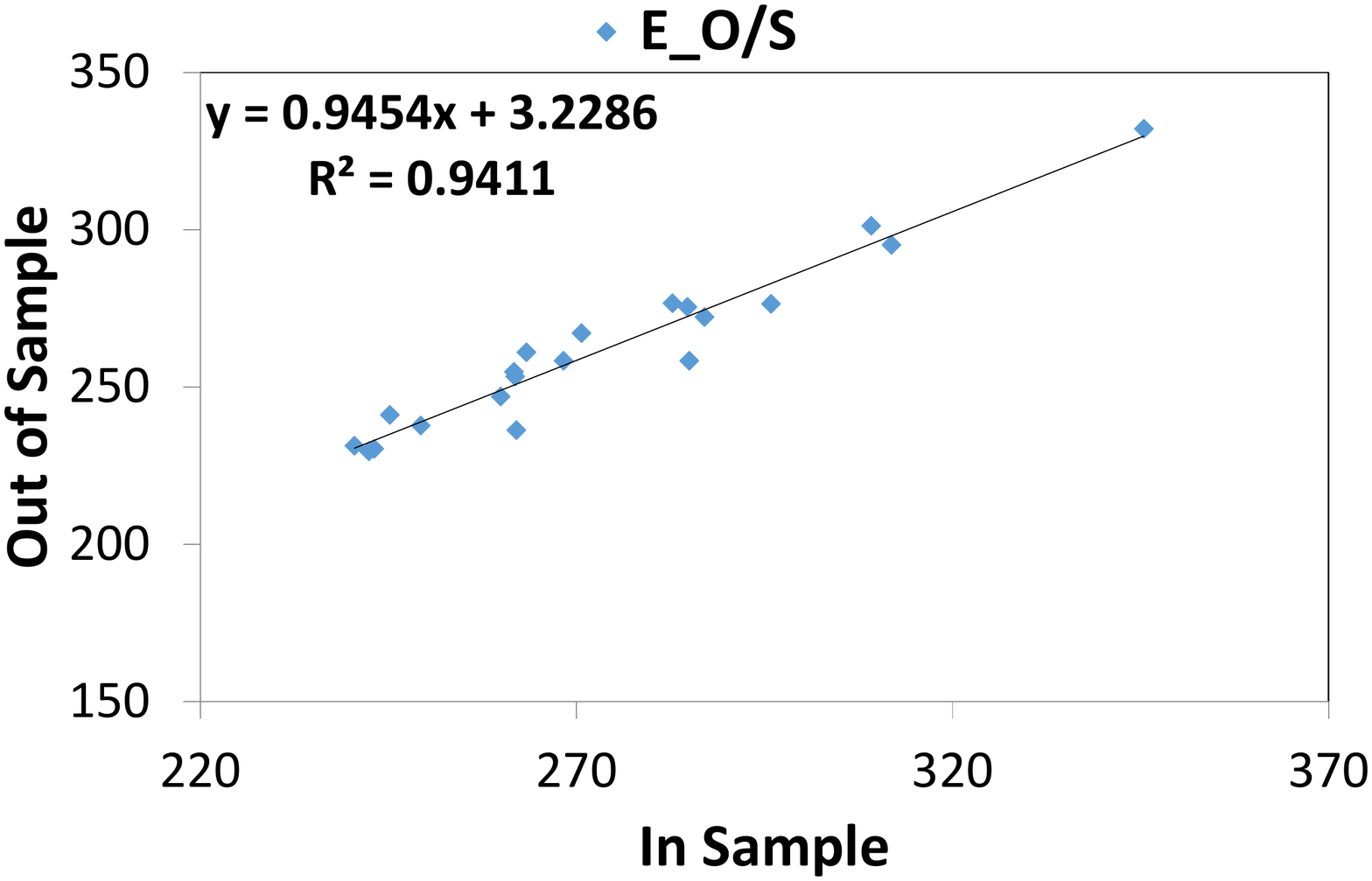}
			\caption{DRSO solutions, expected demand gap.}\label{fig7a}
		\end{subfigure}
		\hspace{1cm}
		\begin{subfigure}{.45\textwidth}
			\includegraphics[width=1.1\linewidth]{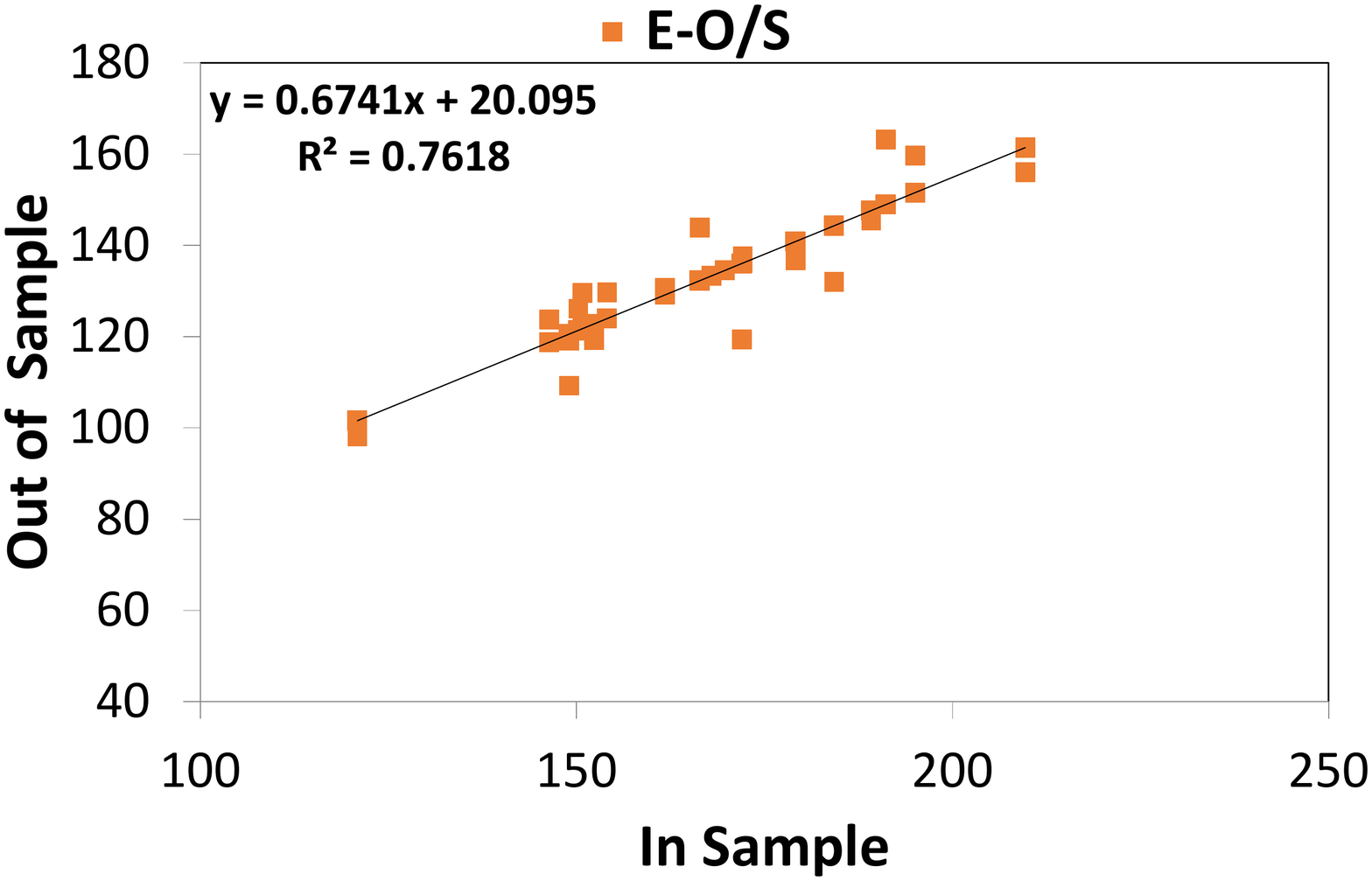}
			\caption{Robust solutions, expected demand gap.}\label{fig7b}
		\end{subfigure}
		\vspace*{5mm}
		\begin{subfigure}{.45\textwidth}
			\includegraphics[width=1.1\linewidth]{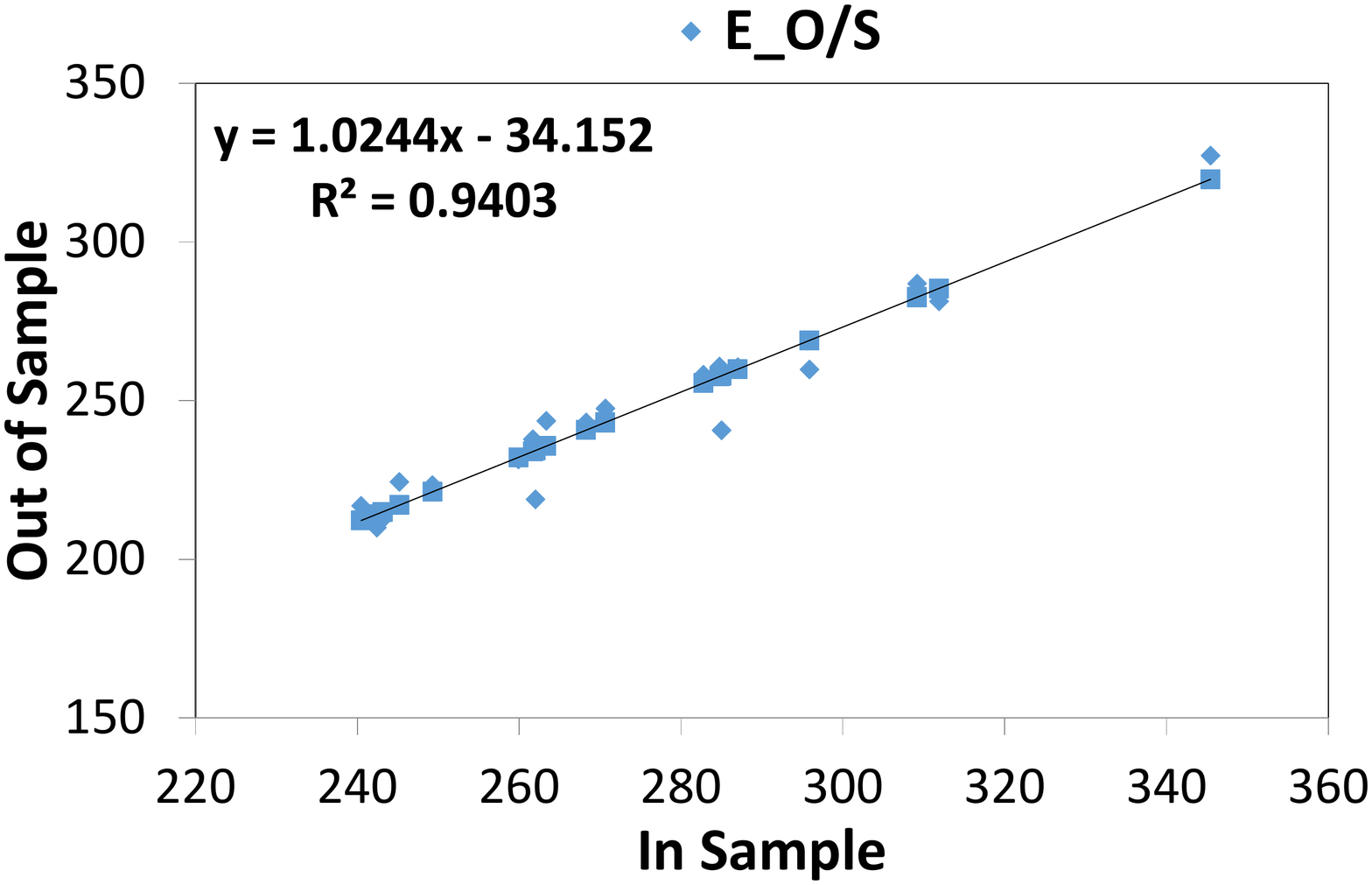}
			\caption{DRSO solutions, expected max demand gap.}\label{fig7c}
		\end{subfigure}
		\hspace{1cm}
		\begin{subfigure}{.45\textwidth}
			\includegraphics[width=1.1\linewidth]{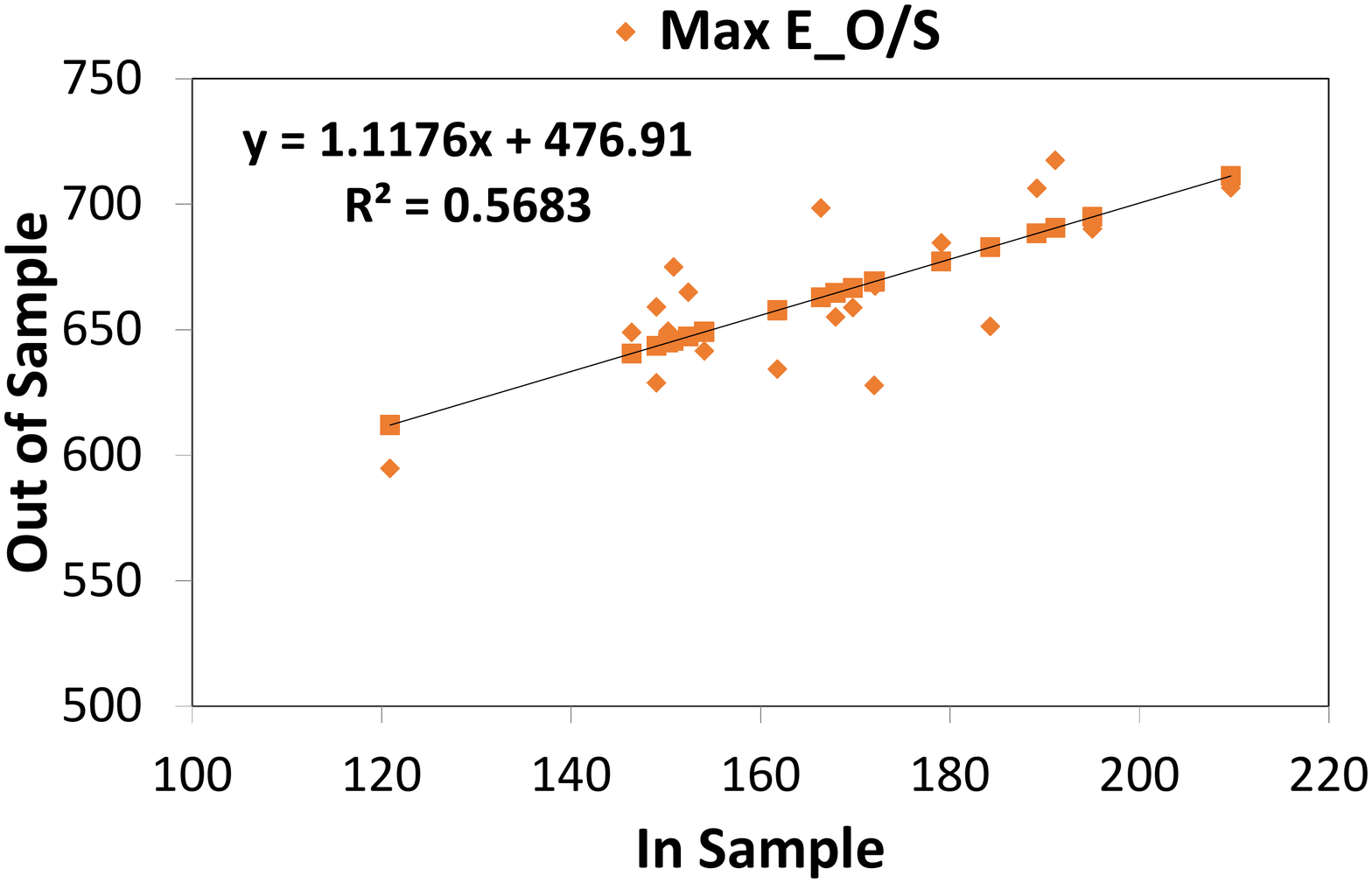}
			\caption{Robust solutions, expected max demand gap.}\label{fig7d}
		\end{subfigure}
		\vspace*{5mm}
		\begin{subfigure}{.45\textwidth}
			\includegraphics[width=1.1\linewidth]{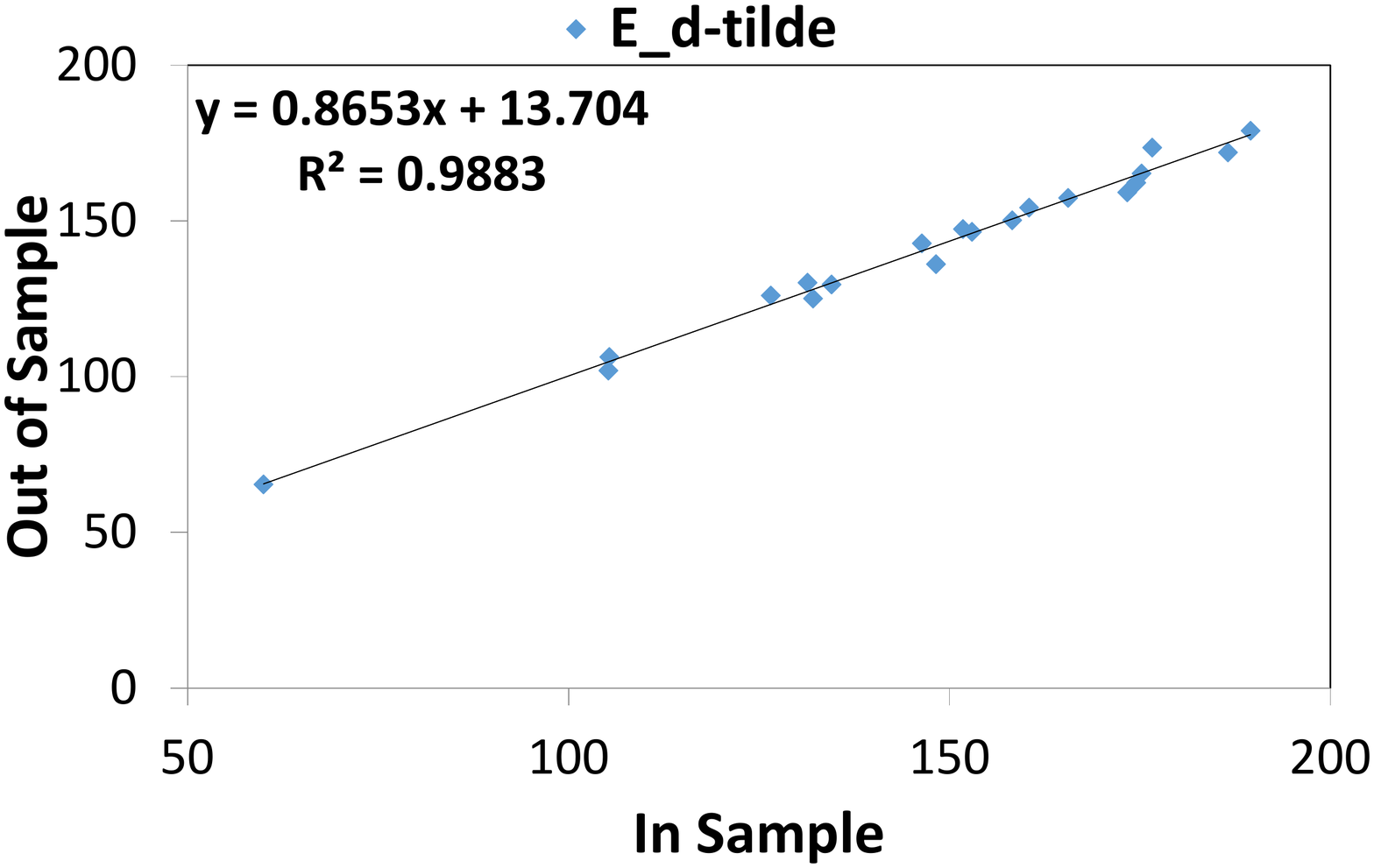}
			\caption{DRSO solutions, expected satisfied demand gap.}\label{fig7e}
		\end{subfigure}
		\hspace{1cm}
		\caption{Results of out-of-sample prediction (commodity C).}\label{fig7}
	\end{center}
\end{figure}

\end{appendices}

\end{document}